\documentclass[12pt,letterpaper,english,twoside]{article}
\usepackage[T1]{fontenc}
\usepackage[latin9]{inputenc}
\usepackage{babel}
\usepackage{refstyle}
\usepackage{amsthm}
\usepackage{amsmath}
\usepackage{amssymb}
\usepackage{esint}
\usepackage[unicode=true,pdfusetitle,
 bookmarks=true,bookmarksnumbered=true,bookmarksopen=true,bookmarksopenlevel=1,
 breaklinks=false,pdfborder={0 0 0},backref=false,colorlinks=false]
 {hyperref}

\makeatletter

\AtBeginDocument{\providecommand\eqref[1]{\ref{eq:#1}}}
\AtBeginDocument{\providecommand\thmref[1]{\ref{thm:#1}}}
\AtBeginDocument{\providecommand\propref[1]{\ref{prop:#1}}}
\AtBeginDocument{\providecommand\lemref[1]{\ref{lem:#1}}}
\AtBeginDocument{\providecommand\secref[1]{\ref{sec:#1}}}
\AtBeginDocument{\providecommand\corref[1]{\ref{cor:#1}}}
\AtBeginDocument{\providecommand\defnref[1]{\ref{defn:#1}}}
\pdfpageheight\paperheight
\pdfpagewidth\paperwidth

\RS@ifundefined{subref}
  {\def\RSsubtxt{section~}\newref{sub}{name = \RSsubtxt}}
  {}
\RS@ifundefined{thmref}
  {\def\RSthmtxt{theorem~}\newref{thm}{name = \RSthmtxt}}
  {}
\RS@ifundefined{lemref}
  {\def\RSlemtxt{lemma~}\newref{lem}{name = \RSlemtxt}}
  {}

\numberwithin{equation}{section}
\theoremstyle{plain}
\newtheorem{thm}{\protect\theoremname}[section]
  \theoremstyle{plain}
  \newtheorem{prop}[thm]{\protect\propositionname}
  \theoremstyle{remark}
  \newtheorem*{rem*}{\protect\remarkname}
  \theoremstyle{plain}
  \newtheorem{lem}[thm]{\protect\lemmaname}
  \theoremstyle{plain}
  \newtheorem{cor}[thm]{\protect\corollaryname}
  \theoremstyle{definition}
  \newtheorem{defn}[thm]{\protect\definitionname}

\usepackage[hmarginratio=1:1]{geometry}
\usepackage{times}
\usepackage{enumerate}
\def\titlerunning#1{\gdef\titrun{#1}}
\makeatletter
\def\author#1{\gdef\autrun{\def\and{\unskip, }#1}\gdef\@author{#1}}
\def\address#1{{\def\and{\\\hspace*{18pt}}\renewcommand{\thefootnote}{}%
\footnote {#1}}%
\markboth{\autrun}{\titrun}}
\makeatother
\def\email#1{e-mail: #1}
\def\subjclass#1{{\renewcommand{\thefootnote}{}%
\footnote{\emph{Mathematics Subject Classification (2010):} #1}}}
\def\keywords#1{\par\medskip
\noindent\textbf{Keywords.} #1}

\DeclareMathOperator{\mes}{mes}
\DeclareMathOperator{\com}{compl}
\DeclareMathOperator{\car}{Car}
\DeclareMathOperator{\dist}{dist}
\DeclareMathOperator{\spec}{spec}
\let \corref=\relax
\let\propref=\relax
\let\defnref=\relax
\newref{eq}{refcmd = {(\ref{#1})} }
\newref{thm}{refcmd = {Theorem \ref{#1}}}
\newref{lem}{refcmd = {Lemma \ref{#1}}}
\newref{prop}{refcmd = {Proposition \ref{#1}}}
\newref{cor}{refcmd = {Corollary \ref{#1}}}
\newref{defn}{refcmd = {Definition \ref{#1}}}
\newref{sec}{refcmd = {Section \ref{#1}}}
 
\renewcommand{\Im}{\mathrm{Im}}

\makeatother

  \providecommand{\corollaryname}{Corollary}
  \providecommand{\definitionname}{Definition}
  \providecommand{\lemmaname}{Lemma}
  \providecommand{\propositionname}{Proposition}
  \providecommand{\remarkname}{Remark}
\providecommand{\theoremname}{Theorem}

\begin{document}
\global\long\def\mb#1{\mathbb{#1}}

\global\long\def\mc#1{\mathcal{#1}}

\global\long\def\mf#1{\frak{#1}}

\global\long\def\t#1{\tilde{#1}}

\thinmuskip=1mu
\medmuskip=2mu
\thickmuskip=4mu
\titlerunning{On Optimal Separation of Eigenvalues for a Quasiperiodic Jacobi Matrix}
\title{On Optimal Separation of Eigenvalues for a Quasiperiodic Jacobi Matrix}
\author{Ilia Binder \and  Mircea Voda}
\date{}
\maketitle
\address{I. Binder: Dept. of Mathematics, University of Toronto, Toronto, ON, M5S 2E4, Canada; \email{ilia@math.utoronto.ca} \and M. Voda: Dept. of Mathematics, University of Toronto, Toronto, ON, M5S 2E4, Canada; \email{mvoda@math.utoronto.ca}}
\begin{abstract} 

We consider quasiperiodic Jacobi matrices of size $N$ with analytic
coefficients. We show that, in the positive Lyapunov exponent regime,
after removing some small sets of energies and frequencies, any eigenvalue
is separated from the rest of the spectrum by $N^{-1}\left(\log N\right)^{-p}$,
with $p>15$. 

  \end{abstract}

\section{Introduction}

It is known that one-dimensional quasiperiodic Schr\"odinger operators
in the regime of positive Lyapunov exponent exhibit exponential localization
of eigenfunctions (see for example \cite{MR2100420}). Can one develop
an inverse spectral theory in such a regime? This is one of two major
questions behind our work. The most studied case is the discrete single
frequency case. Since the inverse spectral theory for the periodic
case is well-understood, it seems very natural to try to understand
how the regime of positive Lyapunov exponent plays out with the periodic
approximation of the frequency via the standard convergent of its
continued fraction. Obviously, the optimal estimate for the separation
of the eigenvalues of the quasiperiodic operator on a fi{}nite interval
is crucial for this kind of approach. This is the second major question
behind this work. It is easy to fi{}gure out that the desired separation
for the operator on the interval $\left[0,N-1\right]$, with appropriate
$N$, is $\gtrsim N^{-1}\left(\log N\right)^{-p}$ with $p<1$. Is
this the correct estimate? A common sense argument suggests that outside
of a small exceptional set of eigenvalues the estimate should be $\gtrsim o\left(N^{-1}\right)$.
What is known about this problem? Goldstein and Schlag \cite{MR2753606}
proved the estimate $\gtrsim\exp\left(-\left(\log N\right)^{A}\right)$,
with $A\gg1$, which is far from optimal. In this paper we improve
the separation to $N^{-1}\left(\log N\right)^{-p}$, with $p>15$.
Moreover, we prove it for quasiperiodic Jacobi matrices. Our interest
in the more general case is motivated by the fact that quasiperiodic
Jacobi operators are necessary for the solution of the inverse spectral
problem for discrete quasiperiodic operators of second order. We note
that this setting is also needed for the study of the extended Harper's
model, which corresponds to $a\left(x\right)=2\cos(2\pi x)$, $b(x)=\lambda_{1}e^{2\pi i(x-\omega/2)}+\lambda_{2}+\lambda_{3}e^{-2\pi i\left(x-\omega/2\right)}$
(see \cite{MR2121278,NOMR1}). At the same time we want to stress
that the main result of this paper improves on the known result for
the Schr\"odinger case and makes it much closer to the optimal one.

We consider the quasiperiodic Jacobi operator $H\left(x,\omega\right)$
defined on $l^{2}\left(\mathbb{Z}\right)$ by
\[
\left[H\left(x,\omega\right)\phi\right]\left(k\right)=-b\left(x+\left(k+1\right)\omega\right)\phi\left(k+1\right)-\overline{b\left(x+k\omega\right)}\phi\left(k-1\right)+a\left(x+k\omega\right)\phi\left(k\right),
\]
where $a:\mb T\rightarrow\mb R$, $b:\mb T\rightarrow\mb C$ ($\mb T:=\mb R/\mb Z$)
are real analytic functions, $b$ is not identically zero, and $\omega\in\mb T_{c,\alpha}$
for some fixed $c\ll1$, $\alpha>1$, where 
\[
\mb T_{c,\alpha}:=\left\{ \omega\in\left(0,1\right):\,\left\Vert n\omega\right\Vert \ge\frac{c}{n\left(\log n\right)^{\alpha}}\right\} .
\]
The special case of the Schr\"odinger operator ($b=1$) has been
studied extensively (see \cite{MR883643,MR1102675}). 

It is known that the Diophantine condition imposed on $\omega$ is
generic, in the sense that $\mes\left(\cup_{c>0}\mb T_{c,\alpha}\right)=1$.
This Diophantine condition, first used by Goldstein and Schlag \cite{MR1847592},
has the advantage of allowing one to prove stronger large deviations
estimates (in the positive Lyapunov exponent case) than for general
irrational frequencies. The use of large deviations estimates in the
study of quasiperiodic Schr\"odinger operators was pioneered by Bourgain
and Goldstein \cite{MR1815703}. Initially these estimates were established
for transfer matrices. More recently Goldstein and Schlag \cite{MR2438997}
proved a large deviations estimate for the entries of the transfer
matrices (or equivalently for the determinants of the finite scale
restrictions of the operator). This estimate is essential for our
work, as it was for the developments in \cite{MR2438997} and \cite{MR2753606}.
The technical details of extending the large deviations estimate for
the entries to the Jacobi setting were dealt with in \cite{2012arXiv1202.2915B}.
This reduces the cost of presenting our result in the more general
Jacobi setting. Large deviations estimates in the quasiperiodic Jacobi
case were also obtained in \cite{MR2563096,MR2825743,MR2915794},
but only for the transfer matrices. 

We proceed by introducing the notation needed to state our main result.
To motivate its statement we will first recall two results from \cite{MR2753606}.

It is known that $a$ and $b$ admit complex analytic extensions.
We will assume that they both extend complex analytically to a set
containing the closure of 
\[
\mb H_{\rho_{0}}:=\left\{ z\in\mb C:\,\left|\Im z\right|<\rho_{0}\right\} ,
\]
for some $\rho_{0}>0$. Let $\t b$ denote the complex analytic extension
of $\bar{b}$ to $\mb H_{\rho_{0}}$ . 

We consider the finite Jacobi submatrix on $\left[0,N-1\right]$,
denoted by $H^{(N)}\left(z,\omega\right)$, and defined by
\[
\left[\begin{array}{ccccc}
a\left(z\right) & -b\left(z+\omega\right) & 0 & \ldots & 0\\
-\t b\left(z+\omega\right) & a\left(z+\omega\right) & -b\left(z+2\omega\right) & \ldots & 0\\
\ddots & \ddots & \ddots & \ldots & \vdots\\
0 & \ldots & 0 & -\t b\left(z+\left(N-1\right)\omega\right) & a\left(z+\left(N-1\right)\omega\right)
\end{array}\right].
\]
It is important for us to use $\t b$ instead of $\bar{b}$, because
we want the determinant to be complex analytic. More generally, we
will denote the finite Jacobi submatrix on $\Lambda=\left[a,b\right]$
by $H_{\Lambda}\left(z,\omega\right)$. Let $E_{j}^{N}\left(z,\omega\right)$,
and $\psi_{j}^{\left(N\right)}\left(z,\omega\right)$, $j=1,\ldots,N$
denote the eigenvalues and the $l^{2}$-normalized eigenvectors of
$H^{\left(N\right)}\left(z,\omega\right)$.

Let $L\left(\omega,E\right)$ be the Lyapunov exponent of the cocycle
associated with $H\left(x,\omega\right)$. Our work deals with the
case of the positive Lyapunov exponent regime. Namely, in this paper
we assume that there exist intervals $\Omega^{0}=\left(\omega',\omega''\right)$,
$\mc E^{0}=\left(E',E''\right)$ such that $L\left(\omega,E\right)>\gamma>0$
for all $\left(\omega,E\right)\in\Omega^{0}\times\mc E^{0}$.

We will be interested in the measure and complexity of sets $S\subset\mb C$.
Writing $\mes\left(S\right)\le c,\,\com\left(S\right)\le C$, will
mean that there exists a set $S'$ such that $S\subset S'\subset\mb C$
and $S'=\cup_{j=1}^{K}\mc D\left(z_{j},r_{j}\right)$, with $K\le C$,
and $\mes\left(S'\right)\le c$.

Goldstein and Schlag proved the following finite scale version of
Anderson localization, in the Schr\"odinger case (see also \cite[Lemma 6.4]{MR2753606}).
We give a restatement of \cite[Corollary 9.10]{MR2753606} adapted
to our setting. Note that in this paper the constants implied by symbols
such as $\lesssim$ will only be absolute constants. 
\begin{prop}
\label{prop:intro-GS-localization}(\cite[Corollary 9.10]{MR2753606})
Given $A>1$ there exists $N_{0}=N_{0}(a,\gamma,\alpha$, $c$, \textup{$\mc E^{0}$,
$A)$} such that for $N\ge N_{0}$ there exist $\Omega_{N}\subset\mb T$,
$\mc E_{N,\omega}\subset\mb R$ with
\[
\mes\left(\Omega_{N}\right)\lesssim\exp\left(-\left(\log\log N\right)^{A}\right),\,\com\left(\Omega_{N}\right)\lesssim N^{4},
\]
\[
\mes\left(\mc E_{N,\omega}\right)\lesssim\exp\left(-\left(\log\log N\right)^{A}\right),\,\com\left(\mc E_{N,\omega}\right)\lesssim N^{4},
\]
satisfying the property that for any $\omega\in\Omega^{0}\cap\mb T_{c,\alpha}\setminus\Omega_{N}$
and any $x\in\mb T$, if $E_{j}^{\left(N\right)}\left(x,\omega\right)\in\mc E^{0}\setminus\mc E_{N,\omega}$
then there exists $\nu_{j}^{\left(N\right)}\left(x,\omega\right)\in\left[0,N-1\right]$
such that if we let
\[
\Lambda_{j}:=\left[\nu_{j}^{\left(N\right)}\left(x,\omega\right)-l,\nu_{j}^{\left(N\right)}\left(x,\omega\right)+l\right]\cap\left[0,N-1\right],\, l=\left(\log N\right)^{4A},
\]
we have that 
\begin{equation}
\left|\psi_{j}^{\left(N\right)}\left(x,\omega;n\right)\right|\le C\exp\left(-\gamma\dist\left(n,\Lambda_{j}\right)/2\right)\label{eq:intro-localization}
\end{equation}
for all $n\in\left[0,N-1\right]$.
\end{prop}
We will call $\nu_{j}^{\left(N\right)}\left(x,\omega\right)$ localization
centre, $\Lambda_{j}$ localization window, and we say that $E_{j}^{\left(N\right)}\left(x,\omega\right)$
is localized when \eqref{intro-localization} holds. By using this
localization result Goldstein and Schlag were able to obtain the following
quantitative separation for the finite scale eigenvalues (see also
\cite[Proposition 7.1]{MR2753606}). As with the previous Proposition,
we give a restatement of \cite[Proposition 10.1]{MR2753606} adapted
to our setting.
\begin{prop}
\label{prop:intro-GS-separation}(\cite[Proposition 10.1]{MR2753606})
Given $0<\delta<1$ there exist large constants $N_{0}=N_{0}\left(\delta,a,\gamma,\alpha,c,\mc E^{0}\right)$
and $A=A\left(\delta,a,\gamma,\alpha,c,\mc E^{0}\right)$ ($\delta A\gg1$)
such that for any $N\ge N_{0}$, and $l=\left(\log N\right)^{A}$
there exist $\Omega_{N}$, $\mc E_{N,\omega}$ as in the previous
Proposition such that for any $\omega\in\Omega^{0}\cap\mb T_{c,\alpha}\setminus\Omega_{N}$
and all $x\in\mb T$ one has
\begin{equation}
\left|E_{j}^{\left(N\right)}\left(x,\omega\right)-E_{k}^{\left(N\right)}\left(x,\omega\right)\right|>\exp\left(-l^{\delta}\right)\label{eq:intro-separation}
\end{equation}
for all $j\neq k$ provided $E_{j}^{\left(N\right)}\left(x,\omega\right)\in\mc E^{0}\setminus\mc E_{N,\omega}$.
\end{prop}
Such separation results play a crucial role in \cite{MR2438997} and
\cite{MR2753606}. It is well-known that $E_{j}^{\left(N\right)}\left(x,\omega\right)$
depends real analytically on $x$ and $\omega$, but we don't have
a priori control on the radius of convergence. Part of the importance
of having such separation results is that they give us control on
the radius of convergence. More specifically, it can be seen that
having the separation from \eqref{intro-separation}, guarantees that
the eigenvalue $E_{j}^{\left(N\right)}\left(\cdot,\cdot\right)$ remains
simple on a polydisk $\mc D\left(x,c\exp\left(-l^{\delta}\right)\right)\times\mc D\left(\omega,c\exp\left(-l^{\delta}\right)/N\right)$,
where $c$ is an absolute constant. Hence we can guarantee that $E_{j}^{\left(N\right)}\left(\cdot,\cdot\right)$
is complex analytic on a polydisk of controlled size. 

The separation achieved through \eqref{intro-separation} is much
smaller than $N^{-1}$, which might be considered the optimal separation.
The goal of our work is to improve the separation given by \eqref{intro-separation},
in an attempt to come closer to the optimal separation. We now state
our main result. A more precise formulation is given by \thmref{slopes-separation-2}.

\theoremstyle{plain}
\newtheorem*{mainthm}{Main Result}

\begin{mainthm}Fix $p>15$. There exist constants $N_{0}=N_{0}\left(a,b,\rho_{0},c,\alpha,\gamma,\mc E^{0},p\right)$,
$c_{0}<1$, such that for any $N\ge N_{0}$ there exists a set $\Omega_{N}$,
with
\[
\mes\left(\Omega_{N}\right)\lesssim\left(\log\log N\right)^{-c_{0}},\,\com\left(\Omega_{N}\right)\lesssim N^{2}\left(\log N\right)^{p},
\]
such that for any $\omega\in\Omega^{0}\cap\mb T_{c,\alpha}\setminus\Omega_{N}$
there exists a set $\mc E_{N,\omega}$, with
\[
\mes\left(\mc E_{N,\omega}\right)\lesssim\left(\log\log N\right)^{-c_{0}},\,\com\left(\mc E_{N,\omega}\right)\lesssim N\left(\log N\right)^{6},
\]
 such that for any $x\in\mb T$, if $E_{j}^{\left(N\right)}\left(x,\omega\right)\in\mathcal{E}^{0}\setminus\mc E_{N,\omega}$,
for some $j$, then
\[
\left|E_{j}^{\left(N\right)}\left(x,\omega\right)-E_{k}^{\left(N\right)}\left(x,\omega\right)\right|\ge\frac{1}{N\left(\log N\right)^{p}},
\]
for any $k\neq j$.\end{mainthm}
\begin{rem*}
The above result is not about an empty set. It is known that 
\[
\mes\left(\cup_{x\in\mb T}\spec\left(H^{\left(N\right)}\left(x,\omega\right)\right)\cap\mc E^{0}\right)\rightarrow\mes\left(\spec\left(H\left(x,\omega\right)\right)\cap\mc E^{0}\right)
\]
 and that $\mes\left(\spec\left(H\left(x,\omega\right)\right)\cap\mc E^{0}\right)>0$
(see \cite[Proposition 13.1 (10),(11)]{MR2753606}). Hence, even though
the set $\mc E_{N,\omega}$ is quite large, the bulk of the spectral
bands will be outside of it.
\end{rem*}
Unsurprisingly, improving the separation comes at the cost of an increase
in size for the sets of bad frequencies and of bad energies. The improved
complexity bound for the set of bad energies is crucial, as we shall
soon see. Our method of proving the main result doesn't directly give
us a complexity bound for $\Omega_{N}$. The stated bound follows
from the stability of the separation under perturbation in $\omega$,
and thus reflects the fact that the separation is less stable under
perturbation when $p$ is larger.

We will obtain our improved separation by first proving an appropriate
finite scale localization result. The known approach for obtaining
localization at scale $N$ is to first eliminate resonances at a smaller
scale $l$. This goes back to Sina\u\i's paper \cite{MR893122}.
 Informally speaking, resonances occur when the spectra of $H_{\Lambda_{1}}$$\left(x,\omega\right)$
and $H_{\Lambda_{2}}\left(x,\omega\right)$ are ``too close'', for
two ``far away'' intervals of length $l$, $\Lambda_{1},\Lambda_{2}\subset\left[0,N-1\right]$
. Specifically, in our case, eliminating resonances on $\left[0,N-1\right]$
at scale $l$ amounts to having the following: there exist constants
$\sigma_{N}$, $Q_{N}$, and a set $\Omega_{N}\subset\mb T$, with
the property that for any $\omega\in\Omega^{0}\cap\mb T_{c,\alpha}\setminus\Omega_{N}$
there exists $\mc E_{N,\omega}\subset\mb R$ such that for any $x\in\mb T$
and any integer $m$, $Q_{N}\le\left|m\right|\le N$, we have
\begin{equation}
\dist\left(\mathcal{E}^{0}\cap\spec\left(H^{\left(l\right)}\left(x,\omega\right)\setminus\mc E_{N,\omega}\right),\spec\left(H^{\left(l\right)}\left(x+m\omega,\omega\right)\right)\right)\ge\sigma_{N}.\label{eq:intro-elimination-of-resonances}
\end{equation}
This condition can be reformulated to hold for all energies in $\mc E^{0}$
at the cost of removing a set of bad phases. However, our improvement
of separation comes at the cost of also losing control over the set
of bad phases, we just have control on the corresponding set of bad
energies. Given such an elimination of resonances, one can prove a
localization result in the spirit of \propref{intro-GS-localization},
with the size of the localization window proportional to $Q_{N}$
(see \thmref{localization}). After establishing localization one
can obtain a separation of the eigenvalues at scale $N$ by $\exp\left(-CQ_{N}\right)$
(see \propref{separation-raw}). Up to this point our strategy is
the one employed by Goldstein and Schlag for the Schr\"odinger case
(see \cite{MR2438997}, \cite{MR2753606}). We will always have $\exp\left(-CQ_{N}\right)\ll\sigma_{N}$,
for the concrete values of $\sigma_{N}$ and $Q_{N}$ that we use.
Using a bootstrapping argument we show that the separation can be
improved to $\sigma_{N}/2$ (see \thmref{separation-bootstrap}).
Note that this can be done only if one is able to ``fatten'' the
set of bad energies $\mc E_{N,\omega}$ by $\sigma_{N}$. For example,
this suggests that the best separation that could be obtained through
\propref{intro-GS-localization} is by $N^{-4+}$. So, our strategy
for obtaining a sharper separation is to improve the elimination of
resonances.

To eliminate resonances we will consider for fixed $j,k,m$, the sets
of $\left(x,\omega\right)$ for which 
\begin{equation}
\left|E_{j}^{\left(l\right)}\left(x,\omega\right)-E_{k}^{\left(l\right)}\left(x+m\omega,\omega\right)\right|<\sigma_{N}.\label{eq:intro-resonance}
\end{equation}
We will need to show that the union over $j$, $k$, $m$ is small
(provided $\left|m\right|$ is large enough). Goldstein and Schlag
approached this problem by using resultants. Let $f_{N}^{a}\left(z,\omega,E\right):=\det\left[H^{\left(N\right)}\left(z,\omega\right)-E\right]$.
The resultant of $f_{l}^{a}\left(x,\omega,E\right)$ and $f_{l}^{a}\left(x+m\omega,\omega,E\right)$
is a polynomial $R\left(x,\omega,E\right)$ with the property that
it vanishes if $E$ is a zero for both determinants. Strictly speaking,
to define $R$, one needs to first use the Weierstrass Preparation
Theorem to factorize the two determinants. For more details see \cite[Section 5]{MR2753606}.
The idea behind considering $R$ is that one can use Cartan's estimate
(see \lemref{prelims-Cartan-estimate}) to eliminate the set where
$\log\left|R\right|$ is too small, and hence remove sets corresponding
to \eqref{intro-resonance}. 

Our approach is based on considering only the parts of the graphs
of the eigenvalues where the slopes are ``good'', i.e. bounded away
from zero. We will be able to control the size of the sets where we
have \eqref{intro-resonance}, by using the following simple observations.
Let $g\left(x,\omega\right)=E_{j}^{\left(l\right)}\left(x,\omega\right)-E_{k}^{\left(l\right)}\left(x+m\omega,\omega\right)$.
If $\left|\partial_{x}E_{k}^{\left(l\right)}\left(x+m\omega,\omega\right)\right|>\tau$,
for some $\tau>0$, it can be seen that $\left|\partial_{\omega}g\left(x,\omega\right)\right|\gtrsim m\tau$,
for $m$ large enough. If for some fixed $x$ and some interval $I$
we have $\left|g\left(x,\omega\right)\right|<\sigma_{N}$ and $\left|\partial_{\omega}g\left(x,\omega\right)\right|\gtrsim m\tau$
for all $\omega\in I$, then the length of $I$ is $\lesssim\sigma_{N}\left(m\tau\right)^{-1}$.
Our main problem will be to control the number of such intervals $I$.
Similar considerations are used by Goldstein and Schlag for the elimination
of the so called triple resonances (see \cite[Section 14]{MR2753606}).
To implement our ideas, one can be tempted to first try to eliminate
$\left(x,\omega\right)$ for which $\left|\partial_{x}E_{k}^{\left(l\right)}\left(x+m\omega,\omega\right)\right|\le\tau$.
Doing this would only yield separation by at most $N^{-2}$ , due
to the dependence on $m$ of the set corresponding to the ``good''
slopes. Instead we will eliminate $\left(x,\omega\right)$ for which
$\left|\partial_{x}E_{j}^{\left(l\right)}\left(x,\omega\right)\right|\le\tau$.
More precisely we will proceed as follows. Using a Sard-type argument
it is possible to show that for fixed $\omega$ and $\tau>0$ we can
find a small set $\mc E_{l,\omega}\left(\tau\right)$ such that for
any $x\in\mb T$, if $E_{j}^{\left(l\right)}\left(x,\omega\right)\notin\mc E_{l,\omega}\left(\tau\right)$,
then $\left|\partial_{x}E_{j}^{\left(l\right)}\left(x,\omega\right)\right|>\tau$.
Let $\t{\mc E}_{l,\omega}\left(\tau\right):=\left\{ E:\,\dist\left(E,\mc E_{l,\omega}\right)<\sigma_{N}\right\} $.
We have that for any $x\in\mb T$, if $E_{j}^{\left(l\right)}\left(x,\omega\right)\notin\t{\mc E}_{l,\omega}\left(\tau\right)$
and \eqref{intro-resonance} holds, then $\left|\partial_{x}E_{k}^{\left(l\right)}\left(x+m\omega,\omega\right)\right|>\tau$.
We stress the fact that the previous statement holds for any $x\in\mb T$,
and thus by fattening the set of bad energies we were able to circumvent
one summation over $m$, which ultimately will allow us to get the
improved separation. We still have to control the complexity of the
set of $\omega$'s such that $\left|g\left(x,\omega\right)\right|<\sigma_{N}$
and $E_{j}^{\left(l\right)}\left(x,\omega\right)\notin\t{\mc E}_{l,\omega}\left(\tau\right)$.
It is not clear how to do this directly. Instead, we will tackle this
problem by working on small intervals $I_{\omega}$ (of controlled
size) around $\omega$ on which we have some stability of the ``good''
slopes, that is, such that there exists a small set $\mc E_{l,I_{\omega}}\left(\tau\right)$
with the property that if $E_{j}^{\left(l\right)}\left(x,\omega'\right)\notin\mc E_{l,I_{\omega}}\left(\tau\right)$,
$\omega'\in I_{\omega}$ then $\left|\partial_{x}E_{j}\left(x,\omega'\right)\right|>\tau$
. In this setting we will need to control the complexity of the set
of frequencies $\omega'\in I_{\omega}$ such that $\left|g\left(x,\omega'\right)\right|<\sigma_{N}$
and $E_{j}^{\left(l\right)}\left(x,\omega'\right)\notin\t{\mc E}_{l,I_{\omega}}\left(\tau\right)$.
This can be achieved by using B\'ezout's Theorem, in the case when
the eigenvalues are algebraic functions (in this case $a$ and $b$
are trigonometric polynomials). The general result will follow through
approximation.

For the stability of the ``good'' slopes under perturbations in
$\omega$ we need the following type of estimate
\[
\left|\partial_{x}E_{j}^{\left(l\right)}\left(x,\omega\right)-\partial_{x}E_{j}^{\left(l\right)}\left(x,\omega'\right)\right|\le C\left|\omega-\omega'\right|.
\]
This can be easily obtained by using Cauchy's Formula, provided we
have control on the size of the polydisk to which $E_{j}^{\left(l\right)}$
extends complex analytically. As we already discussed, such information
can be obtained from a separation result. In the Schr\"odinger case
we have the ``a priori'' separation via resultants. We will need
to prove that this separation also holds in the Jacobi case.

Next we give a brief overview of the article. In \secref{Preliminaries}
we will introduce some more notation, review the basic results needed
for our work, and deduce some useful consequences of these results.
In \secref{localization} and \secref{separation} we establish localization
and separation assuming that we have elimination of resonances, of
the type \eqref{intro-elimination-of-resonances}, with undetermined
$\sigma_{N}$ and $Q_{N}$ (subject to some constraints). Next, in
\secref{Elimination-via-Resultants}, we obtain the elimination of
resonances via resultants and the corresponding localization and separation
results. In \secref{Abstract-Elimination} we prove our elimination
of resonances via slopes in an abstract setting. The reason for choosing
the abstract setting is twofold. First, it makes it straightforward
to obtain elimination with different values of the parameters. We
will need to apply the abstract elimination twice to achieve our stated
separation. Second, we want to emphasize the fact that at its heart
our argument is about algebraic functions, and not specifically about
eigenvalues. In \secref{Elimination-via-Slopes} we will obtain our
main result. Finally, in the Appendix we give the details needed for
some of the results stated in \secref{Preliminaries}.

\section{\label{sec:Preliminaries}Preliminaries}

In this section we present the basic tools that we will be using and
we deduce some useful consequences. We refer to \cite[Section 2]{MR2753606}
for the Schr\"odinger case of these results. 

We proceed by introducing some notation. For $\phi$ satisfying the
difference equation $H\left(z,\omega\right)\phi=E\phi$ let $M_{N}$
be the $N$-step transfer matrix such that
\[
\left[\begin{array}{c}
\phi\left(N\right)\\
\phi\left(N-1\right)
\end{array}\right]=M_{N}\left[\begin{array}{c}
\phi\left(0\right)\\
\phi\left(-1\right)
\end{array}\right],\, N\ge1.
\]
We have 
\[
M_{N}\left(z,\omega,E\right)=\prod_{j=N-1}^{0}\left(\frac{1}{b\left(z+\left(j+1\right)\omega\right)}\left[\begin{array}{cc}
a\left(z+j\omega\right)-E & -\tilde{b}\left(z+j\omega\right)\\
b\left(z+\left(j+1\right)\omega\right) & 0
\end{array}\right]\right),
\]
for $z$ such that $\prod_{j=1}^{N}b\left(z+j\omega\right)\neq0$.
We also consider the following two matrices associated with $M_{N}$:
\begin{equation}
M_{N}^{a}\left(z,\omega,E\right)=\left(\prod_{j=1}^{n}b\left(z+j\omega\right)\right)M_{N}\left(z,\omega,E\right)\label{eq:prelims-Ma-M}
\end{equation}
and
\begin{align*}
M_{N}^{u}\left(z,\omega,E\right) & =\frac{1}{\sqrt{\left|\det M_{N}\left(z,\omega,E\right)\right|}}M_{N}\left(z,\omega,E\right).
\end{align*}
A fundamental property of $M_{N}^{a}$ is that its entries can be
written in terms of the determinant $f_{N}^{a}\left(z,\omega,E\right)$
defined in the introduction:
\begin{equation}
M_{N}^{a}\left(z,\omega,E\right)=\left[\begin{array}{cc}
f_{N}^{a}\left(z,\omega,E\right) & -\t b\left(z\right)f_{N-1}^{a}\left(z+\omega,\omega,E\right)\\
b\left(z+N\omega\right)f_{N-1}^{a}\left(z,\omega,E\right) & -\t b\left(z\right)b\left(z+N\omega\right)f_{N-2}^{a}\left(z+\omega,\omega,E\right)
\end{array}\right]\label{eq:prelims-Ma-fa}
\end{equation}
(see \cite[Chapter 1]{MR1711536}, where such relations are deduced
in a detailed manner). Let $f_{N}^{u}\left(z,\omega,E\right)$ be
such that 
\[
M_{N}^{u}\left(z,\omega,E\right)=\left[\begin{array}{cc}
f_{N}^{u}\left(z,\omega,E\right) & \star\\
\star & \star
\end{array}\right]
\]
($f_{N}^{u}\left(z,\omega,E\right)$ is the determinant of an appropriately
modified Hamiltonian). Based on the definitions, it is straightforward
to check that 
\begin{equation}
\log\left\Vert M_{N}^{u}\left(z,\omega,E\right)\right\Vert =-\frac{1}{2}\left(\t S_{N}\left(z,\omega\right)+S_{N}\left(z+\omega,\omega\right)\right)+\log\left\Vert M_{N}^{a}\left(z,\omega,E\right)\right\Vert ,\label{eq:prelims-Mu-Ma}
\end{equation}
where $S_{N}\left(z,\omega\right)=\sum_{k=0}^{N-1}\log\left|b\left(z+k\omega\right)\right|$
and $\t S_{N}\left(z,\omega\right)=\sum_{k=0}^{N-1}\log\left|\t b\left(z+k\omega\right)\right|$.
Note that $S_{N}\left(x,\omega\right)=\t S_{N}\left(x,\omega\right)$
for $x\in\mb T$. For $y\in\left(-\rho_{0},\rho_{0}\right)$ we let
\[
L_{N}\left(y,\omega,E\right)=\frac{1}{N}\int_{\mathbb{T}}\log\left\Vert M_{N}\left(x+iy,\omega,E\right)\right\Vert dx,
\]
\[
L\left(y,\omega,E\right)=\lim_{N\rightarrow\infty}L_{N}\left(y,\omega,E\right)=\inf_{N\ge1}L_{N}\left(y,\omega,E\right).
\]
We also consider the quantities $L_{N}^{a}$, $L_{N}^{u}$, $L^{a}$,
$L^{u}$ which are defined analogously. Furthermore let $D\left(y\right)=\int_{\mathbb{T}}\log\left|b\left(x+iy\right)\right|dx$.
When $y=0$ we omit the $y$ argument, so for example we write $L\left(\omega,E\right)$
instead of $L\left(0,\omega,E\right)$. It is straightforward to see
that $L_{N}^{u}\left(\omega,E\right)=L_{N}\left(\omega,E\right)$
and hence $L^{u}\left(\omega,E\right)=L\left(\omega,E\right)$. Based
on \eqref{prelims-Mu-Ma} it is easy to conclude that 
\begin{equation}
L\left(\omega,E\right)=-D+L^{a}\left(\omega,E\right).\label{eq:prelims-L-La}
\end{equation}
For a discussion of the objects and quantities introduced above see
\cite[Section 2]{2012arXiv1202.2915B}. We note that in \cite{2012arXiv1202.2915B}
it was more convenient to identify $\mb T$ with the unit circle in
$\mb C$. So for example $a$ and $b$ are considered to be defined
on an annulus $\mc A_{\rho_{0}}$. However, it is trivial to switch
between our setting and that of \cite{2012arXiv1202.2915B}.

In what follows we will keep track of the dependence of the various
constants on the parameters of our problem. In order to simplify the
notation we won't always record the dependence on $\rho_{0}$. Dependence
on any quantity is such that if the quantity takes values in a compact
set, then the constant can be chosen uniformly with respect to that
quantity. We will use $E^{0}$ to denote the quantity $\sup\left\{ \left|E\right|:\, E\in\mc E^{0}\right\} $.
We denote by $\left\Vert \cdot\right\Vert _{\infty}$ the $L^{\infty}$
norm on $\mb H_{\rho_{0}}$ and we let $\left\Vert b\right\Vert _{*}=\left\Vert b\right\Vert _{\infty}+\max_{y\in\left[-\rho_{0},\rho_{0}\right]}$$\left|D\left(y\right)\right|$.
Note that, unless otherwise stated, the constants in different results
are different. Furthermore, in this paper the constants implied by
symbols such as $\lesssim$ will only be absolute constants. 

The following form of the large deviations estimate for the determinants
follows from \cite[Proposition 4.10]{2012arXiv1202.2915B}. We give
a detailed discussion in the Appendix. Note that in the Appendix we
also give a different proof of one of the results \cite{2012arXiv1202.2915B},
which allows us to remove one of the quantities on which the constants
from \cite{2012arXiv1202.2915B} depended.
\begin{prop}
\label{prop:prelims-LDT-determinants}Let $\left(\omega,E\right)\in\mb T_{c,\alpha}\times\mb C$
be such that $L\left(\omega,E\right)>\gamma>0$. There exist constants
$N_{0}=N_{0}\left(\left\Vert a\right\Vert _{\infty},\left\Vert b\right\Vert _{*},\left|E\right|,c,\alpha,\gamma\right)$,
$C_{0}=C_{0}\left(\alpha\right)$, and $C_{1}=C_{1}(\left\Vert a\right\Vert _{\infty},\left\Vert b\right\Vert _{*},\left|E\right|$,
$c$, \textup{$\alpha$, $\gamma)$} such that for every integer $N\ge N_{0}$
and any $H>0$ we have
\[
\mes\left\{ x\in\mathbb{T}:\,\left|\log\left|f_{N}^{a}\left(x,\omega,E\right)\right|-NL^{a}\left(\omega,E\right)\right|>H\left(\log N\right)^{C_{0}}\right\} \le C_{1}\exp\left(-H\right).
\]

\end{prop}
Next we recall a uniform upper bound for the transfer matrix. The
following is a restatement of \cite[Proposition 3.14]{2012arXiv1202.2915B}.
See the appendix for a discussion of this result and of the consequences
that follow.
\begin{prop}
\label{prop:prelims-M^a-upper-bound} Let $\left(\omega,E\right)\in\mb T_{c,\alpha}\times\mb C$
be such that $L\left(\omega,E\right)>\gamma>0$. There exist constants
$C_{0}=C_{0}\left(\alpha\right)$ and $C_{1}=C_{1}\left(\left\Vert a\right\Vert _{\infty},\left\Vert b\right\Vert _{*},\left|E\right|,c,\alpha,\gamma\right)$
such that for any integer $N>1$ we have
\[
\sup_{x\in\mathbb{T}}\log\left\Vert M_{N}^{a}\left(x,\omega,E\right)\right\Vert \le NL^{a}\left(\omega,E\right)+C_{1}\left(\log N\right)^{C_{0}}.
\]

\end{prop}
Note that $\log\left|f_{N}^{a}\left(z,\omega,E\right)\right|\le\log\left\Vert M_{N}^{a}\left(z,\omega,E\right)\right\Vert $,
so this uniform upper bound also applies for the determinants $f_{N}^{a}$.
Next we state two useful consequences of the uniform upper bound from
\propref{prelims-M^a-upper-bound}. See the Appendix for the proofs.

\begin{cor}\label{cor:prelims-uniform-upper-bound}

Let $\left(\omega_{0},E_{0}\right)\in\mb T_{c,\alpha}\times\mb C$
such that $L\left(\omega_{0},E_{0}\right)>\gamma>0$. There exist
constants $N_{0}=\left(\left\Vert a\right\Vert _{\infty},\left\Vert b\right\Vert _{*},\left|E_{0}\right|,c,\alpha,\gamma\right)$,
$C_{0}=C_{0}\left(\alpha\right)$, and $C_{1}=C_{1}(\left\Vert a\right\Vert _{\infty},\left\Vert b\right\Vert _{*}$
$,\left|E_{0}\right|,$ $c,\alpha,\gamma)$ such that for $N\ge N_{0}$
we have
\begin{multline*}
\sup\left\{ \log\left\Vert M_{N}^{a}\left(x+iy,\omega,E\right)\right\Vert :x\in\mb T,\,\left|E-E_{0}\right|,\left|\omega-\omega_{0}\right|\le N^{-C_{1}},\,\left|y\right|\le N^{-1}\right\} \\
\le NL^{a}\left(\omega_{0},E_{0}\right)+\left(\log N\right)^{C_{0}}.
\end{multline*}

\end{cor}

\begin{cor}\label{cor:prelims-lipschitzness}Let $x_{0}\in\mb T$
and $\left(\omega_{0},E_{0}\right)\in\mb T_{c,\alpha}\times\mb C$
such that $L\left(\omega_{0},E_{0}\right)>\gamma>0$. There exist
constants $N_{0}=\left(\left\Vert a\right\Vert _{\infty},\left\Vert b\right\Vert _{*},\left|E_{0}\right|,c,\alpha,\gamma\right)$,
$C_{0}=C_{0}\left(\alpha\right)$, and $C_{1}=C_{1}(\left\Vert a\right\Vert _{\infty},\left\Vert b\right\Vert _{*},\left|E_{0}\right|,$
$c,\alpha,\gamma)$ such that for $N\ge N_{0}$ we have
\begin{multline}
\left\Vert M_{N}^{a}\left(x+iy,\omega,E\right)-M_{N}^{a}\left(x_{0},\omega_{0},E_{0}\right)\right\Vert \le\\
\left(\left|E-E_{0}\right|+\left|\omega-\omega_{0}\right|+\left|x-x_{0}\right|+\left|y\right|\right)\exp\left(NL^{a}\left(\omega_{0},E_{0}\right)+\left(\log N\right)^{C_{0}}\right)\label{eq:prelims-M-lipschitzness}
\end{multline}
and
\begin{multline}
\left|\log\left|f_{N}^{a}\left(x+iy,\omega,E\right)\right|-\log\left|f_{N}^{a}\left(x_{0},\omega_{0},E_{0}\right)\right|\right|\le\\
\left(\left|E-E_{0}\right|+\left|\omega-\omega_{0}\right|+\left|x-x_{0}\right|+\left|y\right|\right)\frac{\exp\left(NL^{a}\left(\omega_{0},E_{0}\right)+\left(\log N\right)^{C_{0}}\right)}{\left|f_{N}^{a}\left(x_{0},\omega_{0},E_{0}\right)\right|},\label{eq:prelims-f-lipschitzness}
\end{multline}
provided $\left|E-E_{0}\right|,\left|\omega-\omega_{0}\right|,\left|x-x_{0}\right|\le N^{-C_{1}}$,
$\left|y\right|\le N^{-1}$, and that the right-hand side of \eqref{prelims-f-lipschitzness}
is less than $1/2$.

\end{cor}

We will also need a version of \corref{prelims-uniform-upper-bound}
for $S_{N}$ and $\t S_{N}$. See the Appendix for a proof.
\begin{lem}
\label{lem:prelims-uniform-bound-S_N} There exist constants $C_{0}=C_{0}\left(\alpha\right)$,
$C_{1}=C_{1}\left(\left\Vert b\right\Vert _{*},c,\alpha\right)$ such
that for every $N>1$ we have
\[
\sup\left\{ S_{N}\left(x+iy,\omega\right):\, x\in\mb T,\,\left|y\right|\le N^{-1}\right\} \le ND+C_{1}\left(\log N\right)^{C_{0}}
\]
and
\[
\sup\left\{ \t S_{N}\left(x+iy,\omega\right):\, x\in\mb T,\,\left|y\right|\le N^{-1}\right\} \le ND+C_{1}\left(\log N\right)^{C_{0}}.
\]

\end{lem}
Next we recall the Avalanche Principle and show how to apply it to
the determinants $f_{N}^{a}$.
\begin{prop}
\label{prop:prelims-AP}( \cite[Proposition 3.3]{MR2438997}) Let
$A_{1},\ldots,A_{n}$, $n\ge2$, be a sequence of $2\times2$ matrices.
If
\begin{equation}
\max_{1\le j\le n}\left|\det A_{j}\right|\le1,\label{eq:prelims-AP-condition1}
\end{equation}
\begin{equation}
\min_{1\le j\le n}\left\Vert A_{j}\right\Vert \ge\mu>n,\label{eq:prelims-AP-condition2}
\end{equation}
and
\begin{equation}
\max_{1\le j<n}\left(\log\left\Vert A_{j+1}\right\Vert +\log\left\Vert A_{j}\right\Vert -\log\left\Vert A_{j+1}A_{j}\right\Vert \right)<\frac{1}{2}\log\mu\label{eq:prelims-AP-condition3}
\end{equation}
then 
\[
\left|\log\left\Vert A_{n}\ldots A_{1}\right\Vert +\sum_{j=2}^{n-1}\log\left\Vert A_{j}\right\Vert -\sum_{j=1}^{n-1}\log\left\Vert A_{j+1}A_{j}\right\Vert \right|<C_{0}\frac{n}{\mu}
\]
with some absolute constant $C_{0}$.\end{prop}
\begin{cor}
\label{cor:prelims-AP-determinants} Let $z\in\mb H_{N^{-1}}$, $\left(\omega,E\right)\in\mb T_{c,\alpha}\times\mb C$
such that $L\left(\omega,E\right)>\gamma>0$, and let $C_{0}$ be
as in \propref{prelims-LDT-determinants}$ $. Let $l_{j}$, $j=1,\ldots,m$,
be positive integers such that $l\le l_{j}\le3l$, $j=1,\ldots,m$,
with $l$ a real number such that $l>2m/\gamma$, and let $s_{k}=\sum_{j<k}l_{j}$
(note that $s_{1}=0$). Assume that there exists $H\in\left(0,l\left(\log l\right)^{-2C_{0}}\right)$
such that 
\[
\log\left|f_{l_{j}}^{a}\left(z+s_{j}\omega,\omega,E\right)\right|>l_{j}L^{a}\left(\omega,E\right)-H\left(\log l_{j}\right)^{C_{0}},\, j=1,\ldots,m,
\]
\[
\log\left|f_{l_{j}+l_{j+1}}^{a}\left(z+s_{j}\omega,\omega,E\right)\right|>\left(l_{j}+l_{j+1}\right)L^{a}\left(\omega,E\right)-H\left(\log\left(l_{j}+l_{j+1}\right)\right)^{C_{0}},
\]
$j=1,\ldots,m-1$. There exists a constant $l_{0}=l_{0}\left(\left\Vert a\right\Vert _{\infty},\left\Vert b\right\Vert _{*},\left|E\right|,c,\alpha,\gamma\right)$
such that if $l\ge l_{0}$ then
\[
\left|\log\left|f_{s_{m+1}}^{a}\left(z,\omega,E\right)\right|+\sum_{j=2}^{m-1}\log\left\Vert A_{j}^{a}\left(z\right)\right\Vert -\sum_{j=1}^{m-1}\log\left\Vert A_{j+1}^{a}\left(z\right)A_{j}^{a}\left(z\right)\right\Vert \right|\lesssim m\exp\left(-\frac{\gamma}{2}l\right),
\]
where
\[
A_{1}^{a}\left(z\right)=A_{1}^{a}\left(z,\omega,E\right)=M_{l_{1}}^{a}\left(z,\omega,E\right)\left[\begin{array}{cc}
1 & 0\\
0 & 0
\end{array}\right],
\]
\[
A_{m}^{a}\left(z\right)=A_{m}^{a}\left(z,\omega,E\right)=\left[\begin{array}{cc}
1 & 0\\
0 & 0
\end{array}\right]M_{l_{m}}^{a}\left(z+s_{m}\omega,\omega,E\right),
\]
and $A_{j}^{a}\left(z\right)=A_{j}^{a}\left(z,\omega,E\right)=M_{l_{j}}^{a}\left(z+s_{j}\omega,\omega,E\right)$,
$j=2,\ldots,m-1$.\end{cor}
\begin{proof}
Note that $\log\left|f_{s_{m+1}}^{a}\left(z\right)\right|=\log\left\Vert \prod_{j=m}^{1}A_{j}^{a}\left(z\right)\right\Vert $.
Essentially, the conclusion follows by applying the Avalanche Principle.
This is straightforward in the Schr\"odinger case. The Jacobi case
is slightly more complicated because the matrices $A_{j}^{a}$ don't
necessarily satisfy \eqref{prelims-AP-condition1}. Let $A_{j}^{u}$
be defined analogously to $A_{j}^{a}$ (using $M_{l}^{u}$ instead
of $M_{l}^{a}$). The matrices $A_{j}^{u}$ satisfy \eqref{prelims-AP-condition1}
and we will be able to apply the Avalanche Principle to them with
$\mu=\exp\left(l\gamma/2\right)$. The conclusion then follows from
the fact that 
\begin{multline*}
\log\left\Vert A_{m}^{u}\left(z\right)\ldots A_{1}^{u}\left(z\right)\right\Vert +\sum_{j=2}^{n-1}\log\left\Vert A_{j}^{u}\left(z\right)\right\Vert -\sum_{j=1}^{n-1}\log\left\Vert A_{j+1}^{u}\left(z\right)A_{j}^{u}\left(z\right)\right\Vert \\
=\log\left\Vert A_{m}^{a}\left(z\right)\ldots A_{1}^{a}\left(z\right)\right\Vert +\sum_{j=2}^{n-1}\log\left\Vert A_{j}^{a}\left(z\right)\right\Vert -\sum_{j=1}^{n-1}\log\left\Vert A_{j+1}^{a}\left(z\right)A_{j}^{a}\left(z\right)\right\Vert .
\end{multline*}
 This identity is a simple consequence of \eqref{prelims-Mu-Ma}.

Now we just need to check that the matrices $A_{j}^{u}$ satisfy \eqref{prelims-AP-condition2}
and \eqref{prelims-AP-condition3} with $\mu=\exp\left(l\gamma/2\right)$.
We have
\begin{multline*}
\log\left\Vert A_{j}^{u}\left(z\right)\right\Vert \ge\log\left|f_{l_{j}}^{u}\left(z+s_{j}\omega,\omega,E\right)\right|\\
=-\frac{1}{2}\left(\t S_{l_{j}}\left(z+s_{j}\omega,\omega\right)+S_{l_{j}}\left(z+\left(s_{j}+1\right)\omega,\omega\right)\right)+\log\left|f_{l_{j}}^{a}\left(z+s_{j}\omega,\omega,E\right)\right|\\
\ge-Dl_{j}-\left(\log l_{j}\right)^{C}+L^{a}l_{j}-H\left(\log l_{j}\right)^{C_{0}}=l_{j}L-\left(\log l_{j}\right)^{C}-H\left(\log l_{j}\right)^{C_{0}}\\
\ge l\frac{\gamma}{2}\ge\log m.
\end{multline*}
For the identities we used \eqref{prelims-Mu-Ma} and \eqref{prelims-L-La}.
For the second inequality we used \lemref{prelims-uniform-bound-S_N}.
The second to last inequality holds for large enough $l$ due to our
assumptions. We also have
\begin{multline*}
\log\left\Vert A_{j}^{u}\left(z\right)\right\Vert +\log\left\Vert A_{j+1}^{u}\left(z\right)\right\Vert -\log\left\Vert A_{j+1}^{u}\left(z\right)A_{j}^{u}\left(z\right)\right\Vert \\
=\log\left\Vert A_{j}^{a}\left(z\right)\right\Vert +\log\left\Vert A_{j+1}^{a}\left(z\right)\right\Vert -\log\left\Vert A_{j+1}^{a}\left(z\right)A_{j}^{a}\left(z\right)\right\Vert \\
\le\log\left\Vert M_{l_{j}}^{a}\left(z+s_{j}\omega\right)\right\Vert +\log\left\Vert M_{l_{j+1}}^{a}\left(z+s_{j+1}\omega\right)\right\Vert -\log\left|f_{l_{j+l_{j+1}}}^{a}\left(z+s_{j}\omega\right)\right|\\
\le l_{j}L^{a}+\left(\log l_{j}\right)^{C}+l_{j+1}L^{a}+\left(\log l_{j+1}\right)^{C}-\left(l_{j}+l_{j+1}\right)L^{a}+H\left(\log\left(l_{j}+l_{j+1}\right)\right)^{C_{0}}\\
\le2\left(\log\left(3l\right)\right)^{C}+H\left(\log\left(6l\right)\right)^{C_{0}}\le\frac{l\gamma}{4}=\frac{1}{2}\log\mu,
\end{multline*}
provided $l$ is large enough. Note that we used \eqref{prelims-Mu-Ma}
and \propref{prelims-M^a-upper-bound}. This concludes the proof.
\end{proof}
The large deviations estimate for the determinants and the uniform
upper bound allows one to use Cartan's estimate. We recall this estimate
in the formulation from \cite{MR2753606}.
\begin{defn}
\label{defn:prelims-Caratheodory}(\cite[Definition 2.1]{MR2753606})
Let $H\gg1$. For an arbitrary set $\mc B\subset\mc D\left(z_{0},1\right)\subset\mb C$
we say that $\mc B\in\car_{1}\left(H,K\right)$ if $\mc B\subset\cup_{j=1}^{j_{0}}\mc D\left(z_{j},r_{j}\right)$
with $j_{0}\le K$, and $\sum_{j}r_{j}\le\exp\left(-H\right)$. If
$d$ is a positive integer greater than one and $\mc B\subset\mc P\left(z^{0},1\right)\subset\mb C^{d}$
then we define inductively that $\mc B\in\car_{d}\left(H,K\right)$
if, for any $1\le j\le d$, there exists $\mc B_{j}\subset\mc D\left(z_{j}^{0},1\right)\subset\mb C$,
$\mc B_{j}\in\car_{1}\left(H,K\right)$ so that $\mc B_{z}^{\left(j\right)}:=\left\{ \left(z_{1},\ldots,z_{d}\right)\in\mc B:\, z_{j}=z\right\} \in\car_{d-1}\left(H,K\right)$
for any $z\in\mb C\setminus\mc B_{j}$.\end{defn}
\begin{lem}
\label{lem:prelims-Cartan-estimate}(\cite[Lemma 2.4]{MR2753606})
Let $\phi\left(z_{1},\ldots,z_{d}\right)$ be an analytic function
defined in a polydisk $\mc P=\mc P\left(z^{0},1\right)$, $z^{0}\in\mb C^{d}$.
Let $M\ge\sup_{z\in\mc P}\log\left|\phi\left(z\right)\right|$, $m\le\log\left|\phi\left(z^{0}\right)\right|$.
Given $H\gg1$, there exists a set $\mc B\subset\mc P$, $\mc B\in\car_{d}\left(H^{1/d},K\right)$,
$K=C_{d}H\left(M-m\right)$, such that
\[
\log\left|\phi\left(z\right)\right|>M-C_{d}H\left(M-m\right),
\]
for any $z\in\mc P\left(z^{0},1/6\right)\setminus\mc B$.
\end{lem}
The following result is a good illustration for the use of Cartan's
estimate. It essentially tells us that the large deviations estimate
for $f_{N}^{a}\left(x,\omega,E\right)$ can only fail if $E$ is close
to the spectrum of $H^{\left(N\right)}\left(x,\omega\right)$.
\begin{prop}
\label{prop:prelims-LDT-failure}Let $H\gg1$ and $\left(\omega,E\right)\in\mb T_{c,\alpha}\times\mb C$
such that $L\left(\omega,E\right)>\gamma>0$. There exist constants
$N_{0}=N_{0}\left(\left\Vert a\right\Vert _{\infty},\left\Vert b\right\Vert _{*},\left|E\right|,c,\alpha,\gamma\right)$
, $C_{0}=C_{0}\left(\alpha\right)$ such that for all $N\ge N_{0}$
and $x\in\mb T$, if 
\begin{equation}
\log\left|f_{N}^{a}\left(x,\omega,E\right)\right|\le NL^{a}\left(\omega,E\right)-H\left(\log N\right)^{C_{0}},\label{eq:ldt-fails}
\end{equation}
then $f_{N}^{a}\left(z,\omega,E\right)=0$ for some $\left|z-x\right|\lesssim N^{-1}\exp\left(-H\right)$.
Furthermore, there exists a constant $C_{1}=C_{1}\left(\left\Vert a\right\Vert _{\infty},\left\Vert b\right\Vert _{\infty}\right)$
such that
\[
\dist\left(E,\spec\left(H^{\left(N\right)}\left(x,\omega\right)\right)\right)\lesssim C_{1}N^{-1}\exp\left(-H\right).
\]
\end{prop}
\begin{proof}
Let $\phi\left(\zeta\right)=f_{N}^{a}\left(x+N^{-1}\zeta,\omega,E\right)$.
By the large deviations estimate for determinants (\propref{prelims-LDT-determinants})
it follows that for large enough $N$ there exists $\zeta_{0}$, $\left|\zeta_{0}\right|<1/100$,
such that $\left|\phi\left(\zeta_{0}\right)\right|>NL^{a}\left(\omega,E\right)-\left(\log N\right)^{C}$.
Using \corref{prelims-uniform-upper-bound} we can apply Cartan's
estimate, \lemref{prelims-Cartan-estimate}, to $\phi$ on $\mc D\left(\zeta_{0},1\right)$,
to get that $\log\left|\phi\left(\zeta\right)\right|>NL^{a}\left(\omega,E\right)-H\left(\log N\right)^{C_{0}}$,
for $\zeta\in\mc D\left(\zeta_{0},1/6\right)\setminus\left(\cup_{j}\mc D\left(\zeta_{j},r_{j}\right)\right)$,
with $\sum_{j}r_{j}\le\exp\left(-H\right)$. By our assumption \eqref{ldt-fails},
it follows that $0\in\mc D\left(\zeta_{j},r_{j}\right)$ for some
$j$. Furthermore there must exist $\zeta'\in\mc D\left(\zeta_{0},1/6\right)\cap\mc D\left(\zeta_{j},r_{j}\right)$
such that $\phi\left(\zeta'\right)=0$, otherwise we can use the minimum
modulus principle to contradict \eqref{ldt-fails}. Now, the first
claim holds with $z=x+N^{-1}\zeta'$. The last claim follows from
the fact that there exists a constant $C_{1}=C_{1}\left(\left\Vert a\right\Vert _{\infty},\left\Vert b\right\Vert _{\infty}\right)$
such that 
\[
\left\Vert H^{\left(N\right)}\left(z,\omega\right)-H^{\left(N\right)}\left(x,\omega\right)\right\Vert \le C_{1}\left|z-x\right|,
\]
 and the fact that $H^{\left(N\right)}\left(x,\omega\right)$ is Hermitian.
\end{proof}
Next we present the key tools for obtaining localization. They are
the Poisson formula in terms of Green's function and a bound on the
off-diagonal terms of Green's function in terms of the deviations
estimate for the determinant $f_{N}^{a}$. We will denote Green's
function by $G_{N}\left(z,\omega,E\right):=\left(H^{\left(N\right)}\left(z,\omega\right)-E\right)^{-1}$,
or in general $G_{\Lambda}\left(z,\omega,E\right):=\left(H_{\Lambda}\left(z,\omega\right)-E\right)^{-1}$.
It is known that any solution $\psi$ of the difference equation $H\left(z,\omega\right)\psi=E\psi$
satisfies the Poisson formula:
\begin{equation}
\psi\left(m\right)=G_{\left[a,b\right]}\left(z,\omega,E\right)\left(m,a\right)\psi\left(a-1\right)+G_{\left[a,b\right]}\left(z,\omega,E\right)\left(m,b\right)\psi\left(b+1\right),\label{eq:Poisson}
\end{equation}
for any $\left[a,b\right]$ and $m\in\left[a,b\right]$. Using Cramer's
rule one can explicitly write the entries of Green's function. Namely,
we have that $G_{N}\left(z,\omega,E\right)\left(j,k\right)$ is given
by 
\[
\begin{cases}
\dfrac{f_{j-1}^{a}\left(z,\omega,E\right)b\left(z+j\omega\right)\ldots b\left(z+\left(k-1\right)\omega\right)f_{N-\left(k+1\right)}^{a}\left(z+\left(k+1\right)\omega,\omega,E\right)}{f_{N}^{a}\left(z,\omega,E\right)} & ,\, j<k\\
\dfrac{f_{k-1}^{a}\left(z,\omega,E\right)\t b\left(z+k\omega\right)\ldots\t b\left(z+\left(j-1\right)\omega\right)f_{N-\left(j+1\right)}^{a}\left(z+\left(j+1\right)\omega,\omega,E\right)}{f_{N}^{a}\left(z,\omega,E\right)} & ,\, j>k\\
\dfrac{f_{j-1}^{a}\left(z,\omega,E\right)f_{N-\left(k+1\right)}^{a}\left(z+\left(k+1\right)\omega,\omega,E\right)}{f_{N}^{a}\left(z,\omega,E\right)} & ,\, j=k.
\end{cases}
\]

\begin{lem}
\label{lem:prelims-bounds-Green-entries} Let $\left(\omega,E\right)\in\mb T_{c,\alpha}\times\mb C$
such that $L\left(\omega,E\right)>\gamma>0$. There exist constants
$N_{0}=N_{0}\left(\left\Vert a\right\Vert _{\infty},\left\Vert b\right\Vert _{*},\left|E\right|,c,\alpha,\gamma\right)$,
$C_{0}=C_{0}\left(\alpha\right)$, such that for $N\ge N_{0}$ we
have that if
\[
\log\left|f_{N}^{a}\left(x,\omega,E\right)\right|\ge NL_{N}^{a}\left(\omega,E\right)-K/2,
\]
for some $x\in\mb T$ and $K>\left(\log N\right)^{C_{0}}$, then 
\[
\left|G_{N}\left(x,\omega,E\right)\left(j,k\right)\right|\le\exp\left(-\gamma\left|k-j\right|+K\right).
\]
\end{lem}
\begin{proof}
Assume $j<k$. Then we have
\begin{multline*}
\left|G_{N}\left(z,\omega,E\right)\right|=\frac{\left|f_{j-1}^{a}\left(x,\omega,E\right)\right|\exp\left(S_{k-j}\left(x+j\omega,\omega\right)\right)\left|f_{N-\left(k+1\right)}^{a}\left(x+\left(k+1\right)\omega,\omega,E\right)\right|}{\left|f_{N}^{a}\left(x,\omega,E\right)\right|}\\
\le\exp\left(\left(j-1\right)L^{a}+\left(k-j\right)D+\left(N-k-1\right)L^{a}-NL^{a}+\frac{K}{2}+\left(\log N\right)^{C}\right)\\
=\exp\left(\left(k-j\right)\left(D-L^{a}\right)-2L^{a}+\frac{K}{2}+\left(\log N\right)^{C}\right)\le\exp\left(-\gamma\left(k-j\right)+K\right).
\end{multline*}
 We used \propref{prelims-M^a-upper-bound}, \lemref{prelims-uniform-bound-S_N},
and \eqref{prelims-L-La}. The cases $j=k$ and $j>k$ are analogous. 
\end{proof}
Finally, the following result is needed for the Weierstrass Preparation
of the determinants (see \propref{resultants-Weierstrass-preparation-for-determinants}).
The statement of the result is adapted to our setting.
\begin{prop}
\label{prop:prelims-number-of-ev}(\cite[Theorem 4.13]{2012arXiv1202.2915B})
Let $\left(\omega,E_{0}\right)\in\mb T_{c,\alpha}\times\mb C$ such
that $L\left(\omega,E_{0}\right)>\gamma>0$. There exist constants
\textup{$C_{0}=C_{0}\left(\alpha\right)$,} $C_{1}=C_{1}\left(\left\Vert a\right\Vert _{\infty},\left\Vert b\right\Vert _{*},\left|E_{0}\right|,c,\alpha,\gamma\right)$,
and $N_{0}=N_{0}(\left\Vert a\right\Vert _{\infty},\left\Vert b\right\Vert _{*},\left|E_{0}\right|,c,\alpha,\gamma)$
such that for any $x_{0}\in\mathbb{T}$ and $N\ge N_{0}$ one has
\[
\#\left\{ E\in\mathbb{R}:\, f_{N}^{a}\left(x_{0},\omega,E\right)=0,\,\left|E-E_{0}\right|<N^{-C_{1}}\right\} \le C_{1}\left(\log N\right)^{C_{0}}
\]
and
\[
\#\left\{ z\in\mb C:\, f_{N}^{a}\left(z,\omega,E_{0}\right)=0,\,\left|z-x_{0}\right|<N^{-1}\right\} \le C_{1}\left(\log N\right)^{C_{0}}.
\]

\end{prop}

\section{\label{sec:localization}Localization}

In this section we will show that elimination of resonances implies
localization. More precisely we will assume that we have the following
elimination of resonances result.

\theoremstyle{plain}
\newtheorem{elimination-assumption}[thm]{Elimination  Assumption}

\begin{elimination-assumption}\label{localization-elimination-assumption}
Let $A=A\left(\alpha\right)$ be a fixed constant, much larger than
the $C_{0}$ constants from \corref{prelims-uniform-upper-bound},
\corref{prelims-lipschitzness}, and \lemref{prelims-bounds-Green-entries}.
Let $l=2\left[\left(\log N\right)^{A}\right]$. We assume that there
exists a constant $N_{0}=N_{0}\left(\left\Vert a\right\Vert _{\infty},\left\Vert b\right\Vert _{*},c,\alpha,\gamma,E^{0}\right)$
such that for any $N\ge N_{0}$ there exist constants $\sigma_{N}\gg\exp\left(-l^{1/4}\right)$,
$Q_{N}\gg l^{3}$, and a set $\Omega_{N}\subset\mb T$, with the property
that for any $\omega\in\Omega^{0}\cap\mb T_{c,\alpha}\setminus\Omega_{N}$
there exists a set $\mc E_{N,\omega}\subset\mb R$ such that for any
$x\in\mb T$ and any integer $m$, $Q_{N}\le\left|m\right|\le N$,
we have
\begin{equation}
\dist\left(\mathcal{E}^{0}\cap\spec\left(H^{\left(l_{1}\right)}\left(x,\omega\right)\right)\setminus\mc E_{N,\omega},\spec\left(H^{\left(l_{2}\right)}\left(x+m\omega,\omega\right)\right)\right)\ge\sigma_{N},\label{eq:localization-hyp}
\end{equation}
$l_{1},l_{2}\in\left\{ l,l+1,2l,2l+1\right\} $.\end{elimination-assumption}
Similarly to \cite{MR2753606}, we could have assumed that we have
elimination between any scales $l_{1}$, $l_{2}$, $l\le l_{1},l_{2}\le3l$.
However, this would lead to an extra $\log N$ power in our final
separation result. We note that for localization it is enough to assume
$l_{1},l_{2}\in\left\{ l,2l\right\} $, and that the stronger assumption
is needed in the next section, for obtaining separation. 

In this section and the next, all the results hold under the implicit
assumption that $N$ is large enough, as needed. The lower bound on
$N$ will depend on all the parameters of the problem (as in the Elimination
Assumption \ref{localization-elimination-assumption}).

The following lemma is the basic mechanism through which elimination
of resonances enters the proof of localization. As a consequence of
\propref{prelims-LDT-failure}, it shows that the large deviations
estimate for $f_{l}^{a}\left(x+m\omega,\omega,E\right)$ can only
fail for shifts $m$ in a ``small'' interval (that will end up being
the localization window). 
\begin{lem}
\label{lem:localization-LDT-failure}For all $x\in\mb T$, $\omega\in\Omega^{0}\cap\mathbb{T}_{c,a}\setminus\Omega_{N}$,
and $E\in\mathcal{E}^{0}$, $\dist\left(E,\mc E_{N,\omega}\cup\left(\mathcal{E}^{0}\right)^{C}\right)\gtrsim\exp\left(-l^{1/4}\right)$,
if we have
\begin{equation}
\log\left|f_{l}^{a}\left(x+n_{1}\omega,\omega,E\right)\right|\le lL_{l}^{a}-\sqrt{l},\label{eq:localization-LDT-fails-scale-l-n1}
\end{equation}
for some $n_{1}\in\left[0,N-1\right]$, then 
\begin{equation}
\log\left|f_{l'}^{a}\left(x+n\omega,\omega,E\right)\right|>l'L^{a}\left(\omega,E\right)-\sqrt{l'},\, l'\in\left\{ l,l+1,2l,2l+1\right\} ,\label{eq:localization-fl>lL-}
\end{equation}
 for all $n\in\left[0,N-1\right]\setminus\left[n_{1}-Q_{N},n_{1}+Q_{N}\right]$. \end{lem}
\begin{proof}
Fix $x\in\mb T$, $\omega\in\mathbb{T}_{c,a}\setminus\Omega_{N}$,
and $E\in\mathcal{E}^{0}$, such that 
\begin{equation}
\dist\left(E,\mc E_{N,\omega}\cup\left(\mathcal{E}^{0}\right)^{C}\right)\gtrsim\exp\left(-l^{1/4}\right).\label{eq:localization-hyp-energy}
\end{equation}
Suppose there exists $n_{1}\in\left[0,N-1\right]$ such that \eqref{localization-LDT-fails-scale-l-n1}
holds. By \propref{prelims-LDT-failure} we have that there exists
$E_{k}^{\left(l\right)}\left(x+n_{1}\omega,\omega\right)$ such that
$\left|E_{k}^{\left(l\right)}\left(x+n_{1}\omega,\omega\right)-E\right|\le\exp\left(-l^{1/3}\right)$.
Due to \eqref{localization-hyp-energy} we have that $E_{k}^{\left(l\right)}\left(x+n_{1}\omega,\omega\right)\in\mathcal{E}^{0}\setminus\mc E_{N,\omega}$.
If \eqref{localization-fl>lL-} doesn't hold for $n\in\left[0,N-1\right]\setminus\left[n_{1}-Q_{N},n_{1}+Q_{N}\right]$,
then there exists $E_{k'}^{\left(l'\right)}\left(x+n\omega,\omega\right)$
such that $\left|E_{k'}^{\left(l'\right)}\left(x+n\omega,\omega\right)-E\right|\le\exp\left(-l^{1/3}\right)$,
and hence 
\[
\left|E_{k}^{\left(l\right)}\left(x+n_{1}\omega,\omega\right)-E_{k'}^{\left(l'\right)}\left(x+n\omega,\omega\right)\right|\lesssim\exp\left(-l^{1/3}\right).
\]
This contradicts \eqref{localization-hyp}, and thus concludes the
proof.
\end{proof}
We can now apply the Avalanche Principle to obtain large deviations
estimates at scales larger than $l$. 
\begin{cor}
\label{cor:localization-AP}Under the same assumptions as in \lemref{localization-LDT-failure}
and with $n_{1}$ as in \lemref{localization-LDT-failure}, we have
\begin{equation}
\left|f_{\left[0,n-1\right]}^{a}\left(x,\omega,E\right)\right|>\exp\left(nL^{a}\left(\omega,E\right)-l^{3}\right)\label{eq:localization-LDT-scale-n}
\end{equation}
for each $n=kl,kl+1$, $0\le n\le n_{1}-Q_{N}$, and
\begin{equation}
\left|f_{\left[n,N-1\right]}^{a}\left(x,\omega,E\right)\right|>\exp\left(\left(N-n\right)L^{a}\left(\omega,E\right)-l^{3}\right),\label{eq:localization-LDT-scale-N-n}
\end{equation}
for each $n_{1}+Q_{N}\le n=N-kl\le N-1$, $k\in\mb Z$.\end{cor}
\begin{proof}
We only prove \eqref{localization-LDT-scale-n} for $n=kl$. The other
claims follow in the same way. 

Suppose that $n=kl$ and \eqref{localization-LDT-scale-n} fails.
Then by \propref{prelims-LDT-failure} we have $f_{\left[0,n-1\right]}^{a}\left(z\right)=f_{n}^{a}\left(z\right)=0$
for $z$ such that $\left|z-x\right|\lesssim n^{-1}\exp\left(-l^{3}/\left(\log n\right)^{C_{0}}\right)\lesssim\exp\left(-l^{2}\right)$
(the last inequality holds due to our choice of $l$ in the Elimination
Assumption \ref{localization-elimination-assumption}). Using \corref{prelims-lipschitzness}
we can conclude that 
\[
\log\left|f_{l'}^{a}\left(z+k\omega\right)\right|>l'L^{a}-2\sqrt{l'},\, l'\in\left\{ l,l+1,2l,2l+1\right\} ,
\]
for all $k\in\left[0,N-1\right]\setminus\left[n_{1}-Q_{N},n_{1}+Q_{N}\right]$.
We can now use \corref{prelims-AP-determinants} and \corref{prelims-uniform-upper-bound}
to get
\begin{multline*}
\log\left|f_{n}^{a}\left(z\right)\right|\gtrsim-k\exp\left(-\frac{\gamma}{2}l\right)-\sum_{j=2}^{k-1}\log\left\Vert A_{j}^{a}\left(z\right)\right\Vert +\sum_{j=1}^{k-1}\log\left\Vert A_{j+1}^{a}\left(z\right)A_{j}^{a}\left(z\right)\right\Vert \\
\gtrsim-k\exp\left(-\frac{\gamma}{2}l\right)-\sum_{j=2}^{k-1}\log\left\Vert M_{l}^{a}\left(z+\left(j-1\right)l\omega\right)\right\Vert +\sum_{j=1}^{k-1}\log\left|f_{2l}^{a}\left(z+\left(j-1\right)l\omega\right)\right|\\
\gtrsim-k\exp\left(-\frac{\gamma}{2}l\right)-\left(k-2\right)\left(lL^{a}+\left(\log l\right)^{C}\right)+\left(k-1\right)\left(2lL^{a}-2\sqrt{2l}\right)\\
\gtrsim klL^{a}-4k\sqrt{l}.
\end{multline*}
This contradicts $f_{n}^{a}\left(z\right)=0$. Hence we proved that
\eqref{localization-LDT-scale-n} holds.
\end{proof}
We have all we need to obtain localization.
\begin{thm}
\label{thm:localization}For all $x\in\mb T$, $\omega\in\Omega^{0}\cap\mathbb{T}_{c,a}\setminus\Omega_{N}$,
if the eigenvalue $E_{j}^{\left(N\right)}\left(x,\omega\right)$ is
such that $\dist\left(E_{j}^{\left(N\right)}\left(x,\omega\right),\mc E_{N,\omega}\cup\left(\mathcal{E}^{0}\right)^{C}\right)\gtrsim\exp\left(-l^{1/4}\right)$,
then there exists $\nu_{j}^{\left(N\right)}\left(x,\omega\right)\in\left[0,N-1\right]$
so that for any $\Lambda=\left[a,b\right]$,
\[
\left[\nu_{j}^{\left(N\right)}\left(x,\omega\right)-3Q_{N},\nu_{j}^{\left(N\right)}\left(x,\omega\right)+3Q_{N}\right]\cap\left[0,N-1\right]\subset\Lambda\subset\left[0,N-1\right],
\]
if we let $Q=\dist\left(\left[0,N-1\right]\setminus\Lambda,\nu_{j}^{\left(N\right)}\left(x,\omega\right)\right)$
we have:
\begin{enumerate}
\item 
\begin{equation}
\sum_{k\in\left[0,N-1\right]\setminus\Lambda}\left|\psi_{j}^{\left(N\right)}\left(x,\omega;k\right)\right|^{2}<\exp\left(-\gamma Q\right),\label{eq:localization-ef}
\end{equation}

\item 
\begin{equation}
\dist\left(E_{j}^{\left(N\right)}\left(x,\omega\right),\spec\left(H_{\Lambda}\left(x,\omega\right)\right)\right)\lesssim\exp\left(-\gamma Q\right).\label{eq:localization-ev}
\end{equation}

\end{enumerate}
\end{thm}
\begin{proof}
Fix $x\in\mb T$, $\omega\in\Omega^{0}\cap\mathbb{T}_{c,a}\setminus\Omega_{N}$,
and $E=E_{j}^{\left(N\right)}\left(x,\omega\right)$, satisfying our
assumptions. Let $\nu_{j}^{\left(N\right)}\left(x,\omega\right)$
be such that
\[
\left|\psi_{j}^{\left(N\right)}\left(x,\omega;\nu_{j}^{\left(N\right)}\left(x,\omega\right)\right)\right|=\max_{0\le n\le N-1}\left|\psi_{j}^{\left(N\right)}\left(x,\omega;n\right)\right|.
\]
Let $\Lambda_{0}=\left[a_{0},b_{0}\right]\subset\left[0,N-1\right]$
be the interval of length $l$ such that 
\[
\Lambda_{0}\supset\left(\left[\nu_{j}^{\left(N\right)}\left(x,\omega\right)-l/2,\nu_{j}^{\left(N\right)}\left(x,\omega\right)+l/2\right]\cap\left[0,N-1\right]\right).
\]
 We claim that
\begin{equation}
\log\left|f_{\Lambda_{0}}^{a}\left(x,\omega,E\right)\right|\le lL^{a}-\sqrt{l}.\label{eq:localization-Lambda0-LDT-fail}
\end{equation}
Otherwise, \lemref{prelims-bounds-Green-entries} implies that 
\[
\left|G_{\Lambda_{0}}\left(x,\omega,E\right)\left(j,k\right)\right|\le\exp\left(-\gamma\left|k-j\right|+2\sqrt{l}\right),
\]
for all $j,k\in\Lambda_{0}$. This, together with Poisson's formula
\eqref{Poisson} would contradict the maximality of $\left|\psi_{j}^{\left(N\right)}\left(x,\omega;\nu_{j}^{\left(N\right)}\left(x,\omega\right)\right)\right|$.

We note for future reference that \eqref{localization-Lambda0-LDT-fail}
and \propref{prelims-LDT-failure} imply the existence of $E_{k}^{\left(l\right)}\left(x+a_{0}\omega,\omega\right)$
such that
\begin{equation}
\left|E_{k}^{\left(l\right)}\left(x+a_{0}\omega,\omega\right)-E_{j}^{\left(N\right)}\left(x,\omega\right)\right|\le\exp\left(-l^{1/3}\right).\label{eq:localization-EN-El}
\end{equation}

Let $k\in\left[0,N-1\right]$, $k\le\nu_{j}^{\left(N\right)}\left(x,\omega\right)-Q$.
Due to \eqref{localization-Lambda0-LDT-fail} we can apply \corref{localization-AP},
with $n_{1}=a_{0}$, $n=l\left[\left(n_{1}-Q_{n}\right)/l\right]$
to get that
\[
\log\left|f_{\left[0,n-1\right]}^{a}\left(x\right)\right|\ge nL^{a}-l^{3}.
\]
Now we can apply \lemref{prelims-bounds-Green-entries} and \eqref{Poisson}
to get
\begin{multline*}
\left|\psi_{j}^{\left(N\right)}\left(x,\omega;k\right)\right|^{2}\le\left|G_{\left[0,n-1\right]}\left(x,\omega\right)\left(k,n-1\right)\right|^{2}\le\exp\left(-2\gamma\left(n-1-k\right)+4l^{3}\right)\\
\le\exp\left(-2\gamma\left(n_{1}-Q_{N}-l-1-\nu_{j}^{\left(N\right)}\left(x,\omega\right)+Q\right)+4l^{3}\right)\le\exp\left(-\frac{3\gamma Q}{2}\right)
\end{multline*}
(we used $\nu_{j}^{\left(N\right)}\left(x,\omega\right)-n_{1}\le l/2$,
$Q\ge3Q_{N}\gg l^{3}$). Similarly, we obtain the same bound when
$k\ge\nu_{j}^{\left(N\right)}\left(x,\omega\right)+Q$. Summing up
these bounds gives us \eqref{localization-ef}.

Due to \eqref{localization-ef} we have
\[
\left\Vert \left(H_{\Lambda}\left(x,\omega\right)-E_{j}^{\left(N\right)}\left(x,\omega\right)\right)\left(\psi_{j}^{\left(N\right)}\vert_{\Lambda}\right)\right\Vert <\exp\left(-\gamma Q\right).
\]
Since $H_{\Lambda}$ is Hermitian, and $\left\Vert \psi_{j}^{\left(N\right)}\vert_{\Lambda}\right\Vert >1-\exp\left(\gamma Q\right)$,
we can conclude that
\[
\dist\left(E_{j}^{\left(N\right)}\left(x,\omega\right),\spec\left(H_{\Lambda}\left(x,\omega\right)\right)\right)<\exp\left(-\gamma Q\right)\left(1-\exp\left(-\gamma Q\right)\right)^{-1}\lesssim\exp\left(-\gamma Q\right).
\]

\end{proof}

\section{\label{sec:separation}Separation of Eigenvalues}

In this section we continue to work under the Elimination Assumption
\ref{localization-elimination-assumption}. The basic idea behind
proving separation of eigenvalues is to use the fact that the eigenvectors
are orthogonal, and so they cannot be too close. It is known that
if $E$ is an eigenvalue of the Dirichlet problem on $\left[0,N-1\right]$
then $\frak{f}:=\left(f_{\left[0,n-1\right]}^{a}\left(x,\omega,E\right)\right)_{n=0}^{N-1}$
is an eigenvector associated with $E$ ($f_{\left[0,-1\right]}^{a}=1$).
Note that we are assuming the boundary conditions $\mf f\left(-1\right)=\mf f\left(N\right)=0$.
We will need the following lemma to argue that if two localized eigenvalues
are close enough, then they have eigenvectors which are also close,
at least before the localization window. 
\begin{lem}
\label{lem:separation-close-ev->close-ef}Let $x\in\mb{\mb T}$, $\omega\in\Omega^{0}\cap\mb T_{c,\alpha}\setminus\Omega_{N}$,
and suppose that 
\[
\dist\left(E_{j}^{\left(N\right)}\left(x,\omega\right),\mc E_{N,\omega}\cup\left(\mathcal{E}^{0}\right)^{C}\right)\gtrsim\exp\left(-l^{1/4}\right)
\]
for some $j$. If $E$ is such that $\left|E-E_{j}^{\left(N\right)}\left(x,\omega\right)\right|\le N^{-C_{1}}$,
with $C_{1}$ as in \corref{prelims-lipschitzness}, then
\begin{multline*}
\left|f_{\left[0,n-1\right]}^{a}\left(x,\omega,E\right)-f_{\left[0,n-1\right]}^{a}\left(x,\omega,E_{j}^{\left(N\right)}\left(x,\omega\right)\right)\right|\\
\le\exp\left(2l^{3}\right)\left|E-E_{j}^{\left(N\right)}\left(x,\omega\right)\right|\left|f_{\left[0,n-1\right]}^{a}\left(x,\omega,E_{j}^{\left(N\right)}\left(x,\omega\right)\right)\right|,
\end{multline*}
for each $n=kl,kl+1$, $k\in\mb Z$, $0\le n\le\nu_{j}^{\left(N\right)}\left(x,\omega\right)-2Q_{N}$,
where $\nu_{j}^{\left(N\right)}\left(x,\omega\right)$ is the localization
center corresponding to $E_{j}^{\left(N\right)}\left(x,\omega\right)$
(as in \thmref{localization}).\end{lem}
\begin{proof}
This follows immediately from \eqref{prelims-M-lipschitzness} and
\corref{localization-AP}.
\end{proof}
The next lemma shows that if two localized eigenvalues are close enough,
then their localization centers are also close. 
\begin{lem}
\label{lem:separation-close-ev->close-localization-centers}Let $x\in\mb{\mb T}$,
$\omega\in\Omega^{0}\cap\mb T_{c,\alpha}\setminus\Omega_{N}$ and
suppose that 
\[
\dist\left(E_{j_{i}}^{\left(N\right)}\left(x,\omega\right),\mc E_{N,\omega}\cup\left(\mathcal{E}^{0}\right)^{C}\right)\gtrsim\exp\left(-l^{1/4}\right),\, i=1,2.
\]
If $\left|E_{j_{1}}^{\left(N\right)}\left(x,\omega\right)-E_{j_{2}}^{\left(N\right)}\left(x,\omega\right)\right|\le\sigma_{N}/2$,
then both eigenvalues are localized and if we denote their localization
centers by $\nu_{j_{i}}^{\left(N\right)}\left(x,\omega\right)$, $i=1,2$,
we have $\left|\nu_{j_{1}}^{\left(N\right)}\left(x,\omega\right)-\nu_{j_{2}}^{\left(N\right)}\left(x,\omega\right)\right|<2Q_{N}$.\end{lem}
\begin{proof}
As was noted in the proof of \thmref{localization} (see \eqref{localization-EN-El})
we have that
\begin{equation}
\left|E_{j_{i}}^{\left(N\right)}\left(x,\omega\right)-E_{k_{i}}^{\left(l\right)}\left(x+n_{i}\omega,\omega\right)\right|\le\exp\left(-l^{1/3}\right),\, i=1,2,\label{eq:separation-evN-close-evl}
\end{equation}
where $n_{i}$ are such that $\left|\nu_{j_{i}}^{\left(N\right)}\left(x,\omega\right)-n_{i}\right|\le l/2$,
$i=1,2$.

Suppose that $\left|n_{1}-n_{2}\right|\ge Q_{N}$. Due to \eqref{separation-evN-close-evl}
we have that $E_{k_{1}}^{\left(l\right)}\left(x,\omega\right)\in\mathcal{E}^{0}\setminus\mc E_{N,\omega}$
and hence, by \eqref{localization-hyp} we have 
\[
\left|E_{k_{1}}^{\left(l\right)}\left(x+n_{1}\omega,\omega\right)-E_{k_{2}}^{\left(l\right)}\left(x+n_{2}\omega,\omega\right)\right|\ge\sigma_{N}.
\]
The above inequality together with \eqref{separation-evN-close-evl}
and the assumption that $\sigma_{N}\gg\exp\left(-l^{1/4}\right)$,
implies that $\left|E_{j_{1}}^{\left(N\right)}\left(x,\omega\right)-E_{j_{2}}^{\left(N\right)}\left(x,\omega\right)\right|>\sigma_{N}/2$,
contradicting our assumptions. So, we must have $\left|n_{1}-n_{2}\right|<Q_{N}$
and consequently $\left|\nu_{j_{1}}^{\left(N\right)}\left(x,\omega\right)-\nu_{j_{2}}^{\left(N\right)}\left(x,\omega\right)\right|\le Q_{N}+l<2Q_{N}$. 
\end{proof}
We are now ready to prove a first version of separation, based on
the size of the localization window. This is a generalization of \cite[Proposition 7.1]{MR2753606}. 
\begin{prop}
\label{prop:separation-raw}There exists a constant $C_{0}=C_{0}\left(\left\Vert a\right\Vert _{\infty},\left\Vert b\right\Vert _{\infty},E^{0}\right)$
such that for all $x\in\mb T$, $\omega\in\Omega^{0}\cap\mathbb{T}_{c,\alpha}\setminus\Omega_{N}$,
if $\dist\left(E_{j}^{\left(N\right)}\left(x,\omega\right),\mc E_{N,\omega}\cup\left(\mathcal{E}^{0}\right)^{C}\right)\gtrsim\exp\left(-l^{1/4}\right)$
for some $j$, then
\[
\left|E_{j}^{\left(N\right)}\left(x,\omega\right)-E_{k}^{\left(N\right)}\left(x,\omega\right)\right|>\exp\left(-C_{0}Q_{N}\right)
\]
for any $k\neq j$.\end{prop}
\begin{proof}
Fix $x\in\mathbb{T}$, $\omega\in\Omega^{0}\cap\mathbb{T}_{c,\alpha}\setminus\Omega_{N}$
and $E_{1}=E_{j}^{\left(N\right)}\left(x,\omega\right)$ satisfying
the assumptions. Suppose there exists $E_{2}=E_{k}^{\left(N\right)}\left(x,\omega\right)\neq E_{1}$
such that $\left|E_{1}-E_{2}\right|\le\exp\left(-C_{0}Q_{N}\right)$.
This implies that $\dist\left(E_{k}^{\left(N\right)}\left(x,\omega\right),\mc E_{N,\omega}\cup\left(\mathcal{E}^{0}\right)^{\mc C}\right)\gtrsim\exp\left(-l^{1/4}\right)$
(recall that $Q_{N}\gg l^{3}$). Hence, by \thmref{localization},
both $E_{j}^{\left(N\right)}\left(x,\omega\right)$, and $E_{k}^{\left(N\right)}\left(x,\omega\right)$
are localized.

We know $\frak{f}_{i}:=\left(f_{\left[0,n-1\right]}^{a}\left(x,\omega,E_{i}\right)\right)_{n=0}^{N-1}$
are eigenvectors corresponding to $E_{i}$, $i=1,2$. Furthermore
$\mf f_{i}\left(-1\right)=\mf f_{i}\left(N\right)=0$. If we let 
\[
\Lambda=\left[a,b\right]=\left[0,N-1\right]\cap\left[\nu_{j}^{\left(N\right)}\left(x,\omega\right)-5Q_{N},\nu_{j}^{\left(N\right)}\left(x,\omega\right)+5Q_{N}\right],
\]
then due to \eqref{localization-ef} and \lemref{separation-close-ev->close-localization-centers}
we have
\begin{equation}
\sum_{n\in\left[0,N-1\right]\setminus\Lambda}\left|\frak{f}_{i}\left(n\right)\right|^{2}\lesssim\exp\left(-5\gamma Q_{N}\right)\sum_{n\in\Lambda}\left|\frak{f}_{i}\left(n\right)\right|^{2},\, i=1,2,\label{eq:separation-ef-localization-for-f}
\end{equation}
and consequently
\begin{equation}
\sum_{n\in\left[0,N-1\right]\setminus\Lambda}\left|\frak{f}_{1}\left(n\right)-\frak{f}_{2}\left(n\right)\right|^{2}\lesssim\exp\left(-5\gamma Q_{N}\right)\sum_{n\in\Lambda}\left(\left|\frak{f}_{1}\left(n\right)\right|^{2}+\left|\frak{f}_{2}\left(n\right)\right|^{2}\right).\label{eq:separation-sum-f1-f2-over-outside-Lambda}
\end{equation}

Let 
\[
m=\begin{cases}
\left[\left(a-2\right)/l\right]l & ,\, a>l+1\\
-1 & ,\, a\le l+1
\end{cases}.
\]
For $n\in\Lambda$ we have
\begin{multline}
\left|\frak{f}_{1}\left(n\right)-\frak{f}_{2}\left(n\right)\right|^{2}\le\left\Vert \left(\begin{array}{c}
\frak{f}_{1}\left(n+1\right)\\
\mf f_{1}\left(n\right)
\end{array}\right)-\left(\begin{array}{c}
\mf f_{2}\left(n+1\right)\\
\mf f_{2}\left(n\right)
\end{array}\right)\right\Vert ^{2}\\
=\left\Vert M_{\left[m+1,n\right]}^{a}\left(E_{1}\right)\left(\begin{array}{c}
\mf f_{1}\left(m+1\right)\\
\mf f_{1}\left(m\right)
\end{array}\right)-M_{\left[m+1,n\right]}^{a}\left(E_{2}\right)\left(\begin{array}{c}
\mf f_{2}\left(m+1\right)\\
\mf f_{2}\left(m\right)
\end{array}\right)\right\Vert ^{2}\\
\lesssim\left\Vert \left(M_{\left[m+1,n\right]}^{a}\left(E_{1}\right)-M_{\left[m+1,n\right]}^{a}\left(E_{2}\right)\right)\left(\begin{array}{c}
\mf f_{1}\left(m+1\right)\\
\mf f_{1}\left(m\right)
\end{array}\right)\right\Vert ^{2}\\
+\left\Vert M_{\left[m+1,n\right]}^{a}\left(E_{2}\right)\left(\begin{array}{c}
\mf f_{1}\left(m+1\right)-\mf f_{2}\left(m+1\right)\\
\mf f_{1}\left(m\right)-\mf f_{2}\left(m\right)
\end{array}\right)\right\Vert ^{2}\\
\lesssim\exp\left(CQ_{N}\right)\left|E_{1}-E_{2}\right|^{2}\left(\left|\mf f_{1}\left(m+1\right)\right|^{2}+\left|\mf f_{1}\left(m\right)\right|^{2}\right)\\
+\exp\left(CQ_{N}\right)\left(\left|\mf f_{1}\left(m+1\right)-\mf f_{2}\left(m+1\right)\right|^{2}+\left|\mf f_{1}\left(m\right)-\mf f_{2}\left(m\right)\right|^{2}\right)\\
\lesssim\exp\left(CQ_{N}\right)\left|E_{1}-E_{2}\right|^{2}\left(\left|\mf f_{1}\left(m+1\right)\right|^{2}+\left|\mf f_{1}\left(m\right)\right|^{2}\right).\label{eq:separation-f1-f2-on-Lambda}
\end{multline}
For the second to last inequality we used \corref{prelims-lipschitzness},
\propref{prelims-M^a-upper-bound}, and the fact that $n-m\lesssim Q_{N}$
for $n\in\Lambda$. For the last inequality, in the case when $a>l+1$,
we used \lemref{separation-close-ev->close-ef} and the assumption
that $Q_{N}\gg l^{3}$. When $a\le l+1$ the last inequality holds
trivially since $\mf f_{i}\left(-1\right)=0$, $\mf f_{i}\left(0\right)=1$,
$i=1,2$.

Assume that $a>0$. We have that either $m,m+1\in\left[0,N-1\right]\setminus\Lambda$,
or $m=-1$, $ $$m+1\in\left[0,N-1\right]\setminus\Lambda$. Since
$\mf f_{1}\left(-1\right)=0$, using \eqref{separation-ef-localization-for-f}
and \eqref{separation-f1-f2-on-Lambda} we can conclude in either
case that
\begin{equation}
\sum_{n\in\Lambda}\left|\mf f_{1}\left(n\right)-\mf f_{2}\left(n\right)\right|^{2}\lesssim\exp\left(-CQ_{N}\right)\sum_{n\in\Lambda}\left|\mf f_{1}\left(n\right)\right|^{2}.\label{eq:separation-sum-f1-f2-over-Lambda}
\end{equation}
If $a=0$, then this follows trivially from \eqref{separation-f1-f2-on-Lambda}.
From \eqref{separation-sum-f1-f2-over-Lambda}, \eqref{separation-sum-f1-f2-over-outside-Lambda},
and the fact that $\mf f_{1}$ and $\mf f_{2}$ are orthogonal, we
get that
\[
\left\Vert \mf f_{1}-\mf f_{2}\right\Vert ^{2}=\sum_{n\in\left[0,N-1\right]}\left(\left|\mf f_{1}\left(n\right)\right|^{2}+\left|\mf f_{2}\left(n\right)\right|^{2}\right)\lesssim\exp\left(-CQ_{N}\right)\sum_{n\in\Lambda}\left(\left|\mf f_{1}\left(n\right)\right|^{2}+\left|\mf f_{2}\left(n\right)\right|^{2}\right).
\]
This is absurd, so we cannot have $\left|E_{1}-E_{2}\right|\le\exp\left(-C_{0}Q_{N}\right)$. 
\end{proof}
Next we use a bootstrapping argument to improve the separation from
the previous proposition.
\begin{thm}
\label{thm:separation-bootstrap} Suppose there exists $N'$, $2Q_{N}^{2}\le N'<N$,
such that $\exp\left(-C_{0}Q_{N'}\right)\ge\sigma_{N}$, with $C_{0}$
as in the previous proposition. Then for all $x\in\mb T$, $\omega\in\Omega^{0}\cap\mb T_{c,\alpha}\setminus\left(\Omega_{N}\cup\Omega_{N'}\right)$,
if $\dist\left(E_{j}^{\left(N\right)}\left(x,\omega\right),\mc E_{N,\omega}\cup\mc E_{N',\omega}\cup\left(\mathcal{E}^{0}\right)^{C}\right)\ge\sigma_{N}$
for some $j$, then
\[
\left|E_{j}^{\left(N\right)}\left(x,\omega\right)-E_{k}^{\left(N\right)}\left(x,\omega\right)\right|>\sigma_{N}/2
\]
for any $k\neq j$.\end{thm}
\begin{proof}
Fix $x\in\mathbb{T}$, $\omega\in\Omega^{0}\cap\mathbb{T}_{c,\alpha}\setminus\left(\Omega_{N}\cup\Omega_{N'}\right)$,
and $j$, such that $E_{1}=E_{j}^{\left(N\right)}\left(x,\omega\right)$
satisfies the assumptions. Suppose that there exists $E_{2}=E_{k}^{\left(N\right)}\left(x,\omega\right)\neq E_{1}$
such that $\left|E_{1}-E_{2}\right|\le\sigma_{N}/2$. We have that
$\dist\left(E_{k}^{\left(N\right)}\left(x,\omega\right),\mc E_{N,\omega}\cup\left(\mathcal{E}^{0}\right)^{C}\right)\ge\sigma_{N}/2\gg\exp\left(-l^{1/4}\right)$,
and due to \lemref{separation-close-ev->close-localization-centers}
it is possible to choose an interval $\Lambda\subset\left[0,N-1\right]$
of length $N'$ such that $\Lambda\supset\left[\nu_{i}^{\left(N\right)}\left(x,\omega\right)-Q_{N}^{2},\nu_{i}^{\left(N\right)}\left(x,\omega\right)+Q_{N}^{2}\right]$,
$i\in\left\{ j,k\right\} $. By \eqref{localization-ev} we know that
there exist $E'_{1},E'_{2}\in\spec\left(H_{\Lambda}\left(x,\omega\right)\right)$
such that $\left|E_{i}-E'_{i}\right|\lesssim\exp\left(-\gamma Q_{N}^{2}\right)$.
Note that $E'_{1}\neq E'_{2}$, since otherwise $\left|E_{1}-E_{2}\right|\lesssim\exp\left(-\gamma Q_{N}^{2}\right)$,
contradicting the conclusion of \propref{separation-raw}. We also
have that
\[
\dist\left(E'_{1},\mc E_{N',\omega}\cup\left(\mathcal{E}^{0}\right)^{C}\right)\gtrsim\sigma_{N}-\exp\left(-\sigma Q_{N}^{2}\right)\ge\sigma_{N}/2\gg\exp\left(-l^{1/4}\right)\ge\exp\left(-l'^{1/4}\right),
\]
where $l'=2\left[\left(\log N'\right)^{A}\right]$, with $A$ as in
the Elimination Assumption \ref{localization-elimination-assumption}.
Applying \propref{separation-raw} at scale $N'$ we get that $\left|E'_{1}-E'_{2}\right|>\exp\left(-C_{0}Q_{N'}\right)\ge\sigma_{N}$,
and consequently $\left|E_{1}-E_{2}\right|>\sigma_{N}-\exp\left(-Q_{N}^{2}\right)\ge\sigma_{N}/2$.
We arrived at a contradiction, and the proof is concluded.
\end{proof}

\section{\label{sec:Elimination-via-Resultants}Elimination, Localization,
and Separation via Resultants}

In this section we will first obtain the elimination of resonances
via resultants using the abstract results from \cite{MR2753606}.
Then we will apply the abstract results of the previous two sections
to get concrete localization and separation.

As was mentioned in the introduction, we first need to apply the Weierstrass
Preparation Theorem to the determinants. For convenience we recall
a version of the Weierstrass Preparation Theorem. In what follows
$f\left(z,w\right)$ is a function defined on the polydisk $\mc P=\mc D\left(z_{0},R_{0}\right)\times\mc P\left(w^{0},R_{0}\right)$,
$z_{0}\in\mb C$, $w^{0}\in\mb C^{d}$, $1/2\ge R_{0}>0$.
\begin{lem}
\label{lem:resultants-Weierstrass-preparation}(\cite[Proposition 2.26]{MR2753606})
Assume that $f\left(\cdot,w\right)$ has no zeros on some circle $\left|z-z_{0}\right|=r_{0}$,
$0<r_{0}<R_{0}/2$, for any $w\in\mc P_{1}=\mc P\left(w^{0},r_{1}\right)$
where $0<r_{1}<R_{0}$. Then there exist a polynomial $P\left(z,w\right)=z^{k}+a_{k-1}\left(w\right)z^{k-1}+\ldots+a_{0}\left(w\right)$
with $a_{j}\left(w\right)$ analytic in $\mc P_{1}$ and an analytic
function $g\left(z,w\right)$, $\left(z,w\right)\in\mc D\left(z_{0},r_{0}\right)\times\mc P_{1}$
so that the following statements hold:
\begin{enumerate}
\item $f\left(z,w\right)=P\left(z,w\right)g\left(z,w\right)$ for any $\left(z,w\right)\in\mc D\left(z_{0},r_{0}\right)\times\mc P_{1}$,
\item $g\left(z,w\right)\neq0$ for any $\left(z,w\right)\in\mc D\left(z_{0},r_{0}\right)\times\mc P_{1}$,
\item For any $w\in\mc P_{1}$, $P\left(\cdot,w\right)$ has no zeros in
$\mb C\setminus\mc D\left(z_{0},r_{0}\right)$.
\end{enumerate}
\end{lem}
We can now obtain the Weierstrass Preparation of the determinants.
\begin{prop}
\label{prop:resultants-Weierstrass-preparation-for-determinants}Given
$x_{0}\in\mb T$, $\left(\omega_{0},E_{0}\right)\in\mb T_{c,\alpha}\times\mb C$
such that $L\left(\omega_{0},E_{0}\right)>\gamma>0$, there exist
constants $N_{0}=N_{0}\left(\left\Vert a\right\Vert _{\infty},\left\Vert b\right\Vert _{*},\left|E_{0}\right|,c,\alpha,\gamma\right)$,
$C_{0}=C_{0}\left(\alpha\right)$, so that for any $N\ge N_{0}$ there
exist $r_{0}\simeq N^{-1}$, a polynomial $P_{N}\left(z,\omega,E\right)=z^{k}+a_{k-1}\left(\omega,E\right)z^{k-1}+\ldots+a_{0}\left(\omega,E\right)$,
with $a_{j}\left(\omega,E\right)$ analytic in $\mc D\left(E_{0},r_{1}\right)\times\mc D\left(\omega_{0},r_{1}\right)$,
$r_{1}=\exp\left(-\left(\log N\right)^{C_{0}}\right)$, and an analytic
function $g_{N}\left(z,\omega,E\right)$, $\left(z,\omega,E\right)\in\mc P:=\mc D\left(x_{0},r_{0}\right)\times\mc D\left(E_{0},r_{1}\right)\times\mc D\left(\omega_{0},r_{1}\right)$
such that:
\begin{enumerate}
\item $f_{N}^{a}\left(z,\omega,E\right)=P_{N}\left(z,\omega,E\right)g_{N}\left(z,\omega,E\right)$,
\item $g_{N}\left(z,\omega,E\right)\neq0$ for any $\left(z,\omega,E\right)\in\mc P$,
\item For any $\left(\omega,E\right)\in\mc D\left(\omega_{0},r_{1}\right)\times\mc D\left(E_{0},r_{1}\right)$
the polynomial $P_{N}\left(\cdot,\omega,E\right)$ has no zeros in
$\mb C\setminus\mc D\left(z_{0},r_{0}\right)$,
\item $k=\deg P_{N}\left(\cdot,\omega,E\right)\le\left(\log N\right)^{C_{0}}$.
\end{enumerate}
\end{prop}
\begin{proof}
Let $f\left(\zeta,w_{1},w_{2}\right):=f_{N}^{a}\left(x_{0}+N^{-1}\zeta,\omega_{0}+N^{-C}w_{1},E_{0}+N^{-C}w_{2}\right)$,
where $C$ is larger than the $C_{1}$ constants from \corref{prelims-uniform-upper-bound}
and \corref{prelims-lipschitzness}. By the large deviations estimate
for determinants (\propref{prelims-LDT-determinants}) it follows
that (for large enough $N$) there exists $\zeta_{0}$, $\left|\zeta_{0}\right|<1/100$,
such that $\left|f\left(\zeta_{0},0,0\right)\right|>NL_{N}\left(\omega_{0},E_{0}\right)-\left(\log N\right)^{C}$.
Using \corref{prelims-uniform-upper-bound} we can apply Cartan's
estimate (\lemref{prelims-Cartan-estimate}) to $\phi\left(\zeta\right)=f\left(\zeta,0,0\right)$
on $\mc D\left(\zeta_{0},1\right)$, to get that there exists $\mc B\in\car_{1}\left(\log N,\left(\log N\right)^{C}\right)$
such that
\begin{equation}
\left|f\left(\zeta,0,0\right)\right|>\exp\left(NL^{a}\left(\omega_{0},E_{0}\right)-\left(\log N\right)^{C}\right),\label{eq:resultants-f_zet_0_0>0}
\end{equation}
for $\zeta\in\mc D\left(\zeta_{0},1/6\right)\setminus\mc B$. In particular,
from \defnref{prelims-Caratheodory}, we can conclude there exists
$r\in(1/5,1/6)$ such that \eqref{resultants-f_zet_0_0>0} holds for
$\left|\zeta\right|=r$. Using \eqref{prelims-M-lipschitzness} we
have 
\begin{multline*}
\left|f\left(\zeta,w_{1},w_{2}\right)\right|\ge\left|f\left(\zeta,0,0\right)\right|-\left|f\left(\zeta,0,0\right)-f\left(\zeta,w_{1},w_{2}\right)\right|\\
\ge\exp\left(NL^{a}\left(\omega_{0},E_{0}\right)-\left(\log N\right)^{C}\right)\\
-\exp\left(NL^{a}\left(\omega_{0},E_{0}\right)+\left(\log N\right)^{C}\right)N^{-C}\left(\left|w_{1}\right|+\left|w_{2}\right|\right)>0,
\end{multline*}
for $\left|\zeta\right|=r$, $\left|w_{1}\right|,\left|w_{2}\right|\le\exp\left(-\left(\log N\right)^{C}\right)$.
Now the first three claims follow by applying \lemref{resultants-Weierstrass-preparation}
with $r_{0}=rN^{-1}$ and $r_{1}=\exp\left(-\left(\log N\right)^{C}\right)$.
The last claim is a consequence of \propref{prelims-number-of-ev}. 
\end{proof}
Next we recall the abstract version of the elimination via resultants
obtained by Goldstein and Schlag. Given $w^{0}\in\mb C^{d}$, $r=\left(r_{1},\ldots,r_{d}\right)$,
$r_{i}>0$, $i=1,\dots,d$, we let 
\[
S_{w_{0},r}\left(w\right)=\left(r_{1}^{-1}\left(w_{1}-w_{1}^{0}\right),\ldots,r_{d}^{-1}\left(w_{d}-w_{d}^{0}\right)\right).
\]
 We will use the notation $\mc Z\left(f\right)$ for the zeros of
a function $f$. We also let $\mc Z\left(f,S\right):=\mc Z\left(f\right)\cap S$
and $\mc Z\left(f,r\right):=\mc Z\left(f,\mb H_{r}\right)$. 
\begin{lem}
\label{lem:resultants-abstract-elimination}(\cite[Lemma 5.4]{MR2753606})
Let $P_{s}\left(z,w\right)=z^{k_{s}}+a_{s,k_{s}-1}\left(w\right)z^{k_{s}-1}+\ldots+a_{s,0}\left(w\right)$,
$z\in\mb C$, $s=1,2$, where $a_{s,j}\left(w\right)$ are analytic
functions defined on a polydisk $\mc{P=}\mc P\left(w^{0},r\right)$,
$w^{0}\in\mb C^{d}$. Assume that $k_{s}>0$, $s=1,2$, and set $k=k_{1}k_{2}$.
Suppose that for any $w\in\mc P$ the zeros of $P_{s}\left(\cdot,w\right)$
belong to the same disk $\mc D\left(z_{0},r_{0}\right)$, $r_{0}\ll1$,
$s=1,2$. Let $\left|t\right|>16kr_{0}r^{-1}$. Given $H\gg1$ there
exists a set 
\[
\mc B_{H,t}\subset\t{\mc P}:=\mc D\left(w_{1}^{0},8kr_{0}/\left|t\right|\right)\times\prod_{j=2}^{d}\mc D\left(w_{j}^{0},r/2\right)
\]
such that $S_{w^{0},\left(16kr_{0}\left|t\right|^{-1},r,\ldots,r\right)}\left(\mc B_{H,t}\right)\in\car_{d}\left(H^{1/d},K\right)$,
$K=CHk$ and for any $w\in\t{\mc P}\setminus\mc B_{H,t}$ one has
\[
\dist\left(\mc Z\left(P_{1}\left(\cdot,w\right)\right),\mc Z\left(P_{2}\left(\cdot+t\left(w_{1}-w_{1}^{0}\right),w\right)\right)\right)\ge e^{-CHk}.
\]

\end{lem}
We can now prove the elimination of resonances via resultants. This
is a generalization of \cite[Proposition 5.5]{MR2753606}.
\begin{prop}
\label{prop:resultants-concrete-elimination}There exist constants
$ $$l_{0}=l_{0}\left(\left\Vert a\right\Vert _{\infty},\left\Vert b\right\Vert _{*},c,\alpha,\gamma,E^{0}\right)$,
$c_{0}$, $C_{0}=C_{0}\left(\alpha\right)$ such that for any $l\ge l'\ge l_{0}$,
$t$ with $\left|t\right|\ge\exp\left(\left(\log l\right)^{C_{0}}\right)$,
and $H\gg1$, there exists a set $\Omega_{l,l',t,H}\subset\mb T,$
with
\[
\mes\left(\Omega_{l,l',t,H}\right)<\exp\left(\left(\log l\right)^{C_{0}}-\sqrt{H}\right),\,\com\left(\Omega_{l,l',t,H}\right)<\left|t\right|H\exp\left(\left(\log l\right)^{C_{0}}\right),
\]
such that for any $\omega\in\Omega^{0}\cap\mb T_{c,\alpha}\setminus\Omega_{l,l',t,H}$
there exists a set $\mc E_{l,l',t,H,\omega}$ with
\[
\mes\left(\mc E_{l,l',t,H,\omega}\right)<\left|t\right|\exp\left(\left(\log l\right)^{C_{0}}-\sqrt{H}\right),\,\com\left(\mc E_{l,l',t,H,\omega}\right)<\left|t\right|H\exp\left(\left(\log l\right)^{C_{0}}\right),
\]
such that:
\begin{enumerate}
\item For any $E\in\mathcal{E}^{0}\setminus\mc E_{l,l',t,H,\omega}$ we
have
\begin{equation}
\dist\left(\mc Z\left(f_{l}^{a}\left(\cdot,\omega,E\right),c_{0}l^{-1}\right),\mc Z\left(f_{l'}^{a}\left(\cdot+t\omega,\omega,E\right),c_{0}l^{-1}\right)\right)\ge\exp\left(-H\left(\log l\right)^{C_{0}}\right).\label{eq:separation-of-zeros}
\end{equation}

\item For any $x\in\mb T$ we have
\begin{multline}
\dist\left(\mathcal{E}^{0}\cap\spec\left(H^{\left(l\right)}\left(x,\omega\right)\right)\setminus\mc E_{l,l',t,H,\omega},\spec\left(H^{\left(l'\right)}\left(x+t\omega,\omega\right)\right)\right)\\
\ge\exp\left(-H\left(\log l\right)^{3C_{0}}\right).\label{eq:separation-of-spectra}
\end{multline}

\end{enumerate}
\end{prop}
\begin{proof}
Let $x_{0}\in\mb T$, $E_{0}\in\mathcal{E}^{0}$, and $\omega_{0}\in\Omega^{0}\cap\mb T_{c,\alpha}$.
Using \propref{resultants-Weierstrass-preparation-for-determinants}
we can write
\[
f_{l}^{a}\left(z,\omega,E\right)=P_{1}\left(z,\omega,E\right)g_{1}\left(z,\omega,E\right)
\]
and
\[
f_{l'}^{a}\left(z+t\omega_{0},\omega,E\right)=P_{2}\left(z,\omega,E\right)g_{2}\left(z,\omega,E\right),
\]
on $\mc P_{0}=\mc D\left(x_{0},r_{0}\right)\times\mc D\left(E_{0},r_{1}\right)\times\mc D\left(\omega_{0},r_{1}\right)$,
where $r_{0}\simeq l^{-1}$, $r_{1}=\exp\left(-\left(\log l\right)^{C}\right)$.
The functions $g_{i},$ $i=1,2$, don't vanish on $\mc P_{0}$, and
the polynomials $P_{i}$, $i=1,2$ are of degrees $k_{i}$, $i=1,2$,
$k_{i}\le\left(\log l\right)^{C}$. Applying \lemref{resultants-abstract-elimination}
to the polynomials $P_{1}\left(\cdot,\omega,E\right)$ and $P_{2}\left(\cdot+t\left(\omega-\omega_{0}\right),\omega,E\right)$,
with $\left|t\right|\ge\exp\left(\left(\log l\right)^{C}\right)>16k_{1}k_{2}r_{0}r_{1}^{-1}$,
yields that there exists $\mc B_{H,t}\subset\t{\mc P_{0}}:=\mc D\left(\omega_{0},8kr_{0}/\left|t\right|\right)\times\mc D\left(E_{0},r_{1}/2\right)$,
with
\begin{equation}
\left\{ \left(\frac{\left|t\right|\left(\omega-\omega_{0}\right)}{16kr_{0}},\frac{E-E_{0}}{r_{1}}\right):\,\left(\omega,E\right)\in\mc B_{H,t}\right\} \in\car_{2}\left(H^{1/2},H\left(\log l\right)^{C}\right),\label{eq:B-is-Car-2}
\end{equation}
so that for any $\left(\omega,E\right)\in\t{\mc P}_{0}\setminus\mc B_{H,t}$
we have
\[
\dist\left(\mc Z\left(P_{1}\left(\cdot,\omega,E\right)\right),\mc Z\left(P_{2}\left(\cdot+t\left(\omega-\omega_{0}\right),\omega,E\right)\right)\right)\ge e^{-H\left(\log l\right)^{C}},
\]
which implies that
\begin{equation}
\dist\left(\mc Z\left(f_{l}^{a}\left(\cdot,\omega,E\right),\mc D\left(x_{0},r_{0}\right)\right),\mc Z\left(f_{l'}^{a}\left(\cdot+t\omega,\omega,E\right),\mc D\left(x_{0},r_{0}\right)\right)\right)\ge e^{-H\left(\log l\right)^{C}}.\label{eq:local-separation-of-zeros}
\end{equation}

Let $\mc N_{x}$ be an $r_{0}/2$-net covering $\mb T$, such that
$\left\{ z:\,\left|\Im z\right|<c_{0}l^{-1}\right\} \subset\cup_{x\in\mc N_{x}}\mc D\left(x,r_{0}/2\right)$
(for this $c_{0}$ has to be small enough, depending on the absolute
constants in $r_{0}\simeq l^{-1}$). Let $\mc N_{\omega}$ be a $8kr_{0}/\left|t\right|$-net
covering $\Omega^{0}\cap\mb T_{c,\alpha}$, $\mc N_{E}$ a $r_{1}/2$-net
covering $\mathcal{E}^{0}$, and $\left\{ \left(x_{j},\omega_{j},E_{j}\right)\right\} _{j}=\mc N_{x}\times\mc N_{\omega}\times\mc N_{E}$.
Denote by $\mc B_{H,t,j}$ the bad set corresponding (as above) to
$\left(x_{j},\omega_{j},E_{j}\right)$. By \eqref{B-is-Car-2} and
\defnref{prelims-Caratheodory} we have that there exists $\Omega_{j}$,
with 
\[
\mes\left(\Omega_{j}\right)\le16kr_{0}\left|t\right|^{-1}\exp\left(-\sqrt{H}\right),\,\com\left(\Omega_{j}\right)\le H\left(\log l\right)^{C},
\]
 so that for each $\omega\in\mc D\left(\omega_{j},8kr_{0}/t\right)\setminus\Omega_{j}$
we have $\left(\mc B_{H,t,j}\right)_{\omega}^{\left(1\right)}=:\mc E_{j,\omega}$
is such that
\[
\mes\left(\mc E_{j,\omega}\right)\le r_{1}\exp\left(-\sqrt{H}\right),\,\com\left(\mc E_{j,\omega}\right)\le H\left(\log l\right)^{C}.
\]
We define $\Omega_{l,l',t,H}:=\cup_{j}\Omega_{j}$ and $\mc E_{l,l',t,H,\omega}:=\cup_{j}\mc E_{j,\omega}$,
for $\omega\in\Omega^{0}\cap\mb T_{c,\alpha}\setminus\Omega_{l,l',t,H}$.
The measure and complexity bounds for these sets are straightforward
to check. If \eqref{separation-of-zeros} fails, there would exist
$\omega\in\Omega^{0}\cap\mb T_{c,\alpha}\setminus\Omega_{l,l',t,H}$,
$E\in\mathcal{E}^{0}\setminus\mc E_{l,l',t,H,\omega}$, and $z_{1},z_{2}$,
$\left|\Im z_{1}\right|,\left|\Im z_{2}\right|<c_{0}l^{-1}$, $\left|z_{1}-z_{2}\right|<\exp\left(-H\left(\log l\right)^{C_{0}}\right)$
such that
\[
f_{l}^{a}\left(z_{1},\omega,E\right)=f_{l'}^{a}\left(z_{2}+t\omega,\omega,E\right)=0.
\]
By our choice of covering nets, we have that $\left(z_{1},\omega,E\right)\in\mc D\left(x_{j},r_{0}/2\right)\times\mc D\left(\omega_{j},8kr_{0}/\left|t\right|\right)\times\mc D\left(E_{j},r_{1}/2\right)$
for some $j$. Since $\left|z_{1}-z_{2}\right|<\exp\left(-H\left(\log l\right)^{C_{0}}\right)$,
we can conclude that we have $\left(z_{i},\omega,E\right)\in\mc D\left(x_{j},r_{0}\right)\times\mc D\left(\omega_{j},8kr_{0}/\left|t\right|\right)\times\mc D\left(E_{j},r_{1}/2\right)$,
which contradicts \eqref{local-separation-of-zeros}. This proves
\eqref{separation-of-zeros}.

If \eqref{separation-of-spectra} fails, there would exist $\omega\in\Omega^{0}\cap\mb T_{c,\alpha}\setminus\Omega_{l,l',t,H}$,
$E_{1}\in\mathcal{E}^{0}\setminus\mc E_{l,l',t,H,\omega}$, $E_{2}\in\mb C$,
$\left|E_{1}-E_{2}\right|<\exp\left(-H\left(\log l\right)^{3C_{0}}\right)$,
and $x\in\mb T$ such that
\[
f_{l}^{a}\left(x,\omega,E_{1}\right)=f_{l'}^{a}\left(x+t\omega,\omega,E_{2}\right)=0.
\]
By \corref{prelims-lipschitzness} we have
\begin{multline*}
\left|f_{l'}^{a}\left(x+t\omega,\omega,E_{1}\right)\right|=\left|f_{l'}^{a}\left(x+t\omega,\omega,E_{1}\right)-f_{l'}^{a}\left(x+t\omega,\omega,E_{2}\right)\right|\\
\le\left|E_{1}-E_{2}\right|\exp\left(l'L^{a}\left(\omega,E_{1}\right)+\left(\log l'\right)^{C}\right)\le\exp\left(l'L^{a}\left(\omega,E_{1}\right)-H\left(\log l\right)^{2C_{0}}\right).
\end{multline*}
By \propref{prelims-LDT-failure}, there exists $z$, $\left|z-x\right|\lesssim l'^{-1}\exp\left(-H\left(\log l\right)^{C_{0}}\right)$
such that 
\[
f_{l'}^{a}\left(z+t\omega,\omega,E_{1}\right)=0.
\]
This contradicts \eqref{separation-of-zeros}, and thus we proved
\eqref{separation-of-spectra}.
\end{proof}
Next we state the elimination of resonances as in the Elimination
Assumption \ref{localization-elimination-assumption}.
\begin{cor}
\label{cor:resultants-elimination} Fix $A>1$. There exist constants
$N_{0}=N_{0}\left(\left\Vert a\right\Vert _{\infty},\left\Vert b\right\Vert _{*},c,\alpha,\gamma,E^{0},A\right)$,
$C_{0}=C_{0}\left(\alpha\right)$, such that for any $N\ge N_{0}$
there exists a set $\Omega_{N}$, with
\[
\mes\left(\Omega_{N}\right)<\exp\left(-\left(\log N\right)^{2}\right),\,\com\left(\Omega_{N}\right)<N^{2}\exp\left(\left(\log\log N\right)^{C_{0}}\right),
\]
such that for any $\omega\in\Omega^{0}\cap\mb T_{c,\alpha}\setminus\Omega_{N}$
there exists a set $\mc E_{N,\omega}$, with
\[
\mes\left(\mc E_{N,\omega}\right)<\exp\left(-\left(\log N\right)^{2}\right),\,\com\left(\mc E_{N,\omega}\right)<N^{2}\exp\left(\left(\log\log N\right)^{C_{0}}\right),
\]
 such that for any $x\in\mb T$ and any integer $m$, $\exp\left(\left(\log\log N\right)^{C_{0}}\right)\le\left|m\right|\le N$,
we have
\[
\dist\left(\mathcal{E}^{0}\cap\spec\left(H^{\left(l_{1}\right)}\left(x,\omega\right)\setminus\mc E_{N,\omega}\right),\spec\left(H^{\left(l_{2}\right)}\left(x+m\omega,\omega\right)\right)\right)\ge\exp\left(-\left(\log N\right)^{6}\right),
\]
$l_{1},l_{2}\in\left\{ l,l+1,2l,2l+1\right\} $, where $l=2\left[\left(\log N\right)^{A}\right]$.\end{cor}
\begin{proof}
It is straightforward to see how this follows from \propref{resultants-concrete-elimination}
by letting $H=\left(\log N\right)^{5}$.
\end{proof}
We now have that the Elimination Assumption \ref{localization-elimination-assumption}
is satisfied with $A=A\left(\alpha\right)\gg1$, $Q_{N}=\exp\left(\left(\log\log N\right)^{C_{0}}\right)$,
$\sigma_{N}=\exp\left(-\left(\log N\right)^{6}\right)$, and $\Omega_{N}$,
$\mc E_{N,\omega}$ as in \corref{resultants-elimination}. The next
result follows immediately from \thmref{localization}.
\begin{prop}
\label{prop:resultants-localization}There exist constants $N_{0}=N_{0}\left(\left\Vert a\right\Vert _{\infty},\left\Vert b\right\Vert _{*},c,\alpha,\gamma,E^{0}\right)$,
$C_{0}=C_{0}\left(\alpha\right)$ , such that for any $N\ge N_{0}$
there exists a set $\Omega_{N}$, with
\[
\mes\left(\Omega_{N}\right)<\exp\left(-\left(\log N\right)^{2}\right),\,\com\left(\Omega_{N}\right)<N^{2}\exp\left(\left(\log\log N\right)^{C_{0}}\right),
\]
such that for any $\omega\in\Omega^{0}\cap\mb T_{c,\alpha}\setminus\Omega_{N}$
there exists a set $\t{\mc E}_{N,\omega}$, with
\[
\mes\left(\mc{\t E}_{N,\omega}\right)\lesssim\exp\left(-\left(\log N\right)^{2}\right),\,\com\left(\t{\mc E}_{N,\omega}\right)\lesssim N^{2}\exp\left(\left(\log\log N\right)^{C_{0}}\right),
\]
such that for any $x\in\mb T$, if $E_{j}^{\left(N\right)}\left(x,\omega\right)\in\mathcal{E}^{0}\setminus\t{\mc E}_{N,\omega}$,
for some $j$, then there exists a point $\nu_{j}^{\left(N\right)}\left(x,\omega\right)\in\left[0,N-1\right]$
so that for any $\Lambda=\left[a,b\right]$,
\[
\left[\nu_{j}^{\left(N\right)}\left(x,\omega\right)-3Q_{N},\nu_{j}^{\left(N\right)}\left(x,\omega\right)+3Q_{N}\right]\cap\left[0,N-1\right]\subset\Lambda\subset\left[0,N-1\right],
\]
$Q_{N}=\exp\left(\left(\log\log N\right)^{C_{0}}\right)$, if we let
$Q=\dist\left(\left[0,N-1\right]\setminus\Lambda,\nu_{j}^{\left(N\right)}\left(x,\omega\right)\right)$
we have:\end{prop}
\begin{enumerate}
\item 
\begin{equation}
\sum_{k\in\left[0,N-1\right]\setminus\Lambda}\left|\psi_{j}^{\left(N\right)}\left(x,\omega;k\right)\right|^{2}<\exp\left(-\gamma Q\right),\label{eq:localization-ef-1}
\end{equation}

\item 
\begin{equation}
\dist\left(E_{j}^{\left(N\right)}\left(x,\omega\right),\spec\left(H_{\Lambda}\left(x,\omega\right)\right)\right)\lesssim\exp\left(-\gamma Q\right).\label{eq:localization-ev-1}
\end{equation}

\end{enumerate}
The next result follows immediately from \propref{separation-raw}.
This is a generalization to the Jacobi case of \cite[Proposition 7.1]{MR2753606}.
\begin{prop}
\label{prop:resultants-separation-of-eigenvalues-scale-N}Let $\delta\in\left(0,1\right)$
and let $\Omega_{N}$, $\t{\mc E}_{N,\omega}$ be as in the previous
proposition. There exist constants $N_{0}=N_{0}\left(\left\Vert a\right\Vert _{\infty},\left\Vert b\right\Vert _{*},c,\alpha,\gamma,E^{0},\delta\right)$,
such that for $N\ge N_{0}$, $x\in\mb T$, $\omega\in\Omega^{0}\cap\mathbb{T}_{c,\alpha}\setminus\Omega_{N}$,
if $E_{j}^{\left(N\right)}\left(x,\omega\right)\in\mathcal{E}^{0}\setminus\t{\mc E}_{N,\omega}$
, for some $j$, then
\[
\left|E_{j}^{\left(N\right)}\left(x,\omega\right)-E_{k}^{\left(N\right)}\left(x,\omega\right)\right|>\exp\left(-N^{\delta}\right)
\]
for all $k\neq j$.
\end{prop}

\section{\label{sec:Abstract-Elimination}Abstract Elimination of Resonances
via Slopes}

In this section we will obtain elimination of resonances via slopes
(as discussed in the introduction) in an abstract setting. We begin
by presenting the assumptions under which we will be working. 

Let $e\left(x\right)=e^{2\pi ix}$. Let $P\left(x,y,z\right)$ be
a polynomial of degree at most $d_{1}$ for any fixed $x$, and of
degree at most $d_{2}$ for any fixed $y$. Let $f_{j}:\mb R^{2}\rightarrow\mb R$,
$j=1,\ldots,n$ be functions which are real-analytic and $1$-periodic
in each variable, and with the property that
\[
P\left(e\left(x\right),e\left(y\right),f_{j}\left(x,y\right)\right)=0,\,\left(x,y\right)\in\mb R^{2},\, j=1,\ldots,n.
\]
Clearly, there exist constants $C_{0}$ and $C_{1}$ such that 
\begin{equation}
\left|\partial_{x}f_{j}\left(x,y\right)\right|\le C_{0},\,\left|\partial_{y}f_{j}\left(x,y\right)\right|\le C_{1},\, x,y\in\mb R,\, j=1,\ldots,n.\label{eq:absslopes-hyp-bounded-derivatives}
\end{equation}
Equivalently we will have 
\begin{equation}
\left|f_{j}\left(x,y\right)-f_{j}\left(x',y'\right)\right|\le C_{0}\left|x-x'\right|+C_{1}\left|y-y'\right|,\, x,x',y,y'\in\mb R,\, j=1,\ldots,n.\label{eq:absslopes-hyp-lipschitz}
\end{equation}
Furthermore, we assume that there exist constants $c_{0}$, $r_{0}$,
$C_{2}$, $C_{3}$, a set $\mc Y^{0}\subset\left[0,1\right]$, and
an interval $\mc Z^{0}$, such that for every $y\in\mc Y^{0}$ there
exists a set $\mc Z_{y}^{0}$, with
\[
\mes\left(\mc Z_{y}^{0}\right)\le c_{0},\,\com\left(\mc Z_{y}^{0}\right)\le C_{2},
\]
such that for any $x\in\mb R$, if $f_{j}\left(x,y\right)\in\mc Z^{0}\setminus\mc Z_{y}^{0}$,
for some $j$, then 
\begin{equation}
\left|\partial_{x}f_{j}\left(x,y\right)-\partial_{x}f_{j}\left(x,y'\right)\right|\le C_{3}\left|y-y'\right|,\label{eq:absslopes-hyp-slopes}
\end{equation}
for any $y'\in\mb R$ such that $\left|y-y'\right|\le r_{0}$. The
rather convoluted form of the assumption is motivated by the concrete
estimate that we have for eigenvalues (see \corref{slopes-lipschitz-slopes-1}). 

By a Sard-type argument we show that for fixed $y$, after removing
some thin horizontal strips from the graphs of $f_{j}\left(\cdot,y\right)$
we have control over the slopes. Furthermore, these strips are stable
under small perturbations in $y$. We refer to \cite[Lemma 10.9-10]{MR2753606}
for similar considerations.
\begin{lem}
\label{lem:absslopes-abstract-slopes} Fix $\tau>0$ and let $\delta=\min\left\{ r_{0},\tau/C_{3},\tau/C_{1}\right\} $.
For each $y\in\mc Y^{0}$ there exists a set $\mc Z_{y}$, with
\begin{equation}
\mes\left(\mc Z_{y}\right)\lesssim\left(n+d_{2}^{2}+C_{2}\right)\tau+c_{0},\,\com\left(\mc Z_{y}\right)\lesssim d_{2}^{2}+C_{2},\label{eq:absslopes-Zy-complexity}
\end{equation}
such that for any $x\in\mb R$ and $y'\in\left(y-\delta,y+\delta\right)$,
if $f_{j}\left(x,y'\right)\in\mc Z^{0}\setminus\mc Z_{y}$, for some
$j$, then $\left|\partial_{x}f_{j}\left(x,y'\right)\right|>\tau$.\end{lem}
\begin{proof}
Fix $y\in\mc Y^{0}$. There exist, possibly degenerate, intervals
$I_{j,k}=I_{j,k}\left(y\right)\subset\left[0,1\right]$ such that
$\left|\partial_{x}f_{j}\left(x,y\right)\right|\le2\tau$ for $x\in\cup_{k}I_{j,k}$
and $\left|\partial_{x}f_{j}\left(x,y\right)\right|>2\tau$ for $x\in\left[0,1\right]\setminus\left(\cup_{k}I_{j,k}\right)$.
We let $Z_{j,k}=\left\{ f_{j}\left(x,y\right):\, x\in I_{j,k}\right\} $,
$Z_{y}=\cup_{j,k}Z_{j,k}$, and we define 
\[
\mc Z_{y}:=\left\{ z\in\mc Z^{0}:\,\dist\left(z,Z_{y}\cup\mc Z_{y}^{0}\cup\left(\mc Z^{0}\right)^{C}\right)\le\tau\right\} .
\]
Suppose that $f_{j}\left(x,y'\right)\in\mc Z^{0}\setminus\mc Z_{y}$,
for some $y'\in\left(y-\delta,y+\delta\right)$. By \eqref{absslopes-hyp-lipschitz}
and $\delta\le\tau/C_{1}$, it follows that $f_{j}\left(x,y\right)\in\mc Z^{0}\setminus\left(Z_{y}\cup\mc Z_{y}^{0}\right)$.
Hence $\left|\partial_{x}f_{j}\left(x,y\right)\right|>2\tau$, and
by \eqref{absslopes-hyp-slopes} and $\delta\le r_{0},\tau/C_{3}$,
it follows that $\left|\partial_{x}f_{j}\left(x,y'\right)\right|>\tau$,
as desired.

We clearly have that $\mes\left(Z_{j,k}\right)\le\tau\mes\left(I_{j,k}\right)$,
and hence $\mes\left(Z_{y}\right)\le n\tau$. At the same time we
have 
\[
\mes\left(\mc Z_{y}\right)\le\mes\left(Z_{y}\right)+\mes\left(\mc Z_{y}^{0}\right)+2\tau\left(\com\left(Z_{y}\right)+\com\left(\mc Z_{y}^{0}\right)+2\right),
\]
\[
\com\left(\mc Z_{y}\right)\le\com\left(Z_{y}\right)+\com\left(\mc Z_{y}^{0}\right)+2
\]
(recall that $\mc Z^{0}$ is an interval). So to get \eqref{absslopes-Zy-complexity}
we just need to estimate the number of intervals $I_{j,k}$. The number
of these intervals is controlled by the number of solutions of $\partial_{x}f_{j}\left(x,y\right)=\pm2\tau$,
$j=1,\ldots,n$ which is bounded by the number of solutions of the
system
\begin{align*}
0=Q_{1}\left(e\left(x\right),z\right): & =P\left(e\left(x\right),e\left(y\right),z\right)\\
0=Q_{2}\left(e\left(x\right),z\right): & =\partial_{1}P\left(e\left(x\right),e\left(y\right),z\right)2\pi ie\left(x\right)\pm2\tau\partial_{3}P\left(e\left(x\right),e\left(y\right),z\right).
\end{align*}
By B\'ezout's Theorem it follows that the number of solutions of
the above system is controlled by $d_{2}^{2}$. This concludes the
proof. 
\end{proof}
Let us make some remarks regarding the use of B\'ezout's Theorem
in the above lemma. To apply the theorem we would want $Q_{1}$ and
$Q_{2}$ to be irreducible and distinct. They are not necessarily
irreducible but we can replace them with some irreducible factors
by the following simple observation. Since $Q_{i}\left(e\left(x\right),f_{j}\left(x,y\right)\right)=0$
and $f_{j}$ is analytic, there must exist an irreducible factor $\t Q_{i}$
of $Q_{i}$ such that $\t Q_{i}\left(e\left(x\right),f_{j}\left(x,y\right)\right)=0$.
We can ensure that $\t Q_{1}$ and $\t Q_{2}$ are different by varying
$\tau$. Of course, for different functions $f_{j}$ we may get different
irreducible factors. It is elementary to argue that when we add up
the numbers of solutions from each combination of irreducible factors
we get a number less than the product of the degrees of $Q_{1}$ and
$Q_{2}$. In what follows, similar considerations apply whenever we
use B\'ezout's Theorem.

We can now obtain elimination of resonances.
\begin{thm}
\label{thm:absslopes-elimination}Let $\tau,\sigma>0$, $Q\ge\max\left\{ 4C_{1}/\tau,d_{1},n\left(d_{2}^{2}+C_{2}\right)\right\} $,
$M\ge Q$, $\delta\le\min\{r_{0},\tau/C_{3}$ $,\tau/C_{1}\}$, $\delta'\le\min\left\{ \sigma/\left(MC_{0}+2C_{1}\right),\delta/2\right\} $.
There exists $\mc Y\subset\left[0,1\right]$, with 
\begin{equation}
\mes\left(\mc Y\right)\lesssim\sqrt{M\sigma d_{2}d_{1}\tau^{-1}\delta^{-1}},\,\com\left(\mc Y\right)\lesssim\sqrt{M\sigma d_{2}d_{1}\tau^{-1}\delta^{-1}}/\delta',\label{eq:absslopes-Y-complexity}
\end{equation}
such that for each $y\in\mc Y^{0}\setminus\mc Y$ there exists $\mc{\t Z}_{y}$,
with 
\begin{multline*}
\mes\left(\t{\mc Z}_{y}\right)\lesssim n\tau+c_{0}+\left(d_{2}^{2}+C_{2}\right)\left(\tau+\sigma\right)+C_{0}\sqrt{M\sigma d_{2}d_{1}\tau^{-1}\delta^{-1}},\\
\,\com\left(\t{\mc Z}_{y}\right)\lesssim Md_{2}\left(d_{2}^{2}+C_{2}\right),
\end{multline*}
 such that for any $x\in\mb R$ we have that if $f_{j}\left(x,y\right)\in\mc Z^{0}\setminus\mc{\t Z}_{y}$,
for some $j$, then 
\begin{equation}
\left|f_{j}\left(x,y\right)-f_{k}\left(x+my,y\right)\right|\ge\sigma,\label{eq:absslopes-shift-separation}
\end{equation}
 for $k=1,\ldots,n$ and any integer $m$, $Q\le\left|m\right|\le M$.\end{thm}
\begin{proof}
Let $\left\{ y_{\alpha}\right\} $ be a $\delta$-net of points from
$\mc Y^{0}$ covering $\mc Y^{0}$. Also let $I_{y_{\alpha}}=\left(y_{\alpha}-\delta,y_{\alpha}+\delta\right)$,
and $\mc Z'_{y_{\alpha}}=\left\{ z\in\mc Z^{0}:\,\dist\left(\mc Z_{y_{\alpha}}\cup\left(\mc Z^{0}\right)^{C}\right)\le\sigma\right\} $,
where $\mc Z_{y_{\alpha}}$ is as in \lemref{absslopes-abstract-slopes}.
By \eqref{absslopes-Zy-complexity}, there exists a union of intervals
$Z{}_{y_{\alpha}}\supset\mc Z'_{y_{\alpha}}$ such that 
\[
\mes\left(Z{}_{y_{\alpha}}\right)\lesssim n\tau+c_{0}+\left(d_{2}^{2}+C_{2}\right)\left(\tau+\sigma\right),\,\com\left(Z{}_{y_{\alpha}}\right)\lesssim d_{2}^{2}+C_{2}.
\]

Let
\[
\mc B\left(y_{\alpha},j\right)=\left\{ \left(x,y\right)\in\left[0,1\right]\times I_{y_{\alpha}}:\, f_{j}\left(x,y\right)\in\left(\mc Z^{0}\right)^{C}\cup Z{}_{y_{\alpha}}\right\} .
\]
We define 
\[
g_{j,k,m}\left(x,y\right)=f_{k}\left(x+my,y\right)-f_{j}\left(x,y\right),
\]
 
\[
\mc B'_{m}\left(y_{\alpha},j,k\right)=\left\{ \left(x,y\right)\in\left(\left[0,1\right]\times I_{y_{\alpha}}\right)\setminus\mc B\left(y_{\alpha},j\right):\,\left|g_{j,k,m}\left(x,y\right)\right|<\sigma\right\} ,
\]
and $\mc B'_{m}\left(y_{\alpha}\right)=\cup_{j,k}\mc B'_{m}\left(y_{\alpha},j,k\right)$.
For $\left(x,y\right)\in\mc B'_{m}\left(y_{\alpha},j,k\right)$ we
have that $f_{j}\left(x,y\right)\in\mc Z^{0}\setminus Z{}_{y_{\alpha}}$
(due to the definition of $\mc B\left(y_{\alpha},j\right)$) and consequently
$f_{k}\left(x+my,y\right)\in\mc Z^{0}\setminus\mc Z_{y_{\alpha}}$
(due to the definitions of $Z{}_{y_{\alpha}}$ and $\mc Z'_{y_{\alpha}}$).
Hence, by \lemref{absslopes-abstract-slopes}, for $\left(x,y\right)\in\mc B'_{m}\left(y_{\alpha},j,k\right)$
we have $\left|\partial_{x}f_{k}\left(x+my,y\right)\right|>\tau$.
Since 
\[
\partial_{y}g_{j,k,m}\left(x,y\right)=m\partial_{x}f_{k}\left(x+my,y\right)+\partial_{y}f_{k}\left(x+my,y\right)-\partial_{y}f_{j}\left(x,y\right)
\]
 we can conclude that $\left|\partial_{y}g_{j,k,m}\left(x,y\right)\right|\ge\left|m\right|\tau-2C_{1}\ge\left|m\right|\tau/2$
for $\left(x,y\right)\in\mc B'_{m}\left(y_{\alpha},j,k\right)$ (we
used \eqref{absslopes-hyp-bounded-derivatives} and $\left|m\right|\ge Q\ge4C_{1}/\tau$).
Let
\[
\mc B''_{m}\left(y_{\alpha},j,k\right)=\left\{ \left(x,y\right)\in\left[0,1\right]\times I_{y_{\alpha}}:\left|g_{j,k,m}\left(x,y\right)\right|<\sigma\right\} .
\]
For a set $S\subset\mb R^{2}$ we will use the notation $S\vert_{x}:=\left\{ y:\,\left(x,y\right)\in S\right\} $,
$S\vert_{y}:=\left\{ x:\left(x,y\right)\in S\right\} $. We have that
$\mc B'_{m}\left(y_{\alpha},j,k\right)\vert_{x}=\mc B''\left(y_{\alpha},j,k\right)\vert_{x}\setminus\mc B\left(y_{\alpha},j\right)\vert_{x}$
is a union of, possibly degenerate, intervals. On each such interval
we have $\partial_{y}g_{j,k,m}\left(x,\cdot\right)\ge\left|m\right|\tau/2$
or $\partial_{y}g_{j,k,m}\left(x,\cdot\right)\le-\left|m\right|\tau/2$,
so by the fundamental theorem of calculus and the fact that on these
intervals we have $\left|g_{j,k,m}\left(x,\cdot\right)\right|<\sigma$,
each such interval must be of size smaller than $2\sigma\left(\left|m\right|\tau\right)^{-1}$.
Consequently we get

\begin{equation}
\mes\left(\mc B_{m}'\left(y_{\alpha}\right)\vert_{x}\right)\le2\sigma\left(\left|m\right|\tau\right)^{-1}\sum_{j,k}\com\left(\mc B'_{m}\left(y_{\alpha},j,k\right)\vert_{x}\right).\label{eq:absslopes-Fubini}
\end{equation}
At the same time we have 
\[
\com\left(\mc B'_{m}\left(y_{\alpha},j,k\right)\vert_{x}\right)\le\com\left(\mc B_{m}''\left(y_{\alpha},j,k\right)\vert_{x}\right)+\com\left(\mc B\left(y_{\alpha},j\right)\vert_{x}\right).
\]
The total number of components in $\mc B\left(y_{\alpha},j\right)\vert_{x}$
for all $j$ is controlled by the number of solutions of 
\[
\begin{cases}
f_{j}\left(x,y\right)=z\\
j\in\left\{ 1,\ldots,n\right\} \\
z\in\mathrm{E\left(\left(\mc Z^{0}\right)^{C}\cup Z{}_{y_{\alpha}}\right)}
\end{cases},
\]
where $E\left(\left(\mc Z^{0}\right)^{C}\cup Z{}_{y_{\alpha}}\right)$
is the set consisting of the endpoints of the intervals in $\left(\mc Z^{0}\right)^{C}\cup Z{}_{y_{\alpha}}$.
The number of solutions of this system is bounded by the number of
solutions of
\[
\begin{cases}
0=Q_{1}\left(e\left(y\right),z\right):=P\left(e\left(x\right),e\left(y\right),z\right)\\
0=Q_{2}\left(e\left(y\right),z\right):=z-z'\\
z'\in E\left(\left(\mc Z^{0}\right)^{C}\cup Z{}_{y_{\alpha}}\right)
\end{cases}.
\]
Using B\'ezout's theorem, we can conclude that
\[
\sum_{j}\com\left(\mc B\left(y_{\alpha},j\right)\right)\lesssim d_{1}\left(d_{2}^{2}+C_{2}\right).
\]
The total number of components in $\mc B''\left(y_{\alpha},j,k\right)\vert_{x}$
for all $j,k$, is controlled by the number of solutions of 
\[
\begin{cases}
g_{j,k,m}\left(x,y\right)=\pm\sigma\\
j,k\in\left\{ 1,\ldots,n\right\} 
\end{cases},
\]
which is bounded by the number of solutions of
\[
\begin{cases}
0=Q_{1}\left(e\left(y\right),z\right):=P\left(e\left(x\right),e\left(y\right),z\right)\\
0=Q_{2}\left(e\left(y\right),z\right):=e\left(\left|m\right|d_{2}y\right)P\left(e\left(x+my\right),e\left(y\right),z\pm\sigma\right)
\end{cases}.
\]
The $e\left(\left|m\right|d_{2}y\right)$ factor ensures that $Q_{2}$
is a polynomial, even when $m<0$. Since $\deg Q_{1}\le d_{1}$ and
$\deg Q_{2}\lesssim\left|m\right|d_{2}+d_{1}$, using B\'ezout's
theorem we can conclude that
\[
\sum_{j,k}\com\left(\mc B_{m}''\left(y_{\alpha},j,k\right)\vert_{x}\right)\lesssim\left|m\right|d_{2}d_{1}+d_{1}^{2}\lesssim\left|m\right|d_{2}d_{1}
\]
(we used $\left|m\right|\ge Q\ge d_{1}$). Now we can conclude that
\[
\sum_{j,k}\com\left(\mc B'_{m}\left(y_{\alpha},j,k\right)\vert_{x}\right)\lesssim\left|m\right|d_{2}d_{1}+nd_{1}\left(d_{2}^{2}+C_{2}\right)\lesssim\left|m\right|d_{2}d_{1}
\]
(we used $\left|m\right|\ge Q\ge n\left(d_{2}^{2}+C_{2}\right)$).
By \eqref{absslopes-Fubini} and Fubini's theorem we can now conclude
that $\mes\left(\mc B_{m}'\left(y_{\alpha}\right)\right)\lesssim\sigma d_{2}d_{1}\tau^{-1}.$

Let $\mbox{\ensuremath{\mc B}}'=\cup_{m,y_{\alpha}}\mc B'_{m}\left(y_{\alpha}\right)$.
We have that $\mes\left(\mc B'\right)\lesssim M\sigma d_{2}d_{1}\tau^{-1}\delta^{-1}$.
We define 
\[
\mc Y:=\left\{ y\in\mc Y^{0}:\,\mes\left(\mc B'\vert_{y}\right)>\sqrt{M\sigma d_{2}d_{1}\tau^{-1}\delta^{-1}}\right\} .
\]
From Chebyshev's inequality we get $\mes\left(\mc Y\right)\lesssim\sqrt{M\sigma d_{2}d_{1}\tau^{-1}\delta^{-1}}$.
Clearly, for $y\in\mc Y_{0}\setminus\mc Y$ we have 
\begin{equation}
\mes\left(\mc B'\vert_{y}\right)\le\sqrt{M\sigma d_{2}d_{1}\tau^{-1}\delta^{-1}}.\label{eq:absslopes-small-bad-set-in-x-for-fixed-y}
\end{equation}
Next we estimate the complexity of the set $\mc Y$. Cover $\mc Y$
by $\lesssim\sqrt{M\sigma d_{2}d_{1}\tau^{-1}\delta^{-1}}/\delta'$
intervals $Y_{i}$ of size $\delta'$, centered at points from $\mc Y$.
Suppose that $y_{i}$ is the center of $Y_{i}$. Since $y_{i}\in\mc Y$
we have that there exists $m$, $Q\le\left|m\right|\le M$, such that
\[
\left|f_{k}\left(x+my_{i},y_{i}\right)-f_{j}\left(x,y_{i}\right)\right|<\sigma,\, x\in\mc B'\vert_{y_{i}}.
\]
Due to \eqref{absslopes-hyp-lipschitz} we can conclude that
\[
\left|f_{k}\left(x+my,y\right)-f_{j}\left(x,y\right)\right|\le\left|f_{k}\left(x+my_{i},y_{i}\right)-f_{j}\left(x,y_{i}\right)\right|+\left(MC_{0}+2C_{1}\right)\left|y-y_{i}\right|<2\sigma,
\]
for $x\in\mc B'\vert_{y_{i}}$ , and $y\in Y_{i}$ (we used $\delta'\le\sigma/\left(MC_{0}+2C_{1}\right)$).
From this we get that
\[
\cup_{i}Y_{i}\subset\left\{ y\in\t{\mc Y}^{0}:\,\mes\left(\mc B'\left(2\sigma\right)\vert_{y}\right)>\sqrt{M\sigma d_{2}d_{1}\tau^{-1}\delta^{-1}}\right\} ,
\]
where $\mc B'\left(2\sigma\right)$ has the same definition as $\mc B'$,
only with $2\sigma$ instead of $\sigma$, and 
\[
\t{\mc Y}^{0}=\left\{ y\in\left[0,1\right]:\,\dist\left(y,\mc Y^{0}\right)\le\delta/2\right\} 
\]
 (we used $\delta'\le\delta/2$). Note that the $\delta$-net $\left\{ y_{\alpha}\right\} $
can be chosen so that it covers $\t{\mc Y}^{0}$, rather than just
$\mc Y^{0}$. By the same argument as above (that led to $\mes\left(\mc Y\right)\lesssim\sqrt{M\sigma d_{2}d_{1}\tau^{-1}\delta^{-1}}$)
we get that $\mes\left(\cup_{i}Y_{i}\right)\lesssim\sqrt{M\sigma d_{2}d_{1}\tau^{-1}\delta^{-1}}$,
and hence we have \eqref{absslopes-Y-complexity}.

Fix $y\in\mc Y^{0}\setminus\mc Y$ and let $y_{\alpha}$ be such that
$y\in I_{y_{\alpha}}$. Let 
\[
Z'_{y}=\cup_{m,j,k}\left\{ f_{j}\left(x,y\right):\,\left(x,y\right)\in\mc B'_{m}\left(y_{\alpha},j,k\right)\right\} ,
\]
and define $\t{\mc Z}_{y}:=Z_{y_{\alpha}}\cup Z'_{y}.$ We have that
\begin{multline*}
\mes\left(\t{\mc Z}_{y}\right)\le\mes\left(Z_{y_{\alpha}}\right)+\mes\left(Z'_{y}\right)\\
\lesssim n\tau+c_{0}+\left(d_{2}^{2}+C_{2}\right)\left(\tau+\sigma\right)+C_{0}\sqrt{M\sigma d_{2}d_{1}\tau^{-1}\delta^{-1}},
\end{multline*}
\begin{multline*}
\com\left(\t{\mc Z}_{y}\right)\le\com\left(Z_{y_{\alpha}}\right)+\com\left(Z_{y}'\right)\lesssim d_{2}^{2}+C_{2}+M\left(d_{2}^{2}+d_{2}\left(d_{2}^{2}+C_{2}\right)\right)\\
\lesssim Md_{2}\left(d_{2}^{2}+C_{2}\right).
\end{multline*}
To get the bound on $\mes\left(Z'_{y}\right)$ we used \eqref{absslopes-small-bad-set-in-x-for-fixed-y}
and \eqref{absslopes-hyp-lipschitz}. The estimate on $\com\left(Z'_{y}\right)$
is obtained by noticing that 
\[
\com\left(Z'_{y}\right)\le\sum_{j,k,m}\com\left(\mc B'_{m}\left(y_{\alpha},j,k\right)\vert_{y}\right),
\]
 and by using B\'ezout's theorem in the same way we did to estimate
\[
\sum_{j,k}\com\left(\mc B_{m}'\left(y_{\alpha},j,k\right)\vert_{x}\right).
\]

It is easy to see that with this choice of $\t{\mc Z}{}_{y}$ we have
that \eqref{absslopes-shift-separation} holds. Indeed, suppose that
$f_{j}\left(x,y\right)\in\mc Z^{0}\setminus\t{\mc Z}_{y}$ and suppose
that there exist $k,m$, such that
\[
\left|f_{j}\left(x,\omega\right)-f_{k}\left(x+m\omega,\omega\right)\right|<\sigma.
\]
This implies $\left(x,y\right)\in\mc B''_{m}\left(y_{\alpha},j,k\right)\subset\mc B'_{m}\left(y_{\alpha},j,k\right)\cup\mc B\left(y_{\alpha},j\right)$.
If $\left(x,y\right)\in\mc B'_{m}\left(y_{\alpha},j,k\right)$ then
$f_{j}\left(x,y\right)\in Z'_{y}$, and if $\left(x,y\right)\in\mc B\left(y_{\alpha},j\right)$
then $f_{j}\left(x,y\right)\in Z_{y_{\alpha}}\cup\left(\mc Z^{0}\right)^{C}$.
Either way, we arrived at a contradiction. This concludes the proof.
\end{proof}

\section{\label{sec:Elimination-via-Slopes}Elimination of Resonances and
Separation of Eigenvalues via Slopes}

In this section we apply \thmref{absslopes-elimination} to our concrete
setting to obtain a sharper elimination of resonances, based on which
we will obtain our main result (by applying \thmref{separation-bootstrap}).
As was mentioned in the introduction, to get the stability of slopes
needed for \thmref{absslopes-elimination} we will at first use the
``a priori'' separation via resultants (\propref{resultants-separation-of-eigenvalues-scale-N}).
This will yield a better separation, but still weaker than the one
we desire (see \propref{slopes-separation-1}). By using the improved
separation to get better stability of slopes and then repeating our
steps we will obtain the desired separation. 

We proceed by setting things up for the use of \thmref{absslopes-elimination}.
First, we need to approximate $a$ and $b$ by trigonometric polynomials,
so that the eigenvalues will be algebraic. Let
\[
a\left(x\right)=\sum_{n=-\infty}^{\infty}a_{n}e\left(nx\right),
\]
and
\[
b\left(x\right)=\sum_{n=-\infty}^{\infty}b_{n}e\left(nx\right),
\]
be the Fourier series expansions for $a$ and $b$ (recall that $e\left(x\right)=\exp\left(2\pi ix\right)$).
It is known that there exist constant $C=C\left(\left\Vert a\right\Vert _{\infty},\left\Vert b\right\Vert _{\infty}\right)$
and $c=c\left(\rho_{0}\right)$ such that
\begin{equation}
\left|a_{n}\right|,\left|b_{n}\right|\le C\exp\left(-\pi\rho_{0}\left|n\right|\right),\, n\in\mb Z,\label{eq:slopes-Fourier-coef-bound}
\end{equation}
with $C=\sup_{x\in\mb T}\left(\left|a\left(x\pm i\rho_{0}/2\right)\right|+\left|b\left(x\pm i\rho_{0}/2\right)\right|\right)$.
Let 
\[
a_{K}\left(x\right)=\sum_{n=-K}^{K}a_{n}e\left(nx\right),
\]
and 
\[
b_{K}\left(x\right)=\sum_{n=-K}^{K}b_{n}e\left(nx\right).
\]
By \eqref{slopes-Fourier-coef-bound}, there exists $C=C\left(\left\Vert a\right\Vert _{\infty},\left\Vert b\right\Vert _{\infty},\rho_{0}\right)$
such that

\begin{equation}
\sup_{\left|\Im z\right|<\rho_{0}/3}\left|a\left(z\right)-a_{K}\left(z\right)\right|+\left|b\left(z\right)-b_{K}\left(z\right)\right|\le C\exp\left(-\pi\rho_{0}K/3\right).\label{eq:slopes-V-VK}
\end{equation}
Let $H_{K}^{\left(l\right)}\left(x,\omega\right)$ denote the matrix,
at scale $l$, associated with $a_{K},b_{K}$, and let $E_{K,j}^{\left(l\right)}\left(x,\omega\right)$
be its eigenvalues. As a consequence of \eqref{slopes-V-VK} we get
\[
\sup_{x,\omega\in\mb T}\left\Vert H^{\left(l\right)}\left(x,\omega\right)-H_{K}^{\left(l\right)}\left(x,\omega\right)\right\Vert \lesssim C\exp\left(-cK\right),
\]
and, since the matrices are Hermitian for $x,\omega\in\mb T$, we
also have
\begin{equation}
\sup_{x,\omega\in\mb T}\left|E_{j}^{\left(l\right)}\left(x,\omega\right)-E_{K,j}^{\left(l\right)}\left(x,\omega\right)\right|\lesssim C\exp\left(-cK\right).\label{eq:slopes-E-EK}
\end{equation}

It is easy to see that there exists a constant $C=C\left(\left\Vert a\right\Vert _{\infty},\left\Vert b\right\Vert _{*},\rho_{0}\right)$
such that
\[
\left\Vert H^{\left(l\right)}\left(z,w\right)-H^{\left(l\right)}\left(z',w'\right)\right\Vert \le C\left(\left|z-z'\right|+l\left|w-w'\right|\right),
\]
for any $z,z'\in\mb H_{\rho_{0}/3}$ and $w,w'\in l^{-1}\mb H_{\rho_{0}/3}$.
Furthermore, due to \eqref{slopes-V-VK}, it can be seen that there
exists a constant $C=C\left(\left\Vert a\right\Vert _{\infty},\left\Vert b\right\Vert _{*},\rho_{0}\right)$
such that for any $K$ we have 

\begin{equation}
\left\Vert H_{K}^{\left(l\right)}\left(z,w\right)-H_{K}^{\left(l\right)}\left(z',w'\right)\right\Vert \le C\left(\left|z-z'\right|+l\left|w-w'\right|\right),\label{eq:slopes-HK-Lipschitz}
\end{equation}
for any $z,z'\in\mb H_{\rho_{0}/3}$ and $w,w'\in l^{-1}\mb H_{\rho_{0}/3}$.
In particular, since $H_{K}^{\left(l\right)}\left(x,\omega\right)$
is Hermitian for $x,\omega\in\mb T$, we have that
\begin{equation}
\left|E_{K,j}^{\left(l\right)}\left(x,\omega\right)-E_{K,j}^{\left(l\right)}\left(x',\omega'\right)\right|\le C\left(\left|x-x'\right|+l\left|\omega-\omega'\right|\right),\label{eq:slopes-EK-lipschitzness}
\end{equation}
for any $x,\omega\in\mb T$. This will give us the values of the constants
in \eqref{absslopes-hyp-lipschitz}. To get the constants related
to \eqref{absslopes-hyp-slopes} we will use the following lemma.
\begin{lem}
\label{lem:slopes-lipschitz-derivatives}Fix $x,\omega\in\mb T$,
$j\in\left\{ 1,\ldots,l\right\} $, and suppose that $\left|E_{j}^{\left(l\right)}\left(x,\omega\right)-E_{i}^{\left(l\right)}\left(x,\omega\right)\right|\ge\sigma$,
for all $i\neq j$. Furthermore, suppose that $K$ is large enough
so that 
\[
\left|E_{i}^{\left(l\right)}\left(x,\omega\right)-E_{K,i}^{\left(l\right)}\left(x,\omega\right)\right|\le\sigma/2
\]
 for all $i$. There exists a constant $C_{0}=C_{0}\left(\left\Vert a\right\Vert _{\infty},\left\Vert b\right\Vert _{*},\rho_{0}\right)$
such that
\[
\left|\partial_{x}E_{K,j}^{\left(l\right)}\left(x,\omega\right)-\partial_{x}E_{K,j}^{\left(l\right)}\left(x,\omega'\right)\right|\le C_{0}l\sigma^{-1}\left|\omega-\omega'\right|,
\]
for any $\omega'\in\mb R$ such that $\left|\omega-\omega'\right|\le C_{0}^{-1}l^{-1}\sigma$. \end{lem}
\begin{proof}
We clearly have that 
\[
\left|E_{K,j}^{\left(l\right)}\left(x,\omega\right)-E_{K,i}^{\left(l\right)}\left(x,\omega\right)\right|\ge\sigma/2.
\]
From \eqref{slopes-HK-Lipschitz} and standard perturbation theory
it follows that 
\begin{equation}
\left|E_{K,j}^{\left(l\right)}\left(z,w\right)-E_{K,i}^{\left(l\right)}\left(z,w\right)\right|\ge\sigma/4,\label{eq:slopes-EjK-EiK}
\end{equation}
for any $i\neq j$ and $\left(z,w\right)\in\mc D\left(x,c\sigma\right)\times\mc D\left(\omega,c\sigma/l\right)=:\mc P$.
We can choose $c=c(\left\Vert a\right\Vert _{\infty},\left\Vert b\right\Vert _{*},$
$\rho_{0})$ small enough so that we also have
\begin{equation}
\left\Vert H_{K}^{\left(l\right)}\left(z,w\right)-H_{K}^{\left(l\right)}\left(x,\omega\right)\right\Vert \le C\left(\left|z-x\right|+l\left|w-\omega\right|\right)\le\sigma/8,\label{eq:slopes-HKzw-HKxo}
\end{equation}
for any $\left(z,w\right)\in\mc P$. Since $E_{K,j}^{\left(l\right)}$
is simple on $\mc P$ it follows from the implicit function theorem
that it is analytic on $\mc P$. From \eqref{slopes-HKzw-HKxo} it
follows that given $\left(z,w\right)\in\mc P$ we have 
\[
\left|E_{K,j}^{\left(l\right)}\left(z,w\right)-E_{K,j'}^{\left(l\right)}\left(x,\omega\right)\right|\le C\left(\left|z-x\right|+l\left|w-\omega\right|\right)\le\sigma/8,
\]
for some $j'=j'\left(z,w\right)$. Due to \eqref{slopes-EjK-EiK}
and the continuity of $E_{K,j}^{\left(l\right)}$ it follows that
in fact for $\left(z,w\right)\in\mc P$ we have 
\[
\left|E_{K,j}^{\left(l\right)}\left(z,w\right)-E_{K,j}^{\left(l\right)}\left(x,\omega\right)\right|\le C\left(\left|z-x\right|+l\left|w-\omega\right|\right).
\]
This estimate and Cauchy's formula yield the desired conclusion.\end{proof}
\begin{cor}
\label{cor:slopes-lipschitz-slopes-1}Fix $A>1$ and let $l=2\left[\left(\log N\right)^{A}\right]$.
There exists a constant $N_{0}=N_{0}(\left\Vert a\right\Vert _{\infty}$
$,\left\Vert b\right\Vert _{*},c,\alpha,\gamma,E^{0},A)$ such that
for $N\ge N_{0}$ there exists $\Omega_{l}$, with
\[
\mes\left(\Omega_{l}\right)<\exp\left(-\left(\log\log N\right)^{2}\right),\,\com\left(\Omega_{l}\right)<\left(\log N\right)^{2A+1},
\]
such that for any $\omega\in\Omega^{0}\cap\mb T_{c,\alpha}\setminus\Omega_{l}$
there exists a set $\mc E_{l,\omega}$, with 
\[
\mes\left(\mc E_{l,\omega}\right)<\exp\left(-\left(\log\log N\right)^{2}\right),\,\com\left(\mc E_{l,\omega}\right)<\left(\log N\right)^{2A+1},
\]
such that for any $x\in\mb T$, $K\ge\left(\log N\right)^{1/2}$,
if $E_{K,j}^{\left(l\right)}\left(x,\omega\right)\in\mathcal{E}^{0}\setminus\mc E_{l,\omega}$,
, for some $j$, then
\[
\left|\partial_{x}E_{K,j}^{\left(l\right)}\left(x,\omega\right)-\partial_{x}E_{K,j}^{\left(l\right)}\left(x,\omega\right)\right|\le\exp\left(\left(\log N\right)^{1/2}\right)\left|\omega-\omega'\right|,
\]
for any $\omega'\in\mb R$ such that $\left|\omega-\omega'\right|\le\exp\left(-\left(\log N\right)^{1/2}\right)$.\end{cor}
\begin{proof}
The result follows immediately from \lemref{slopes-lipschitz-derivatives}
and \propref{resultants-separation-of-eigenvalues-scale-N} with $\delta=1/3A$.
\end{proof}
Let $f_{K,l}\left(x,\omega,E\right)=\det\left[H_{K}^{\left(l\right)}\left(x,\omega\right)-E\right]$.
It is straightforward to see that
\begin{multline}
P\left(e\left(x\right),e\left(\omega\right),E\right):=e\left(20Klx\right)e\left(20Kl^{2}\omega\right)f_{K,l}\left(x,\omega,E\right)f_{K,l+1}\left(x,\omega,E\right)\\
\cdot f_{K,2l}\left(x,\omega,E\right)f_{K,2l+1}\left(x,\omega,E\right)\label{eq:slopes-definition-of-P}
\end{multline}
is a polynomial of degree $\lesssim Kl^{2}$ when the first variable
is fixed, and of degree $\lesssim Kl$ when the second variable is
fixed. Let $K_{0}=C\left[\log N\right]$, where $C=C\left(\left\Vert a\right\Vert _{\infty},\left\Vert b\right\Vert _{*},\rho_{0}\right)$
is chosen such that
\begin{equation}
\sup_{x,\omega\in\mb T}\left|E_{j}^{\left(l\right)}\left(x,\omega\right)-E_{K_{0},j}^{\left(l\right)}\left(x,\omega\right)\right|\le\frac{1}{N^{2}}\label{eq:slopes-K0}
\end{equation}
(we used \eqref{slopes-E-EK}) .

We can now apply \thmref{absslopes-elimination}.
\begin{prop}
\label{prop:slopes-elimination-1}Fix $A>1$, $p\in\left(1,2\right)$,
and let $l=2\left[\left(\log N\right)^{A}\right]$. There exists a
constant $N_{0}=N_{0}\left(\left\Vert a\right\Vert _{\infty},\left\Vert b\right\Vert _{*},c,\alpha,\gamma,E^{0},A,p\right)$,
such that for any $N\ge N_{0}$ there exists a set $\Omega_{N}$,
with
\[
\mes\left(\Omega_{N}\right)<\exp\left(-\left(\log\log N\right)^{2}/2\right),\,\com\left(\Omega_{N}\right)<N^{1+p},
\]
such that for any $\omega\in\Omega^{0}\cap\mb T_{c,\alpha}\setminus\Omega_{N}$
there exists a set $\mc E_{N,\omega}$, with
\[
\mes\left(\mc E_{N,\omega}\right)<\left(\log N\right)^{-A},\,\com\left(\mc E_{N,\omega}\right)<N\left(\log N\right)^{4A},
\]
 such that for any $x\in\mb T$ and any integer $m$, $\left(\log N\right)^{6A}\le\left|m\right|\le N$,
we have
\[
\dist\left(\mathcal{E}^{0}\cap\spec\left(H^{\left(l_{1}\right)}\left(x,\omega\right)\setminus\mc E_{N,\omega}\right),\spec\left(H^{\left(l_{2}\right)}\left(x+m\omega,\omega\right)\right)\right)\ge\frac{2}{N^{p}},
\]
$l_{1},l_{2}\in\left\{ l,l+1,2l,2l+1\right\} $.\end{prop}
\begin{proof}
We begin by identifying all the parameters used in \secref{Abstract-Elimination}.
The polynomial $P$ is given by \eqref{slopes-definition-of-P}, and
we can take $d_{1}=C\left(\log N\right)^{2A+1}$, $d_{2}=C\left(\log N\right)^{A+1}$,
with $C=C\left(\left\Vert a\right\Vert _{\infty},\left\Vert b\right\Vert _{*},\rho_{0}\right)$.
We have $\left\{ f_{j}\right\} =\left\{ E_{K_{0},i}^{\left(l'\right)}:\, i\in\left\{ 1,\ldots,l'\right\} ,l'\in\left\{ l,l+1,2l,2l+1\right\} \right\} $,
and $n=6l+2$. By \eqref{slopes-EK-lipschitzness} we can choose $C_{0}=C$,
$C_{1}=C\left(\log N\right)^{A}$, $C=C\left(\left\Vert a\right\Vert _{\infty},\left\Vert b\right\Vert _{*},\rho_{0}\right)$.
Let $\Omega_{l}$, $\mc E_{l,\omega}$ be as in \corref{slopes-lipschitz-slopes-1}.
By \corref{slopes-lipschitz-slopes-1} we can choose $\mc Y^{0}=\Omega^{0}\cap\mb T_{c,\alpha}\setminus\Omega_{l}$,
$\mc Z^{0}=\mathcal{E}^{0}$, $\mc Z_{y}^{0}=\mc E_{l,\omega}$, $c_{0}=\exp\left(-\left(\log\log N\right)^{2}\right)$,
$C_{2}=\left(\log N\right)^{2A+1}$, $C_{3}=\exp\left(\left(\log N\right)^{1/2}\right)$,
$r_{0}=\exp\left(-\left(\log N\right)^{1/2}\right)$. 

Next we apply \thmref{absslopes-elimination} with $\tau=\left(\log N\right)^{-5A}$,
$\sigma=4N^{-p}$, $Q=\left(\log N\right)^{6A}$, $M=N$, $\delta=\exp\left(-\left(\log N\right)^{2/3}\right)$,
$\delta'=cN^{-\left(1+p\right)}$, $c=c\left(\left\Vert a\right\Vert _{\infty},\left\Vert b\right\Vert _{*},\rho_{0}\right)$.
Let $\Omega_{K_{0},l}=\Omega_{l}\cup\mc Y$ and $\mc E_{K_{0},l,\omega}=\t{\mc Z}_{y}$.
We have 
\[
\mes\left(\Omega_{K_{0},l}\right)<\exp\left(-\left(\log\log N\right)^{2}/2\right),\,\com\left(\Omega_{K_{0},l}\right)<N^{1+p},
\]
\[
\mes\left(\mc E_{K_{0},l,\omega}\right)<\left(\log N\right)^{-2A},\,\com\left(\mc E_{K_{0},l,\omega}\right)<N\left(\log N\right)^{4A},
\]
and 
\begin{equation}
\dist\left(\mathcal{E}^{0}\cap\spec\left(H_{K_{0}}^{\left(l_{1}\right)}\left(x,\omega\right)\setminus\mc E_{K_{0},l,\omega}\right),\spec\left(H_{K_{0}}^{\left(l_{2}\right)}\left(x+m\omega,\omega\right)\right)\right)\ge\frac{4}{N^{p}},\label{eq:slopes-elimination-K-1}
\end{equation}
$l_{1},l_{2}\in\left\{ l,l+1,2l,2l+1\right\} $, for any $\omega\in\Omega^{0}\cap\mb T_{c,\alpha}\setminus\Omega_{K_{0},l}$,
any $x\in\mb T$, and any integer $m$, $\left(\log N\right)^{6A}\le\left|m\right|\le N$.

Let $\Omega_{N}=\Omega_{K_{0},l}$, and $\mc E_{N,\omega}=\left\{ E\in\mathcal{E}^{0}:\,\dist\left(E,\mc E_{K_{0},l,\omega}\cup\left(\mathcal{E}^{0}\right)^{C}\right)\le N^{-2}\right\} $.
The measure and complexity bounds for $\Omega_{N}$ and $\mc E_{N,\omega}$
are clearly satisfied. Fix $\omega\in\Omega^{0}\cap\mb T_{c,\alpha}\setminus\Omega_{N}$
and $x\in\mb T$, and suppose $E_{j}^{\left(l_{1}\right)}\left(x,\omega\right)\in\mathcal{E}^{0}\setminus\mc E_{N,\omega}$,
for some $j$ and $l_{1}\in\left\{ l,l+1,2l,2l+1\right\} $. By \eqref{slopes-K0}
it follows that $E_{K_{0},j}^{\left(l_{1}\right)}\in\mathcal{E}^{0}\setminus\mc E_{K_{0},l,\omega}$.
Hence the conclusion follows from \eqref{slopes-elimination-K-1}
and \eqref{slopes-K0}.
\end{proof}
We can now improve the separation of eigenvalues at scale $N$ by
applying \thmref{separation-bootstrap}.
\begin{prop}
\label{prop:slopes-separation-1}Fix $p\in\left(1,2\right)$. There
exist constants $N_{0}=N_{0}(\left\Vert a\right\Vert _{\infty},\left\Vert b\right\Vert _{*},c,\alpha,\gamma,E^{0},$
$p)$, $C_{0}=C_{0}\left(\alpha\right)$ such that for any $N\ge N_{0}$
there exists a set $\t{\Omega}_{N}$, with
\[
\mes\left(\t{\Omega}_{N}\right)\le\exp\left(-\left(\log\log N\right)^{2}/4\right),\,\com\left(\t{\Omega}_{N}\right)\lesssim N^{1+p},
\]
such that for any $\omega\in\Omega^{0}\cap\mb T_{c,\alpha}\setminus\t{\Omega}_{N}$
there exists a set $\t{\mc E}_{N,\omega}$, with
\[
\mes\left(\t{\mc E}_{N,\omega}\right)\le\left(\log N\right)^{-1/10},\,\com\left(\t{\mc E}_{N,\omega}\right)\le N\left(\log N\right)^{C_{0}},
\]
 such that for any $x\in\mb T$, if $E_{j}^{\left(N\right)}\left(x,\omega\right)\in\mathcal{E}^{0}\setminus\t{\mc E}_{N,\omega}$,
for some $j$, then
\[
\left|E_{j}^{\left(N\right)}\left(x,\omega\right)-E_{k}^{\left(N\right)}\left(x,\omega\right)\right|>\frac{1}{N^{p}},
\]
for any $k\neq j$.\end{prop}
\begin{proof}
We start by identifying the parameters from the Elimination Assumption
\ref{localization-elimination-assumption}. Apply \propref{slopes-elimination-1}
with $A=A\left(\alpha\right)$ as in the Elimination Assumption \ref{localization-elimination-assumption}.
Now we can choose $\Omega_{N}$, $\mc E_{N,\omega}$ as in \propref{slopes-elimination-1}
and we also have $Q_{N}=\left(\log N\right)^{6A}$, $\sigma_{N}=2N^{-p}$. 

Next we apply \thmref{separation-bootstrap} with $N'=\left[\exp\left(\left(\log N\right)^{1/7A}\right)\right]$.
The conclusion follows by setting $\t{\Omega}_{N}=\Omega_{N}\cup\Omega_{N'}$,
and 
\[
\t{\mc E}_{N,\omega}=\left\{ E\in\mathcal{E}^{0}:\dist\left(\mc E_{N,\omega}\cup\mc E_{N',\omega}\cup\left(\mathcal{E}^{0}\right)^{C}\right)<2N^{-p}\right\} .
\]
 $ $
\end{proof}
We can now repeat our steps, starting with \corref{slopes-lipschitz-slopes-1}
to obtain a better separation. This time we will eliminate the resonances
at scale $l=100\left[\left(\log N\right)/\gamma\right]$ and then
use localization to eliminate the resonances at the scale $l=2\left[\left(\log N\right)^{A}\right]$,
as needed to apply \thmref{separation-bootstrap}. Working at a scale
$l=C\left[\log N\right]$ is needed to get separation by $\left(N\left(\log N\right)^{p}\right)^{-1}$
with $p$ as small as possible. We need to have $C=C\left(\gamma\right)$
in order to be able to apply localization.
\begin{lem}
\label{lem:slopes-lipschitz-slopes-2} Fix $p\in\left(1,2\right)$
and let $l=100\left[\left(\log N\right)/\gamma\right]$. There exist
constants $C_{0}=C_{0}\left(\alpha\right)$, $N_{0}=N_{0}(\left\Vert a\right\Vert _{\infty},\left\Vert b\right\Vert _{*},c,\alpha,\gamma,\mathcal{E}^{0},p)$,
such that for $N\ge N_{0}$ there exists $\Omega_{l}$, with
\[
\mes\left(\Omega_{l}\right)\le\exp\left(-\left(\log\log\log N\right)^{2}/8\right),\,\com\left(\Omega_{l}\right)\lesssim\left(\log N/\gamma\right)^{1+p},
\]
such that for any $\omega\in\Omega^{0}\cap\mb T_{c,\alpha}\setminus\Omega_{l}$
there exists a set $\mc E_{l,\omega}$, with 
\[
\mes\left(\mc E_{l,\omega}\right)\lesssim\left(\log\log N\right)^{-1/10},\,\com\left(\mc E_{l,\omega}\right)\le\log N\left(\log\log N\right)^{C_{0}},
\]
such that for any $x\in\mb T$, if $E_{K_{0},j}^{\left(l\right)}\left(x,\omega\right)\in\mathcal{E}^{0}\setminus\mc E_{l,\omega}$,
for some $j$, then
\[
\left|\partial_{x}E_{K_{0},j}^{\left(l\right)}\left(x,\omega\right)-\partial_{x}E_{K_{0},j}^{\left(l\right)}\left(x,\omega\right)\right|\lesssim\left(\log N\right)^{1+p}\gamma^{-1}\left|\omega-\omega'\right|,
\]
for any $\omega'\in\mb R$ such that $\left|\omega-\omega'\right|\le\left(\log N\right)^{-\left(1+p\right)}\gamma$.\end{lem}
\begin{proof}
The result follows immediately from \lemref{slopes-lipschitz-derivatives}
and \propref{slopes-separation-1}.
\end{proof}
We can now apply \thmref{absslopes-elimination} again.
\begin{prop}
\label{prop:slopes-elimination-2-logN}Fix $\t p>15$, and let $l=100\left[\left(\log N\right)/\gamma\right]$.
There exists a constant $N_{0}=N_{0}\left(\left\Vert a\right\Vert _{\infty},\left\Vert b\right\Vert _{*},c,\alpha,\gamma,E^{0},\t p\right)$
such that for any $N\ge N_{0}$ there exists a set $\Omega_{N}$,
with
\[
\mes\left(\Omega_{N}\right)<\exp\left(-\left(\log\log\log N\right)^{2}/10\right),\,\com\left(\Omega_{N}\right)<N^{2}\left(\log N\right)^{\t p},
\]
such that for any $\omega\in\Omega^{0}\cap\mb T_{c,\alpha}\setminus\Omega_{l}$
there exists a set $\mc E_{N,\omega}$, with
\[
\mes\left(\mc E_{N,\omega}\right)\lesssim\left(\log\log N\right)^{-1/10},\,\com\left(\mc E_{N,\omega}\right)\lesssim N\left(\log N\right)^{6},
\]
 such that for any $x\in\mb T$ and any integer $m$, $\left(\log N\right)^{6}\le\left|m\right|\le2N$,
we have
\[
\dist\left(\mathcal{E}^{0}\cap\spec\left(H^{\left(l\right)}\left(x,\omega\right)\setminus\mc E_{N,\omega}\right),\spec\left(H^{\left(l\right)}\left(x+m\omega,\omega\right)\right)\right)\ge\frac{3}{N\left(\log N\right)^{\t p}}.
\]
\end{prop}
\begin{proof}
We begin by identifying all the parameters used in \secref{Abstract-Elimination}.
The polynomial $P$ is given by 
\[
P\left(e\left(x\right),e\left(\omega\right),E\right)=e\left(K_{0}lx\right)e\left(K_{0}l^{2}\omega\right)f_{K_{0},l}\left(x,\omega,E\right),
\]
and we can take $d_{1}=C\left(\log N\right)^{3}$, $d_{2}=C\left(\log N\right)^{2}$,
with $C=C\left(\left\Vert a\right\Vert _{\infty},\left\Vert b\right\Vert _{*},\rho_{0},\gamma\right)$.
We have $\left\{ f_{j}\right\} =\left\{ E_{K_{0},i}^{\left(l\right)}:\, i\in\left\{ 1,\ldots,l\right\} \right\} $,
and $n=l$. By \eqref{slopes-EK-lipschitzness} we can choose $C_{0}=C$,
$C_{1}=C\log N$, with $C=C\left(\left\Vert a\right\Vert _{\infty},\left\Vert b\right\Vert _{*},\rho_{0},\gamma\right)$.
Let $\Omega_{l}$, $\mc E_{l,\omega}$ be as in \lemref{slopes-lipschitz-slopes-2}.
By \lemref{slopes-lipschitz-slopes-2}, with $p\in\left(1,2\right)$
such that $14+p<\t p$, we can choose $\mc Y^{0}=\mb T_{c,\alpha}\setminus\Omega_{l}$,
$\mc Z^{0}=\mathcal{E}^{0}$, $\mc Z_{y}^{0}=\mc E_{l,\omega}$, $c_{0}=\left(\log\log N\right)^{-1/10}$,
$C_{2}=\log N\left(\log\log N\right)^{C}$, with $C=C\left(\alpha\right)$,
$C_{3}=\left(\log N\right)^{1+p}\gamma^{-1}$, $r_{0}=1/C_{3}$. 

Next we apply \thmref{absslopes-elimination} with $\tau=\left(\log N\right)^{-4}\left(\log\log N\right)^{-1}$,
$\sigma=4N^{-1}\left(\log N\right)^{-\t p}$, $Q=\left(\log N\right)^{6}$,
$M=2N$, $\delta=\left(\log N\right)^{-\left(5+p\right)}\left(\log\log N\right)^{-2}$,
$\delta'=cN^{-2}\left(\log N\right)^{-\tilde{p}}$, with $c=c\left(\left\Vert a\right\Vert _{\infty},\left\Vert b\right\Vert _{*},\rho_{0},\gamma\right)$,
$M=2N$. Let $\Omega_{K_{0},l}=\Omega_{l}\cup\mc Y$ and $\mc E_{K_{0},l,\omega}=\t{\mc Z}_{y}$.
We have
\[
\mes\left(\Omega_{K_{0},l}\right)\le\exp\left(-\left(\log\log\log N\right)^{2}/10\right),\,\com\left(\Omega_{K_{0},l}\right)\le N^{2}\left(\log N\right)^{\t p},
\]
\[
\mes\left(\mc E_{K_{0},l,\omega}\right)\lesssim\left(\log\log N\right)^{-1/10},\,\com\left(\mc E_{K_{0},l,\omega}\right)\lesssim N\left(\log N\right)^{6},
\]
and
\[
\dist\left(\mathcal{E}^{0}\cap\spec\left(H_{K_{0}}^{\left(l\right)}\left(x,\omega\right)\setminus\mc E_{K_{0},l,\omega}\right),\spec\left(H_{K_{0}}^{\left(l\right)}\left(x+m\omega,\omega\right)\right)\right)\ge\frac{4}{N\left(\log N\right)^{\t p}},
\]
for any $\omega\in\Omega^{0}\cap\mb T_{c,\alpha}\setminus\Omega_{K_{0},l}$,
any $x\in\mb T$ and any integer $m$, $\left(\log N\right)^{6}\le\left|m\right|\le2N$.

The conclusion follows just as in the proof of \propref{slopes-elimination-1},
by setting $\Omega_{N}=\Omega_{K_{0},l}$, and $\mc E_{N,\omega}=\left\{ E\in\mathcal{E}^{0}:\,\dist\left(E,\mc E_{K_{0},l,\omega}\cup\left(\mathcal{E}^{0}\right)^{C}\right)\le N^{-2}\right\} $.
\end{proof}
Next we obtain the new version of \propref{slopes-elimination-1}. 
\begin{prop}
\label{prop:slopes-elimination-2}Fix $p>15$, $A>1$ and let $l=2\left[\left(\log N\right)^{A}\right]$.
There exists a constant $N_{0}=N_{0}\left(\left\Vert a\right\Vert _{\infty},\left\Vert b\right\Vert _{*},c,\alpha,\gamma,E^{0},p,A\right)$,
such that for any $N\ge N_{0}$ there exists a set $\Omega_{N}$,
with
\[
\mes\left(\Omega_{N}\right)\lesssim\exp\left(-\left(\log\log\log N\right)^{2}/10\right),\,\com\left(\Omega_{N}\right)\lesssim N^{2}\left(\log N\right)^{p},
\]
such that for any $\omega\in\Omega^{0}\cap\mb T_{c,\alpha}\setminus\Omega_{N}$
there exists a set $\mc E_{N,\omega}$, with
\[
\mes\left(\mc E_{N,\omega}\right)\lesssim\left(\log\log N\right)^{-1/10},\,\com\left(\mc E_{N,\omega}\right)\lesssim N\left(\log N\right)^{6},
\]
 such that for any $x\in\mb T$ and any integer $m$, $\left(\log N\right)^{6A}\le\left|m\right|\le N$,
we have
\[
\dist\left(\mathcal{E}^{0}\cap\spec\left(H^{\left(l_{1}\right)}\left(x,\omega\right)\setminus\mc E_{N,\omega}\right),\spec\left(H^{\left(l_{2}\right)}\left(x+m\omega,\omega\right)\right)\right)\ge\frac{2}{N\left(\log N\right)^{p}},
\]
$l_{1},l_{2}\in\left\{ l,l+1,2l,2l+1\right\} $.\end{prop}
\begin{proof}
Let $\Omega_{N}^{1}$, $\mc E_{N,\omega}^{1}$ denote the sets $\Omega_{N}$,
$\t{\mc E}_{N,\omega}$ from \propref{resultants-localization}. Let
$\Omega_{N}^{2}$, $\mc E_{N,\omega}^{2}$ denote the sets $\Omega_{N}$,
$\mc E_{N,\omega}$ from \propref{slopes-elimination-2-logN}, with
$\t p=p$. We define $\Omega_{N}=\Omega_{l}^{1}\cup\Omega_{l+1}^{1}\cup\Omega_{2l}^{1}\cup\Omega_{2l+1}^{1}\cup\Omega_{N}^{2}$
and 
\[
\mc E_{N,\omega}=\left\{ E\in\mathcal{E}^{0}:\,\dist\left(\mc E_{l,\omega}^{1}\cup\mc E_{l+1,\omega}^{1}\cup\mc E_{2l,\omega}^{1}\cup\mc E_{2l+1,\omega}^{1}\cup\mc E_{N,\omega}^{2}\cup\left(\mathcal{E}^{0}\right)^{C}\right)\le\frac{2}{N\left(\log N\right)^{p}}\right\} .
\]
It is straightforward to check the measure and complexity bounds for
$\Omega_{N}$ and $\mc E_{N,\omega}$.

To obtain the conclusion we argue by contradiction. Fix $x\in\mb T$
and $\omega\in\Omega^{0}\cap\mb T_{c,\alpha}\setminus\Omega_{N}$.
Suppose there exist $l_{1},l_{2}\in\left\{ l,l+1,2l,2l+1\right\} $,
$j_{1}$ , $j_{2}$, and $m$, $\left|m\right|\ge\left(\log N\right)^{6A}$
$ $such that $E_{j_{1}}^{\left(l_{1}\right)}\left(x,\omega\right)\in\mathcal{E}^{0}\setminus\mc E_{N,\omega}$
and
\begin{equation}
\left|E_{j_{1}}^{\left(l_{1}\right)}\left(x,\omega\right)-E_{j_{2}}^{\left(l_{2}\right)}\left(x+m\omega,\omega\right)\right|<\frac{2}{N\left(\log N\right)^{p}}.\label{eq:slopes-elimination2-contradiction-assumption}
\end{equation}
By the definition of $\mc E_{l,\omega}$ we have that $E_{j_{i}}^{\left(l_{i}\right)}\left(x,\omega\right)\in\mathcal{E}^{0}\setminus\mc E_{l_{i},\omega}^{1}$,
$i=1,2$. We can apply \propref{resultants-localization} to conclude
that there exist eigenvalues $E_{k_{1}}^{\left(l'\right)}\left(x+n_{1}\omega,\omega\right)$,
$n_{1}\in\left[0,l_{1}-1\right]$ and $E_{k_{2}}^{\left(l'\right)}\left(x+n_{2}\omega,\omega\right)$,
$n_{2}\in\left[m,m+l_{2}-1\right]$, $l'=100\left[\left(\log N\right)/\gamma\right]$,
such that
\begin{equation}
\left|E_{j_{1}}^{\left(l_{1}\right)}\left(x,\omega\right)-E_{k_{1}}^{\left(l'\right)}\left(x+n_{1}\omega,\omega\right)\right|\lesssim\exp\left(-\gamma l'/3\right)<1/N^{2},\label{eq:slopes-elimination2-localization1}
\end{equation}
\begin{equation}
\left|E_{j_{2}}^{\left(l_{2}\right)}\left(x+m\omega,\omega\right)-E_{k_{2}}^{\left(l'\right)}\left(x+n_{2}\omega,\omega\right)\right|\lesssim\exp\left(-\gamma l'/3\right)<1/N^{2}.\label{eq:slopes--elimination2-localization2}
\end{equation}
By the definition of $\mc E_{N,\omega}$ we have $E_{k_{1}}\left(x+n_{1}\omega\right)\in\mathcal{E}^{0}\setminus\mc E_{N,\omega}^{2}$.
We can apply \propref{slopes-elimination-2-logN}, with $\t p=p$,
to get
\[
\left|E_{k_{1}}^{\left(l'\right)}\left(x+n_{1}\omega,\omega\right)-E_{k_{2}}^{\left(l'\right)}\left(x+n_{2}\omega,\omega\right)\right|\ge\frac{3}{N\left(\log N\right)^{p}}.
\]
The above inequality, together with \eqref{slopes-elimination2-localization1},
and \eqref{slopes--elimination2-localization2} contradicts \eqref{slopes-elimination2-contradiction-assumption}.
This concludes the proof.
\end{proof}
Finally we obtain our main result.
\begin{thm}
\label{thm:slopes-separation-2}Fix $p>15$. There exists a constant
$N_{0}=N_{0}\left(\left\Vert a\right\Vert _{\infty},\left\Vert b\right\Vert _{*},c,\alpha,\gamma,E^{0},p\right)$
such that for any $N\ge N_{0}$ there exists a set $\t{\Omega}_{N}$,
with
\[
\mes\left(\t{\Omega}_{N}\right)\lesssim\exp\left(-\left(\log\log\log N\right)^{2}/20\right),\,\com\left(\t{\Omega}_{N}\right)\lesssim N^{2}\left(\log N\right)^{p},
\]
such that for any $\omega\in\Omega^{0}\cap\mb T_{c,\alpha}\setminus\t{\Omega}_{N}$
there exists a set $\t{\mc E}_{N,\omega}$, with
\[
\mes\left(\t{\mc E}_{N,\omega}\right)\lesssim\left(\log\log N\right)^{-1/10},\,\com\left(\t{\mc E}_{N,\omega}\right)\lesssim N\left(\log N\right)^{6},
\]
 such that for any $x\in\mb T$, if $E_{j}^{\left(N\right)}\left(x,\omega\right)\in\mathcal{E}^{0}\setminus\t{\mc E}_{N,\omega}$,
for some $j$, then
\[
\left|E_{j}^{\left(N\right)}\left(x,\omega\right)-E_{k}^{\left(N\right)}\left(x,\omega\right)\right|\ge\frac{1}{N\left(\log N\right)^{p}},
\]
for any $k\neq j$.\end{thm}
\begin{proof}
We start by identifying the parameters from the Elimination Assumption
\ref{localization-elimination-assumption}. Apply \propref{slopes-elimination-1}
with $A=A\left(\alpha\right)$ as in the Elimination Assumption \ref{localization-elimination-assumption}.
We can choose $\Omega_{N}$, $\mc E_{N,\omega}$ as in \propref{slopes-elimination-2}
and we also have $Q_{N}=\left(\log N\right)^{6A}$, $\sigma_{N}=2N^{-1}\left(\log N\right)^{-p}$. 

Next we apply \thmref{separation-bootstrap} with $N'=\exp\left(\left(\log N\right)^{1/7A}\right)$.
The conclusion follows by setting $\t{\Omega}_{N}=\Omega_{N}\cup\Omega_{N'}$,
and 
\[
\t{\mc E}_{N,\omega}=\left\{ E\in\mathcal{E}^{0}:\dist\left(\mc E_{N,\omega}\cup\mc E_{N',\omega}\cup\left(\mathcal{E}^{0}\right)^{C}\right)<2N^{-1}\left(\log N\right)^{-p}\right\} .
\]
 $ $
\end{proof}
\appendix

\section{Appendix}

In this section we discuss how to obtain some of the results stated
in \secref{Preliminaries} from the results of \cite{2012arXiv1202.2915B}.

We start by discussing the large deviations estimate for determinants
as stated in \propref{prelims-LDT-determinants}. For convenience
we recall three relevant results from \cite{2012arXiv1202.2915B}.
Note that in what follows the assumption $\left(\omega,E\right)\in\mb T_{c,\alpha}\times\mb C$,
$L\left(\omega,E\right)>\gamma>0$ is implicit. Also, we use the notation
$\left\langle \log\left|f_{n}^{a}\right|\right\rangle =\int_{\mb T}\log\left|f_{n}^{a}\left(x\right)\right|dx$
and $I_{a,E}=\int_{\mb T}\log\left|a\left(x\right)-E\right|dx$.
\begin{prop}
\label{prop:appendix-LDT-determinants} (\cite[Proposition 4.10]{2012arXiv1202.2915B})
There exist constants $c_{0}=c_{0}(\left\Vert a\right\Vert _{\infty},I_{a,E},$
$\left\Vert b\right\Vert _{*},\left|E\right|,\omega,\gamma)$, $C_{0}=C_{0}\left(\omega\right)>\alpha+2$,
and $C_{1}=C_{1}\left(\left\Vert a\right\Vert _{\infty},I_{a,E},\left\Vert b\right\Vert _{*},\left|E\right|,\omega,\gamma\right)$
such that for every integer $n>1$ and any $\delta>0$ we have
\[
\mes\left\{ x\in\mb T:\,\left|\log\left|f_{n}^{a}\left(x\right)\right|-\left\langle \log\left|f_{n}^{a}\right|\right\rangle \right|>n\delta\right\} \le C_{1}\exp\left(-c_{0}\delta n\left(\log n\right)^{-C_{0}}\right).
\]
\end{prop}
\begin{lem}
\label{lem:prelims-<f>-nL} (\cite[Lemma 4.11]{2012arXiv1202.2915B})
There exists a constant $C_{0}=C_{0}(\left\Vert a\right\Vert _{\infty},I_{a,E},\left\Vert b\right\Vert _{*},$
$\left|E\right|,\omega,\gamma)$ such that
\[
\left|\left\langle \log\left|f_{n}^{a}\right|\right\rangle -nL_{n}^{a}\right|\le C_{0}
\]
for all integers.
\end{lem}
\begin{lem}\label{lem:apx-Ln_L}(\cite[Lemma 3.9]{2012arXiv1202.2915B})
For any integer $n>1$ we have
\[
0\le L_{n}-L=L_{n}^{u}-L^{u}=L_{n}^{a}-L^{a}<C_{0}\frac{\left(\log n\right)^{2}}{n}
\]
where $C_{0}=C_{0}\left(\left\Vert a\right\Vert _{\infty},\left\Vert b\right\Vert _{*},\left|E\right|,\omega,\gamma\right)$.\end{lem}

\propref{prelims-LDT-determinants} is a straightforward consequence
of the above results. Note that the constants depend on $\omega$
rather than $c,\alpha$ as in \secref{Preliminaries}. However, in
\cite{2012arXiv1202.2915B} it was noted that the dependence on $\omega$
only comes through the large deviations estimate for subharmonic functions
\cite[Theorem 3.8]{MR1847592}. The dependence there is only on $c,\alpha$,
so we can replace $\omega$ with $c,\alpha$. The dependence of the
constants on $I_{a,E}$ in \cite{2012arXiv1202.2915B} came through
\cite[Lemma 4.2]{2012arXiv1202.2915B}. We provide a different proof
of this lemma that gets rid of the dependence on $I_{a,E}$. 

First we need to recall three results that will be needed for the
proof. The following theorem is a restatement of the large deviations
estimate for subharmonic functions, \cite[Theorem 3.8]{MR1847592}.
In what follows $\mc A_{\rho}$ denotes the annulus $\left\{ z:\,\left|z\right|\in\left(1-\rho,1+\rho\right)\right\} $.
\begin{thm}
\label{thm:apx-sh_ldt}(\cite[Theorem 3.8]{MR1847592}) Fix $p>\alpha+2$.
Let $u$ be a subharmonic function and let 
\[
u\left(z\right)=\int_{\mathbb{C}}\log\left|z-\zeta\right|d\mu\left(\zeta\right)+h\left(z\right)
\]
be its Riesz representation on a neighborhood of $\mc A_{\rho}$.
If $\mu\left(\mc A_{\rho}\right)+\left\Vert h\right\Vert _{L^{\infty}\left(\mc A_{\rho}\right)}\le M$
then for any $\delta>0$ and any positive integer $n$ we have
\[
\mes\left(\left\{ x\in\mb T:\left|\sum_{k=1}^{n}u\left(x+k\omega\right)-n\left\langle u\right\rangle \right|>\delta n\right\} \right)<\exp\left(-c_{0}\delta n+r_{n}\right)
\]
where $c_{0}=c_{0}\left(c,\alpha,M,\rho\right)$ and
\[
r_{n}=\begin{cases}
C_{0}\left(\log n\right)^{p} & ,\, n>1\\
C_{0} & ,\, n=1,
\end{cases}
\]
with $C_{0}=C_{0}\left(c,\alpha,p\right)$. If $p_{s}/q_{s}$ is a
convergent of $\omega$ and $n=q_{s}>1$ then one can choose $r_{n}=C_{0}\log n$.\end{thm}
\begin{prop}
\label{prop:apx-ldt-M}(\cite[Theorem 3.10]{2012arXiv1202.2915B})
Fix $p>\alpha+2$. For any $\delta>0$ and any integer $n>1$ we have
\[
\mes\left\{ x\in\mb T:\left|\log\left\Vert M_{n}^{a}\left(x\right)\right\Vert -nL_{n}^{a}\right|>\delta n\right\} <\exp\left(-c_{0}\delta n+C_{0}\left(\log n\right)^{p}\right)
\]
where $c_{0}=c_{0}\left(\left\Vert a\right\Vert _{\infty},\left\Vert b\right\Vert _{*},\left|E\right|,c,\alpha,\gamma\right)$
and $C_{0}=C_{0}\left(\left\Vert a\right\Vert _{\infty},\left\Vert b\right\Vert _{*},\left|E\right|,c,\alpha,\gamma,p\right)$.
The same estimate, with possibly different constants, holds with $L^{a}$
instead of $L_{n}^{a}$. \end{prop}
\begin{lem}
\label{lem:apx-sh_better_upper_bound}(\cite[Lemma 2.4]{MR2438997})
Let $u$ be a subharmonic function defined on $\mc A_{\rho}$ such
that $\sup_{\mc A_{\rho}}u\le M$. There exist constants $C_{1}=C_{1}\left(\rho\right)$
and $C_{2}$ such that, if for some $0<\delta<1$ and some $L$ we
have 
\[
\mes\left\{ x\in\mb T:\, u\left(x\right)<-L\right\} >\delta,
\]
then
\[
\sup_{\mb T}u\le C_{1}M-\frac{L}{C_{1}\log\left(C_{2}/\delta\right)}.
\]

\end{lem}
We can now reprove \cite[Lemma 4.2]{2012arXiv1202.2915B}. Analogously
to $f_{N}^{a}$ and $f_{N}^{u}$, $f_{N}$ will be the top left entry
in $M_{N}$. From \eqref{prelims-Ma-M} it follows that
\begin{equation}
f_{N}\left(z\right)=\left(\prod_{j=1}^{N}b\left(z+j\omega\right)\right)^{-1}f_{N}^{a}\left(z\right).\label{eq:apx-f-fa}
\end{equation}

\begin{lem}
(cf. \cite[Lemma 4.2]{2012arXiv1202.2915B}) Let $\left(\omega,E\right)\in\mb T_{c,\alpha}\times\mb C$
be such that $L\left(\omega,E\right)>\gamma>0$. There exists $l_{0}=l_{0}\left(\left\Vert a\right\Vert _{\infty},\left\Vert b\right\Vert _{*},\left|E\right|,c,\alpha,\gamma\right)$
such that
\[
\mes\left\{ x\in\mb T:\,\left|f_{l}\left(x\right)\right|\le\exp\left(-l^{3}\right)\right\} \le\exp\left(-l\right)
\]
for all $l\ge l_{0}$.\end{lem}
\begin{proof}
We argue by contradiction. Assume
\[
\mes\left\{ x\in\mb T:\,\left|f_{l}\left(x\right)\right|\le\exp\left(-l^{3}\right)\right\} >\exp\left(-l\right)
\]
for some sufficiently large $l$. We have that
\begin{align*}
\left|f_{l}^{a}\left(x\right)\right| & =\left|f_{l}\left(x\right)\right|\prod_{j=1}^{l}\left|b\left(x+j\omega\right)\right|\le\exp\left(-l^{3}\right)C^{l}\le\exp\left(-l^{3}/2\right)
\end{align*}
on a set of measure greater than $\exp\left(-l\right)$. Hence
\[
\mes\left\{ x\in\mb T:\,\left|f_{l}^{a}\left(x\right)\right|\le\exp\left(-l^{3}/2\right)\right\} >\exp\left(-l\right).
\]
At the same time we have 
\[
\sup_{\mb T}\log\left|f_{l}^{a}\right|\le\sup_{\mb T}\log\left\Vert M_{l}^{a}\right\Vert \le Cl,
\]
so by applying \lemref{apx-sh_better_upper_bound} we get
\begin{equation}
\sup_{\mb T}\left|f_{l}^{a}\right|\le\exp\left(C_{1}l-\frac{l^{3}}{C_{2}\log\left(C_{3}\exp\left(l\right)\right)}\right)\le\exp\left(-Cl^{2}\right).\label{eq:apx-fa-ub}
\end{equation}

Using \propref{apx-ldt-M} and \eqref{prelims-Ma-fa} (recall that
$\t b=\bar{b}$ on $\mb T$) we get 
\begin{multline}
\exp\left(lL^{a}-l^{1/3}\right)\le\left\Vert M_{l}^{a}\left(x\right)\right\Vert \le\Bigg(\left|f_{l}^{a}\left(x\right)\right|^{2}+\left|b\left(x+l\omega\right)f_{l-1}^{a}\left(x\right)\right|^{2}\\
+\left|b\left(x\right)f_{l-1}^{a}\left(x+\omega\right)\right|^{2}+\left|b\left(x\right)b\left(x+l\omega\right)f_{l-2}^{a}\left(x+\omega\right)\right|^{2}\Bigg)^{1/2}\label{eq:apx-Ml-lb}
\end{multline}
for all $x$ except for a set of measure less than $\exp\left(-c_{1}l^{1/3}+C\left(\log l\right)^{p}\right)<\exp\left(-cl^{1/3}\right)$.
Our plan is to contradict \eqref{apx-Ml-lb} by showing that
\begin{multline}
\left|f_{l}^{a}\left(x\right)\right|^{2}+\left|b\left(x+l\omega\right)f_{l-1}^{a}\left(x\right)\right|^{2}+\left|b\left(x\right)f_{l-1}^{a}\left(x+\omega\right)\right|^{2}\\
+\left|b\left(x\right)b\left(x+l\omega\right)f_{l-2}^{a}\left(x+\omega\right)\right|^{2}<\exp\left(2lL^{a}-2l^{1/3}\right),\label{eq:apx-contradiction}
\end{multline}
for $x$ in some set of measure much larger than $\exp\left(-cl^{1/3}\right)$.
The first term is already taken care of by \eqref{apx-fa-ub}. We
will show that we can provide a convenient upper bound for the next
two terms when $x$ is in some set of measure much larger than $\exp\left(-cl^{1/3}\right)$.
For this we argue again by contradiction. Suppose
\begin{equation}
\left|f_{l-1}^{a}\left(x\right)\right|\ge\exp\left(lL^{a}-l^{1/2}\right)\label{eq:apx-f_l-1_lb}
\end{equation}
for $x\in G$, with $\mes\left(G\right)\ge1/2-l^{-2}$. Using \corref{prelims-uniform-upper-bound}
we can apply Cartan's estimate \lemref{prelims-Cartan-estimate} (with
$H=l^{1/4}$) to $\log\left|f_{l-1}^{a}\left(\cdot\right)\right|$
on $\mc D\left(x_{0},l^{^{-1}}\right)$, for any $x_{0}\in G$, to
get
\begin{equation}
\left|f_{l-1}^{a}\left(x\right)\right|\ge\exp\left(lL^{a}-l^{5/6}\right),\label{eq:apx-lb-fl-1}
\end{equation}
for $x\in\mc D\left(x_{0},l^{-1}/6\right)\setminus\mc B_{x_{0}}$,
$\mes\left(\mc B_{x_{0}}\right)\le\exp\left(-l^{1/4}\right)$. It
is straightforward to see that \eqref{apx-lb-fl-1} holds on a set
$G'\supset G$, with$\mes\left(G'\right)\ge1/2+cl^{-1}$. Hence, we
have
\begin{equation}
\left|f_{l-1}^{a}\left(x\right)\right|,\left|f_{l-1}^{a}\left(x+\omega\right)\right|\ge\exp\left(lL^{a}-l^{5/6}\right),\label{eq:apx-fl-1x-fl-1x+o}
\end{equation}
on the set $G''=G'\cap\left(G'+\omega\right)$, with $\mes\left(G''\right)>cl^{-1}$.
Let $P_{l}\left(x,\omega\right)=\prod_{j=0}^{l-1}b\left(x+j\omega\right)$.
We will obtain a contradiction by using the identity 
\begin{multline*}
\overline{P_{l}\left(x,\omega\right)}P_{l}\left(x+\omega,\omega\right)=\det M_{l}^{a}\left(x\right)=-\overline{b\left(x\right)}b\left(x+l\omega\right)f_{l}^{a}\left(x\right)f_{l-2}^{a}\left(x+\omega\right)\\
+\overline{b\left(x\right)}b\left(x+l\omega\right)f_{l-1}^{a}\left(x\right)f_{l-1}^{a}\left(x+\omega\right).
\end{multline*}
Indeed, from the above identity it follows that
\[
\overline{P_{l-1}\left(x+\omega,\omega\right)}P_{l-1}\left(x+\omega,\omega\right)=-f_{l}^{a}\left(x\right)f_{l-2}^{a}\left(x+\omega\right)+f_{l-1}^{a}\left(x\right)f_{l-1}^{a}\left(x+\omega\right),
\]
and hence
\begin{equation}
\frac{\left|f_{l}^{a}\left(x\right)f_{l-2}^{a}\left(x+\omega\right)\right|}{\left|P_{l-1}\left(x+\omega,\omega\right)\right|^{2}}\ge\frac{\left|f_{l-1}^{a}\left(x\right)f_{l-1}^{a}\left(x+\omega\right)\right|}{\left|P_{l-1}\left(x+\omega,\omega\right)\right|^{2}}-1.\label{eq:apx-det-contradiction}
\end{equation}
From \thmref{apx-sh_ldt} it follows that
\[
\exp\left(lD-l^{1/2}\right)\le\left|P_{l-1}\left(x+\omega,\omega\right)\right|\le\exp\left(lD+l^{1/2}\right),
\]
for $x\in\mb T\setminus\mc B$, with $\mes\left(\mc B\right)<\exp\left(-l^{1/3}\right)$.
On one hand we have
\[
\frac{\left|f_{l}^{a}\left(x\right)f_{l-2}^{a}\left(x+\omega\right)\right|}{\left|P_{l-1}\left(x+\omega,\omega\right)\right|^{2}}\le\exp\left(-Cl^{2}+lL^{a}+\left(\log l\right)^{C}-2lD+2l^{1/2}\right)\le\exp\left(-cl^{2}\right),
\]
for $x\in\mb T\setminus\mc B$. On the other hand, for $x\in G''\setminus\mc B$
we have
\begin{multline*}
\frac{\left|f_{l-1}^{a}\left(x\right)f_{l-1}^{a}\left(x+\omega\right)\right|}{\left|P_{l-1}\left(x,\omega\right)\right|^{2}}\ge\exp\left(2lL^{a}-2l^{5/6}-2lD-2l^{1/2}\right)\\
=\exp\left(2lL-2l^{5/6}-2l^{1/2}\right)\ge\exp\left(2\gamma l-4l^{5/6}\right).
\end{multline*}
Since $G''\setminus\mc B\neq\emptyset$, the previous two inequalities
contradict \eqref{apx-det-contradiction}. Hence we must have $\mes\left(G\right)<1/2-l^{-2}$.
In other words we have 
\[
\left|f_{l-1}^{a}\left(x\right)\right|<\exp\left(lL^{a}-l^{1/2}\right),
\]
for $x\in B=\mb T\setminus G$, $\mes\left(B\right)\ge1/2+l^{-2}$.
It follows that
\begin{equation}
\left|f_{l-1}^{a}\left(x\right)\right|,\left|f_{l-1}^{a}\left(x+\omega\right)\right|<\exp\left(lL^{a}-l^{1/2}\right),\label{eq:apx-fl-1-ub}
\end{equation}
for $x\in B'=B\cap\left(B+\omega\right)$, $\mes\left(B'\right)>l^{-2}$.

By writing
\[
M_{l}^{a}\left(x-\omega\right)=M_{l}^{a}\left(x\right)\left[\begin{array}{cc}
a\left(x-\omega\right)-E & -\overline{b\left(x-\omega\right)}\\
b\left(x\right) & 0
\end{array}\right],
\]
we get
\[
f_{l}^{a}\left(x-\omega\right)=\left(a\left(x-\omega\right)-E\right)f_{l-1}^{a}\left(x\right)-\left|b\left(x\right)\right|^{2}f_{l-2}^{a}\left(x+\omega\right).
\]
From this we get
\begin{multline}
\left|b\left(x\right)b\left(x+l\omega\right)f_{l-2}^{a}\left(x+\omega\right)\right|=\frac{\left|b\left(x+l\omega\right)\right|}{\left|b\left(x\right)\right|}\left|f_{l}^{a}\left(x-\omega\right)-\left(a\left(x-\omega\right)-E\right)f_{l-1}^{a}\left(x\right)\right|\\
\le C\exp\left(-D+l^{1/3}\right)\left(\exp\left(-Cl^{2}\right)+C\exp\left(lL^{a}-l^{1/2}\right)\right)\le C\exp\left(lL^{a}-l^{1/2}/2\right),\label{eq:apx-fl-2-ub}
\end{multline}
for $x\in B''\subset B'$, $\mes\left(B''\right)>l^{-2}-\exp\left(-l^{1/4}\right)>l^{-3}$
(note that we used \thmref{apx-sh_ldt}). By \eqref{apx-fa-ub}, \eqref{apx-fl-1-ub},
and \eqref{apx-fl-2-ub} we have that \eqref{apx-contradiction} holds
for $x\in B''$. Since $\mes\left(B''\right)\gg\exp\left(-cl^{1/3}\right)$,
this contradicts \eqref{apx-Ml-lb}, as desired, and concludes the
proof.
\end{proof}
Next we discuss the uniform upper bound result, \propref{prelims-M^a-upper-bound},
and its consequences, \corref{prelims-uniform-upper-bound}, \corref{prelims-lipschitzness},
as well as \lemref{prelims-uniform-bound-S_N}. For convenience we
state the uniform upper bound result from \cite{2012arXiv1202.2915B}.
\begin{prop}
(\cite[Proposition 3.14]{2012arXiv1202.2915B}) Fix $p>\alpha+2$.
For any integer $n>1$ we have that
\[
\sup_{x\in\mb T}\log\left\Vert M_{n}^{a}\left(x\right)\right\Vert \le nL_{n}^{a}+C_{0}\left(\log n\right)^{p}
\]
where $C_{0}=C_{0}\left(\left\Vert a\right\Vert _{\infty},\left\Vert b\right\Vert _{*},\left|E\right|,c,\alpha,\gamma,p\right)$.
\end{prop}
\propref{prelims-M^a-upper-bound} follows immediately from the above
proposition and \lemref{apx-Ln_L}.

Next we recall some further results needed to prove \corref{prelims-uniform-upper-bound}.
The statement of the results is adapted to our setting.
\begin{lem}
\label{lem:apx-Ly-Ly'}(\cite[Corollary 3.13]{2012arXiv1202.2915B})
There exists a constant $C_{0}=C_{0}(\left\Vert a\right\Vert _{\infty},\left\Vert b\right\Vert _{*},\left|E\right|,$
$\rho_{0})$ such that
\[
\left|L_{n}^{u}\left(y_{1}\right)-L_{n}^{u}\left(y_{2}\right)\right|=\left|L_{n}^{a}\left(y_{1}\right)-L_{n}^{a}\left(y_{2}\right)\right|\le C_{0}\left|r_{1}-r_{2}\right|
\]
for any $y_{1},y_{2}\in\left(1-\rho_{0},1+\rho_{0}\right)$ and any
positive integer $n$.
\end{lem}
\begin{lem}\label{lem:apx-LE-LE'}(\cite[Corollary 3.17]{2012arXiv1202.2915B})
Let $\left(\omega,E_{0}\right)\in\mb T_{c,\alpha}\times\mb C$ such
that $L\left(\omega,E_{0}\right)>\gamma>0$. There exist constants
$C_{0}=C_{0}\left(\left\Vert a\right\Vert _{\infty},\left\Vert b\right\Vert _{*},\left|E_{0}\right|,c,\alpha,\gamma\right)$,
$C_{1}=C_{1}(\left\Vert a\right\Vert _{\infty},\left\Vert b\right\Vert _{*}\left|E_{0}\right|,c,\alpha,$
$\gamma)$, and $n_{0}=n_{0}\left(\left\Vert a\right\Vert _{\infty},\left\Vert b\right\Vert _{*},\left|E_{0}\right|,c,\alpha,\gamma\right)$
such that we have
\[
\left|n\left(L_{n}\left(\omega,E\right)-L_{n}\left(\omega,E_{0}\right)\right)\right|=\left|n\left(L_{n}^{a}\left(\omega,E\right)-L_{n}^{a}\left(\omega,E_{0}\right)\right)\right|\le n^{-C_{0}}
\]
 for $n\ge n_{0}$ and $\left|E-E_{0}\right|<n^{-C_{1}}$.\end{lem}

A straightforward replacement of $E$ with $\omega$ in the proof
of \cite[Corollary 3.17]{2012arXiv1202.2915B} yields the following
result. 
\begin{lem}
\label{lem:apx-Lo-Lo'}Let $\left(\omega_{0},E\right)\in\mb T_{c,\alpha}\times\mb C$
such that $L\left(\omega_{0},E\right)>\gamma>0$. There exist constants
$C_{0}=C_{0}\left(\left\Vert a\right\Vert _{\infty},\left\Vert b\right\Vert _{*},\left|E\right|,c,\alpha,\gamma\right)$,
$C_{1}=C_{1}\left(\left\Vert a\right\Vert _{\infty},\left\Vert b\right\Vert _{*},\left|E\right|,c,\alpha,\gamma\right)$,
$n_{0}=n_{0}(\left\Vert a\right\Vert _{\infty},\left\Vert b\right\Vert _{*},$
$\left|E\right|,c,\alpha,\gamma)$ such that we have
\[
\left|n\left(L_{n}\left(\omega,E\right)-L_{n}\left(\omega_{0},E\right)\right)\right|=\left|n\left(L_{n}^{a}\left(\omega,E\right)-L_{n}^{a}\left(\omega_{0},E\right)\right)\right|\le n^{-C_{0}}
\]
 for $n\ge n_{0}$ and $\left|\omega-\omega_{0}\right|<n^{-C_{1}}$.
\end{lem}
We have all we need to prove \corref{prelims-uniform-upper-bound}. 
\begin{proof}
(of \corref{prelims-uniform-upper-bound}) From \lemref{apx-Ln_L},
\lemref{apx-Ly-Ly'}, \lemref{apx-LE-LE'}, \lemref{apx-Lo-Lo'} it
is straightforward to conclude that there exists $\rho\ll\rho_{0}$
such that $L\left(y,\omega,E\right)>\gamma/2$ for $\left|y\right|,\left|\omega-\omega_{0}\right|,\left|E-E_{0}\right|\le\rho$.
We can apply \propref{prelims-M^a-upper-bound} to get 
\[
\sup_{x\in\mb T}\log\left\Vert M_{N}^{a}\left(x+iy,\omega,E\right)\right\Vert \le NL^{a}\left(y,\omega,E\right)+C\left(\log N\right)^{C'},
\]
for $\left|y\right|,\left|\omega-\omega_{0}\right|,\left|E-E_{0}\right|\le\rho$.
The conclusion follows from \lemref{apx-Ly-Ly'}, \lemref{apx-LE-LE'},
and \lemref{apx-Lo-Lo'}.
\end{proof}
Based on the already established results it is straightforward to
see that \corref{prelims-lipschitzness} follows with the same proof
as \cite[Corollary 2.15]{MR2753606}. For the sake of completness
we include the proof.
\begin{proof}
(of \corref{prelims-lipschitzness}) Let $\partial$ denote any of
the partial derivatives $\partial_{x}$, $\partial_{y}$, $\partial_{\omega}$,
$\partial_{E}$. We have
\begin{multline*}
\partial M_{N}^{a}\left(x+iy,\omega,E\right)=\sum_{j=0}^{N-1}M_{N-j-1}^{a}\left(x+iy+\left(j+1\right)\omega,\omega,E\right)\\
\cdot\partial\left[\begin{array}{cc}
a\left(x+iy+j\omega\right)-E & -\t b\left(x+iy+j\omega\right)\\
-b\left(x+iy+\left(j+1\right)\omega\right) & 0
\end{array}\right]M_{j}^{a}\left(x+iy,\omega,E\right).
\end{multline*}
The estimate \eqref{prelims-M-lipschitzness} follows from \corref{prelims-uniform-upper-bound}
by using the above identity and the mean value theorem. 

From \eqref{prelims-M-lipschitzness} it follows that 
\begin{multline*}
\left|f_{N}^{a}\left(x+iy,\omega,E\right)-f_{N}^{a}\left(x_{0},\omega_{0},E_{0}\right)\right|\\
\le\left(\left|E-E_{0}\right|+\left|\omega-\omega_{0}\right|+\left|x-x_{0}\right|+\left|y\right|\right)\exp\left(NL^{a}\left(\omega_{0},E_{0}\right)+\left(\log N\right)^{C}\right).
\end{multline*}
The estimate \eqref{prelims-f-lipschitzness} follows by dividing
both sides by $\left|f_{N}^{a}\left(x_{0},\omega_{0},E_{0}\right)\right|$
and by using the fact that $\left|\log x\right|\lesssim\left|x-1\right|$
for $x\in\left(1/2,3/2\right)$.
\end{proof}
Finally, \lemref{prelims-uniform-bound-S_N} can be proved along the
same lines as \cite[Proposition 3.14]{2012arXiv1202.2915B} and \corref{prelims-uniform-upper-bound}.
We also need to recall the following result. 
\begin{lem}
(\cite[Lemma 4.1]{MR2438997})\label{lem:apx-uy-uy'} Let $u$ be
a subharmonic function and let 
\[
u\left(z\right)=\int_{\mathbb{C}}\log\left|z-\zeta\right|d\mu\left(\zeta\right)+h\left(z\right)
\]
be its Riesz representation on a neighborhood of $\mc A_{\rho}$.
If $\mu\left(\mc A_{\rho}\right)+\left\Vert h\right\Vert _{L^{\infty}\left(\mc A_{\rho}\right)}\le M$
then for any $r_{1},r_{2}\in\left(1-\rho,1+\rho\right)$ we have
\[
\left|\left\langle u\left(r_{1}\left(\cdot\right)\right)\right\rangle -\left\langle u\left(r_{2}\left(\cdot\right)\right)\right\rangle \right|\le C_{0}\left|r_{1}-r_{2}\right|,
\]
where $C_{0}=C_{0}\left(M,\rho\right)$.\end{lem}
\begin{proof}
(of \lemref{prelims-uniform-bound-S_N}) It is enough to prove the
estimate for $S_{N}$. Note that $\left\Vert b\right\Vert _{*}=\Vert\t b\Vert_{*}$.

We first prove the uniform upper bound with $y=0$. It is sufficient
to establish the estimate for large $N$. Fix $p>\alpha+2$. From
\thmref{apx-sh_ldt} we have
\begin{equation}
S_{N}\left(x+iy,\omega\right)-ND\left(y\right)\le C\left(\log n\right)^{p}\label{eq:apx-S-D}
\end{equation}
except for a set $\mc B\left(y\right)$ of measure less than 
\[
\exp\left(-c_{1}C\left(\log n\right)^{p}+C'\left(\log n\right)^{p}\right)<\exp\left(-c\left(\log n\right)^{p}\right).
\]
By the subharmonicity of $S_{N}$ we have
\begin{multline}
S_{N}\left(x_{0},\omega\right)-ND\le\frac{1}{\pi N^{-2}}\int_{\mc D\left(x_{0},N^{-1}\right)}\left(S_{N}\left(z,\omega\right)-ND\right)dA\left(z\right)\\
\le\frac{1}{\pi N^{-2}}\int_{-N^{-1}}^{N^{-1}}\int_{x_{0}-N^{-1}}^{x_{0}+N^{-1}}\left|S_{N}\left(x+iy,\omega\right)-ND\right|dxdy.\label{eq:apx-S-D-upper_bound}
\end{multline}
For $y\in\left(-N^{-1},N^{-1}\right)$, by using \eqref{apx-S-D}
and \lemref{apx-Ly-Ly'}, we have
\begin{multline*}
\int_{x_{0}-N^{-1}}^{x_{0}+N^{-1}}\left|S_{N}\left(x+iy,\omega\right)-ND\right|dx\\
\le\int_{x_{0}-N^{-1}}^{x_{0}+N^{-1}}\left|S_{N}\left(x+iy,\omega\right)-ND\left(y\right)\right|dx+2\left|D-D\left(y\right)\right|\\
\le C\left(\log N\right)^{p}N^{-1}+CN\exp\left(-c\left(\log N\right)^{p}/2\right)+CN^{-1}\le C\left(\log N\right)^{p}N^{-1}.
\end{multline*}
We used the fact that $\left\Vert S_{N}\left(\cdot,\omega\right)-ND\right\Vert _{L^{2}\left(\mb T\right)}\le CN$
(a straightforward consequence of \thmref{apx-sh_ldt}) to deal with
the exceptional set $\mc B\left(y\right)$. Plugging this estimate
in \eqref{apx-S-D-upper_bound} yields that
\[
\sup_{x\in\mb T}S_{N}\left(x,\omega\right)\le ND+C\left(\log N\right)^{p},
\]
with $C=C\left(\left\Vert b\right\Vert _{*},c,\alpha\right)$. This
yields (by replacing $b$ with $b\left(\cdot+iy\right)$) that
\[
\sup_{x\in\mb T}S_{N}\left(x+iy,\omega\right)\le ND\left(y\right)+C\left(\log N\right)^{p}.
\]
The conclusion follows from \lemref{apx-uy-uy'}.
\end{proof}
\bigskip \noindent\textit{Acknowledgments.} We are indebted to Michael Goldstein for suggesting the problem and the numerous discussions which were instrumental for the completion of the project. The first author was partially supported by the NSERC Discovery grant  5810-2009-298433.

\currentpdfbookmark{References}{References} \bibliographystyle{alpha}
\bibliography{schrodinger}

\end{document}